\documentclass[10pt,oneside]{amsart}

\usepackage[
paper=a4paper,
headsep=20pt,text={137mm,208mm},centering,includehead
]{geometry}

\usepackage[bookmarks]{hyperref}
\usepackage{amssymb,amsxtra,dsfont}
\usepackage[all]{xy}
\usepackage{pb-diagram,pb-xy}
\usepackage{graphicx,color}
\usepackage{epstopdf}
\usepackage{pinlabel}
\usepackage{float}
\usepackage{array,calc,booktabs}
\usepackage[
  textwidth=1.2in,
  backgroundcolor=yellow,
  bordercolor=orange,
  textsize=small
]{todonotes}

\iftrue
\makeatletter
\def\@settitle{%
  \vspace*{-0pt}
  \begin{flushleft}%
    \LARGE\bfseries
    \strut\@title\strut
  \end{flushleft}%
}
\def\@setauthors{%
  \begingroup
  \def\thanks{\protect\thanks@warning}%
  \trivlist
  \raggedright
  \large \@topsep27\p@\relax
  \advance\@topsep by -\baselineskip
  \item\relax
  \author@andify\authors
  \def\\{\protect\linebreak}%
  \authors
  \ifx\@empty\contribs
  \else
    ,\penalty-3 \space \@setcontribs
    \@closetoccontribs
  \fi
  \normalfont
  \endtrivlist
  \endgroup
}
\def\@setaddresses{\par
  \nobreak \begingroup
  \small\raggedright
  \def\author##1{\nobreak\addvspace\smallskipamount}%
  \def\\{\unskip, \ignorespaces}%
  \interlinepenalty\@M
  \def\address##1##2{\begingroup
    \par\addvspace\bigskipamount\noindent
    \@ifnotempty{##1}{(\ignorespaces##1\unskip) }%
    {\ignorespaces##2}\par\endgroup}%
  \def\curraddr##1##2{\begingroup
    \@ifnotempty{##2}{\nobreak\noindent\curraddrname
      \@ifnotempty{##1}{, \ignorespaces##1\unskip}\/:\space
      ##2\par}\endgroup}%
  \def\email##1##2{\begingroup
    \@ifnotempty{##2}{\nobreak\noindent E-mail address%
      \@ifnotempty{##1}{, \ignorespaces##1\unskip}\/:\space
      \ttfamily##2\par}\endgroup}%
  \def\urladdr##1##2{\begingroup
    \def~{\char`\~}%
    \@ifnotempty{##2}{\nobreak\noindent\urladdrname
      \@ifnotempty{##1}{, \ignorespaces##1\unskip}\/:\space
      \ttfamily##2\par}\endgroup}%
  \addresses
  \endgroup
  \global\let\addresses=\@empty
}
\def\@setabstracta{%
    \ifvoid\abstractbox
  \else
    \skip@17pt \advance\skip@-\lastskip
    \advance\skip@-\baselineskip \vskip\skip@
    \box\abstractbox
    \prevdepth\z@ % because \abstractbox is a vtop
    \vskip-28pt
  \fi
}
\renewenvironment{abstract}{%
  \ifx\maketitle\relax
    \ClassWarning{\@classname}{Abstract should precede
      \protect\maketitle\space in AMS document classes; reported}%
  \fi
  \global\setbox\abstractbox=\vtop \bgroup
    \normalfont\small
    \list{}{\labelwidth\z@
      \leftmargin0pc \rightmargin\leftmargin
      \listparindent\normalparindent \itemindent\z@
      \parsep\z@ \@plus\p@
      
    }%
    \item[\hskip\labelsep\bfseries\abstractname.]%
}{%
  \endlist\egroup
  \ifx\@setabstract\relax \@setabstracta \fi
}

\def\ps@headings{\ps@empty
  \def\@evenhead{%
    \setTrue{runhead}%
    \normalfont\scriptsize
    \rlap{\thepage}\hfill
    \def\thanks{\protect\thanks@warning}%
    \leftmark{}{}}%
  \def\@oddhead{%
    \setTrue{runhead}%
    \normalfont\scriptsize
    \def\thanks{\protect\thanks@warning}%
    \rightmark{}{}\hfill \llap{\thepage}}%
  \let\@mkboth\markboth
}\ps@headings

\def\section{\@startsection{section}{1}%
  \z@{-1.4\linespacing\@plus-.5\linespacing}{.8\linespacing}%
  {\normalfont\bfseries\Large}}
\def\subsection{\@startsection{subsection}{2}%
  \z@{-.8\linespacing\@plus-.3\linespacing}{.5\linespacing\@plus.2\linespacing}%
  {\normalfont\bfseries\large}}
\def\subsubsection{\@startsection{subsubsection}{3}%
  \z@{.7\linespacing\@plus.2\linespacing}{-1.5ex}%
  {\normalfont\bfseries}}
\def\@secnumfont{\bfseries}

\renewcommand\contentsnamefont{\bfseries}
\def\@starttoc#1#2{\begingroup
  \setTrue{#1}%
  \par\removelastskip\vskip\z@skip
  \@startsection{}\@M\z@{\linespacing\@plus\linespacing}%
    {.5\linespacing}{%\centering
      \contentsnamefont}{#2}%
  \ifx\contentsname#2%
  \else \addcontentsline{toc}{section}{#2}\fi
  \makeatletter
  \@input{\jobname.#1}%
  \if@filesw
    \@xp\newwrite\csname tf@#1\endcsname
    \immediate\@xp\openout\csname tf@#1\endcsname \jobname.#1\relax
  \fi
  \global\@nobreakfalse \endgroup
  \addvspace{32\p@\@plus14\p@}%
  \let\tableofcontents\relax
}
\def\contentsname{Contents}
\def\l@section{\@tocline{2}{.5ex}{0mm}{5pc}{}}
\def\l@subsection{\@tocline{2}{0pt}{2em}{5pc}{}}
\makeatother
\fi %\iftrue/false for amsart.cls hack

\def\to{\mathchoice{\longrightarrow}{\rightarrow}{\rightarrow}{\rightarrow}}
\makeatletter
\newcommand{\shortxra}[2][]{\ext@arrow 0359\rightarrowfill@{#1}{#2}}
\def\longrightarrowfill@{\arrowfill@\relbar\relbar\longrightarrow}
\newcommand{\longxra}[2][]{\ext@arrow 0359\longrightarrowfill@{#1}{#2}}
\renewcommand{\xrightarrow}[2][]{\mathchoice{\longxra[#1]{#2}}%
  {\shortxra[#1]{#2}}{\shortxra[#1]{#2}}{\shortxra[#1]{#2}}}

\def\addtagsub#1{\let\oldtf=\tagform@\def\tagform@##1{\oldtf{##1}\hbox{$_{#1}$}}}

\makeatother

\makeatletter
\def\Nopagebreak{\@nobreaktrue\nopagebreak}

\newtheoremstyle{theorem-giventitle}
        {}{}              %%% space between body and thm
        {\itshape}                      %%% Thm body font
        {}                              %%% Indent amount (empty = no indent)
        {\bfseries}                     %%% Thm head font
        {.}                             %%% Punctuation after thm head
        {\thm@headsep}                             %%% Space after thm head
        {\thmnote{\bfseries#3}}%%% Thm head spec

\newtheoremstyle{definition-giventitle}
        {}{}              %%% space between body and thm
        {}                      %%% Thm body font
        {}                              %%% Indent amount (empty = no indent)
        {\bfseries}                     %%% Thm head font
        {.}                             %%% Punctuation after thm head
        {\thm@headsep}                             %%% Space after thm head
        {\thmnote{\bfseries#3}}%%% Thm head spec

\makeatother

\newtheorem{theorem}{Theorem}[section]
\newtheorem{proposition}[theorem]{Proposition}
\newtheorem{corollary}[theorem]{Corollary}
\newtheorem{lemma}[theorem]{Lemma}
\newtheorem{conjecture}[theorem]{Conjecture}

\theoremstyle{definition}
\newtheorem{definition}[theorem]{Definition}
\newtheorem{question}[theorem]{Question}
\newtheorem{example}[theorem]{Example}
\newtheorem{remark}[theorem]{Remark}

\theoremstyle{theorem-giventitle}
\newtheorem{theorem-named}{}

\theoremstyle{definition-giventitle}
\newtheorem{definition-named}{}
\newtheorem{step-named}{}

\numberwithin{equation}{section}

\def\Z{\mathbb{Z}}
\def\Q{\mathbb{Q}}
\def\R{\mathbb{R}}
\def\C{\mathbb{C}}

\def\N{\mathcal{N}}

\def\cA{\mathcal{A}}

\def\bA{\mathbb{A}}

\def\Im{\operatorname{Im}}

\def\sign{\operatorname{sign}}
\def\rank{\operatorname{rank}}

\def\inte{\operatorname{int}}
\def\id{\mathrm{id}}

\def\lk{\operatorname{lk}}

\def\ldim{\dim^{(2)}}
\def\lsign{\sign^{(2)}}

\def\rhot{\rho^{(2)}}

\def\Mod{\operatorname{Mod}}
\def\Mor{\operatorname{Mor}}

\def\SO{\mathrm{SO}}
\def\STOP{\mathrm{STOP}}

\def\setminus{\smallsetminus}

\def\nstrut{{\vphantom{1}}}

\def\Chp{{\mathord{\mathbf{Ch}_{+}}}}

\def\Chpb{{\mathord{\mathbf{Ch}_{+}^{b}}}}
\def\sSet{{\mathord{\mathbf{sSet}}}}
\def\Gp{{\mathord{\mathbf{Gp}}}}

\def\Bsc{B^{\textup{sc}}}
\def\BHL{B^{\textup{HL}}}
\def\Bsurg{B^{\textup{surg}}}

\def\dBDH{\delta_{\textup{BDH}}}
\def\dEZ{\delta_{\textup{EZ}}}
\def\dconj{\delta_{\textup{conj}}}

\begin{document}

\vspace*{0mm}

\title%
[Cheeger-Gromov universal bounds for von Neumann rho-invariants]%
{A topological approach to Cheeger-Gromov universal bounds for von
  Neumann rho-invariants}

\author{Jae Choon Cha}
\address{
  Department of Mathematics\\
  POSTECH\\
  Pohang 790--784\\
  Republic of Korea
  -- and --\linebreak
  School of Mathematics\\
  Korea Institute for Advanced Study \\
  Seoul 130--722\\
  Republic of Korea
}
\email{jccha@postech.ac.kr}

\def\subjclassname{\textup{2010} Mathematics Subject Classification}
\expandafter\let\csname subjclassname@1991\endcsname=\subjclassname
\expandafter\let\csname subjclassname@2000\endcsname=\subjclassname
\subjclass{%
%  57M25, % Knots and links in $S^3$
%  57M27, % Invariants of knots and 3-manifolds
%  57N70%, % Cobordism and concordance (in low dimension)
%  57Q60; % Cobordism and concordance (in high dimension)
%  57M07, % Topological methods in group theory
}

% \keywords{}

\begin{abstract}
  Using deep analytic methods, Cheeger and Gromov showed that for any
  smooth $(4k-1)$-manifold there is a universal bound for the von
  Neumann $L^2$ $\rho$-invariants associated to arbitrary regular
  covers.  We present a proof of the existence of a universal bound
  for topological $(4k-1)$-manifolds, using $L^2$-signatures of
  bounding $4k$-manifolds.  For $3$-manifolds, we give explicit linear
  universal bounds in terms of triangulations, Heegaard splittings,
  and surgery descriptions respectively.  We show that our explicit
  bounds are asymptotically optimal.  As an application, we give new
  lower bounds of the complexity of $3$-manifolds which can be
  arbitrarily larger than previously known lower bounds.  As
  ingredients of the proofs which seem interesting on their own, we
  develop a geometric construction of efficient $4$-dimensional
  bordisms of $3$-manifolds over a group, and develop an algebraic
  topological notion of uniformly controlled chain homotopies.
\end{abstract}

% \vbox to0mm{\vss\vbox to10mm{\hbox{\large\textsl{Not for general
%         distribution}}}}\hrule height 0mm%

\maketitle

\setcounter{tocdepth}{2}
\tableofcontents

\section{Introduction and main results}
\label{section:introduction}

In~\cite{Cheeger-Gromov:1985-1}, Cheeger and Gromov studied the $L^2$
$\rho$-invariant $\rhot(M,\phi)\in \R$, which they defined for a
closed $(4k-1)$-dimensional smooth manifold $M$ and a homomorphism
$\phi\colon \pi_1(M)\to G$ to a group~$G$.  Briefly speaking, for a
Riemannian metric on $M$, $\rhot(M,\phi)$ is the difference of the
$\eta$-invariant of the signature operator of $M$ and the $L^2$
$\eta$-invariant of that of the $G$-cover of $M$ which is defined
using the von Neumann trace.
% It can be seen that $\rho(M,\phi)$ is defined for any $\phi$ even
% when $G$ is uncountable, since $\rhot(M,\phi)$ is left unchanged
% when $G$ is replaced by the image of~$\phi$.
As a key ingredient of their study of topological invariance, Cheeger
and Gromov showed that there is a universal bound of the $L^2$
$\eta$-invariants of arbitrary coverings of~$M$, by using deep
\emph{analytic} methods.  Equivalently, there is a universal bound on
the Cheeger-Gromov $\rho$-invariants of~$M$:
 
\begin{theorem}[Cheeger-Gromov~\cite{Cheeger-Gromov:1985-1}]
  \label{theorem:cheeger-gromov-bound}
  For any closed smooth $(4k-1)$-manifold $M$, there is a constant
  $C_M$ such that $|\rhot(M,\phi)| \le C_M$ for any homomorphism
  $\phi\colon \pi_1(M) \to G$ to any group~$G$.
\end{theorem}

In this paper we develop a \emph{topological} approach to the
Cheeger-Gromov universal bound~$C_M$.  Our method presents a
topological proof of the existence, and gives new topological
understanding of the universal bound with applications to low
dimensional topology.  In particular, we reveal a relationship of the
Cheeger-Gromov $\rho$-invariant and the complexity theory of
$3$-manifolds.

In this section, we discuss some backgrounds and motivations, state
our main results and applications, and introduce some ingredients of
the proofs developed in this paper, which seem interesting on their
own.  In particular, we introduce an algebraic topological notion of
controlled chain homotopy in
Section~\ref{subsection:intro-controlled-chain-homotopy}.

As a convention, we assume that manifolds are compact and oriented
unless stated otherwise.

\subsection{Background and motivation}

A known approach to $\rho$-invariants is to use a standard index
theoretic fact that if a $(4k-1)$-manifold $M$ is the boundary of a
$4k$-manifold $W$ to which the given representation of $\pi_1(M)$
extends, then the $\rho$-invariant of $M$ may be computed as a
signature defect of~$W$.  For the von Neumann $L^2$ case, as first
appeared in the work of Chang and
Weinberger~\cite{Chang-Weinberger:2003-1}, we can recast this index
theoretic \emph{computation} to provide a topological
\emph{definition}: for any $M$ and~$\phi$, $\rhot(M,\phi)$ can be
defined as a topological $L^2$-signature defect of a certain bounding
manifold, in the topological category as well as the smooth category.
This is done using a theorem of Kan and Thurston that an arbitrary
group embeds into an acyclic group~\cite{Kan-Thurston:1976-1} and
using the invariance of the von Neumann trace under composition with a
monomorphism.  Also, instead of Hilbert modules and
$L^2$-(co)homology, we can use standard homology over the group von
Neumann algebra, by employing the $L^2$-dimension theory of
L\"uck~\cite{Lueck:1998-1,Lueck:2002-1}.  For the reader's
convenience, we provide precise definitions and detailed arguments in
Section~\ref{subsection:topological-definition-rho} for topological
$(4k-1)$-manifolds.

Although the Cheeger-Gromov $\rho$-invariant can be defined
topologically, known proofs of the existence of a universal bound are
entirely analytic~\cite{Cheeger-Gromov:1985-1,Ramachandran:1993-1},
and provide hardly any information on the topology of~$M$.  From this
a natural question arises:

\begin{question}
  Can we understand the Cheeger-Gromov bound topologically?
\end{question}

This question is intriguing on its own, along the long tradition of
the interplay between geometry and topology.  Attempts to understand
the Cheeger-Gromov bound using $L^2$-signature defects have failed
(for instance see \cite[p.~348]{Cochran-Teichner:2003-1}).  The key
reason is that the bounding $4k$-manifold used to define
$\rhot(M,\phi)$ in known arguments depends on the choice of~$\phi$.

Topological understanding of the Cheeger-Gromov bound is also of
importance for applications, particularly to knots, links, and low
dimensional manifolds.  Since the work of Cochran, Orr, and Teichner
on knot concordance~\cite{Cochran-Orr-Teichner:1999-1}, several
recently discovered rich structures on topological concordance of
knots and links, topological homology cobordism of $3$-manifolds, and
symmetric Whitney towers and gropes in $4$-manifolds have been
understood by using the Cheeger-Gromov invariant.  The most general
known obstructions from the Cheeger-Gromov invariant in this context
are given as the amenable signature theorems in \cite[Theorems 1.1
and~7.1]{Cha-Orr:2009-1} and \cite[Theorem~3.2]{Cha:2012-1}.  In many
applications, it is essential to control $\rhot(M,\phi)$ for certain
homomorphisms~$\phi$.  In~\cite{Cochran-Teichner:2003-1}, Cochran and
Teichner first introduced the influential idea that the Cheeger-Gromov
bound is useful for this purpose.  Since then, the Cheeger-Gromov
bound has been used as a key ingredient in various interesting works
(some of them are discussed in
Remark~\ref{remark:applications-to-various-explicit-examples}).
 % (e.g., see~\cite{Cha-Friedl-Powell:2012-1, Cha:2006-1,
 %  Cha:2010-1, Cha:2012-1, Cochran-Harvey-Leidy:2009-1,
 %  Cochran-Harvey-Leidy:2008-1, Cochran-Harvey-Leidy:2009-02,
 %  Cochran-Harvey-Leidy:2009-3, Cha-Orr:2011-1, Cha-Powell:2013-1,
 %  Franklin:2013-1, Kim:2004-1, Kim:2005-1, Kim:2006-1}).
It is known that many existence theorems in these works could be
improved to give explicit examples if we had a better understanding of
the Cheeger-Gromov bound.  A key question arising in this context is
the following: if $M$ is the zero surgery manifold of a given knot
$K$, how large is~$C_M$?  For instance, for the simplest ribbon knot
$K=6_1$ (the stevedore knot), is $C_M$ less than a billion?

In spite of these desires, almost nothing beyond its existence was
known about the Cheeger-Gromov bound.

\subsection{Main results on the Cheeger-Gromov universal bound}
\label{subsection:main-results-on-universal-bound}

As our first result, we present a topological proof of the existence
of the Cheeger-Gromov bound that directly applies to topological
manifolds, based on the $L^2$-signature defect approach.

\begin{theorem}
  \label{theorem:existence-of-universal-bound-intro}
  For any closed topological $(4k-1)$-manifold $M$, there is a
  constant $C_M$ such that $|\rhot(M,\phi)| \le C_M$ for any
  homomorphism $\phi\colon \pi_1(M) \to G$ to any group~$G$.
\end{theorem}

The outline of the proof is as follows.  As the heart of the argument,
we show that for an arbitrary $(4k-1)$-manifold $M$, there is a
\emph{single} $4k$-manifold $W$ with $\partial W=M$ from which
\emph{every} Cheeger-Gromov invariant $\rhot(M,\phi)$ of $M$ can be
computed as an $L^2$-signature defect.  Once it is proven, it follows
that twice the number of $2$-cells in a CW structure of $W$ is a
Cheeger-Gromov bound, by using the observation that any $L^2$-signature
of $W$ is not greater than the number of $2$-cells.  A key ingredient
used to show the existence of $W$ is a \emph{functorial} embedding of
groups into acyclic groups due to Baumslag, Dyer, and
Heller~\cite{Baumslag-Dyer-Heller:1980-1}.  More details are discussed
in Section~\ref{section:topological-proof}.

Beyond giving a topological proof of the existence, our approach
provides us a new topological understanding of the Cheeger-Gromov
bound.  For $3$-manifolds, we relate the Cheeger-Gromov bound to the
fundamental $3$-manifold presentations: \emph{triangulations},
\emph{Heegaard splittings}, and \emph{surgery on framed links}, by
giving explicit estimates in terms of topological complexities defined
from combinatorial, group theoretic, and knot theoretic information
respectively.

Regarding triangulations, we consider the following natural
combinatorial measure of how complicated a $3$-manifold is
topologically.  In this paper, a triangulation designates a simplicial
complex structure.

\begin{definition}
  \label{definition:complexity-of-manifold}
  The \emph{simplicial complexity} of a $3$-manifold $M$ is the minimal
  number of $3$-simplices in a triangulation of~$M$.
\end{definition}

The following result relates the combinatorial data to the
Cheeger-Gromov bound, which was analytic, via a topological
method.

\begin{theorem}
  \label{theorem:linear-universal-bound}
  Suppose $M$ is a closed $3$-manifold with simplicial complexity~$n$.
  Then
  \[
  |\rhot(M,\phi)| \le 363090\cdot n
  \]
  for any homomorphism $\phi\colon \pi_1(M) \to G$ to any group~$G$.
\end{theorem}

In the next subsection, we will discuss an application of
Theorem~\ref{theorem:linear-universal-bound} to the complexity theory
of $3$-manifolds.  In the last two subsections of this introduction, we
will introduce two key ingredients of the proof of
Theorem~\ref{theorem:linear-universal-bound} (and
Theorems~\ref{theorem:universal-bound-for-heegaard-splitting} and
\ref{theorem:universal-bound-for-surgery-presentation} below), which
are essentially topological and algebraic respectively.

The linear bound given in Theorem~\ref{theorem:linear-universal-bound}
is \emph{asymptotically optimal}.  To state it formally, we define the
``most efficient'' Cheeger-Gromov bound as a function $\Bsc(n)$ in the
simplicial complexity $n$, as follows:
\[
\Bsc(n) = \sup\bigg\{ |\rhot(M,\phi)|\, \bigg| \,
\begin{tabular}{@{}c@{}}
  $M$ has simplicial complexity $\le n$ and\\
  $\phi$ is a homomorphism of $\pi_1(M)$
\end{tabular}
\bigg\}.
\]
Theorem~\ref{theorem:linear-universal-bound} tells us that $\Bsc(n)$
is at most linear asymptotically.  In other words, $\Bsc(n)\in O(n)$;
recall that $f(n)\in O(g(n))$ if $\limsup_{n\to\infty} |f(n)/g(n)| <
\infty$.  In our case, by
Theorem~\ref{theorem:linear-universal-bound}, we have
\[
\limsup_{n\to\infty}\frac{\Bsc(n)}{n} \le 363090.
\]
Also, recall that the small $o$ notation formalizes the notion that
$f(n)$ is strictly smaller than $g(n)$ asymptotically, that is, $f(n)$
is dominated by~$g(n)$: we say $f(n)\in o(g(n))$ if $\lim_{n\to\infty}
|f(n)/g(n)| = 0$.  As another standard notation, we say that $f(n)\in
\Omega(g(n))$ if $f(n)$ is not dominated by~$g(n)$, that is,
$\limsup_{n\to\infty} |f(n)/g(n)|>0$.  We prove the following result
in Section~\ref{subsection:asymtotically-optimal}.

\begin{theorem}
  \label{theorem:linear-bound-is-optimal}
  $\Bsc(n) \in \Omega(n)$.  In fact,
  $\displaystyle\limsup_{n\to\infty}\frac{\Bsc(n)}{n} \ge
  \frac{1}{288}$.
\end{theorem}

Recall that a Heegaard splitting of a closed 3-manifold is determined
by a mapping class $h$ in the mapping class group $\Mod(\Sigma_g)$ of
a surface $\Sigma_g$ of genus~$g$.  To make it precise, we use the
following convention.  We fix a standard embedding of $\Sigma_g$ into
$S^3$ as in Figure~\ref{figure:lickorish-generators}.  Let $H_1$,
$H_2$ be the inner and outer handlebody that $\Sigma_g$ bounds
in~$S^3$, let $i_j\colon \Sigma_g\to H_j$ ($j=1,2$) be the inclusion,
and let $\alpha_i$ and $\beta_i$ be the basis curves in
Figure~\ref{figure:lickorish-generators}.  Then the mapping class
$h\in\Mod(\Sigma_g)$ of a homeomorphism $f\colon \Sigma_g\to \Sigma_g$
gives a Heegaard splitting $(\Sigma_g, \{\beta_i\}, \{f(\alpha_i)\})$
of the $3$-manifold
\[
M = (H_1 \cup H_2)/i_1(f(x))\sim i_2(x),\ x\in \Sigma_g.
\]
In other words, $M$ is obtained by attaching $g$ $2$-handles to $H_1$
along the curves $f(\alpha_i)$ and then attaching a $3$-handle.  Note
that the identity mapping class gives us~$S^3$.

\begin{figure}[ht]
  \labellist
  \small\hair 0mm
  \pinlabel {$\alpha_1$} at 46 66
  \pinlabel {$\beta_1$} at 34 14
  \pinlabel {$\gamma_1$} at 93 59
  \pinlabel {$\alpha_2$} at 138 66
  \pinlabel {$\beta_2$} at 124 14
  \pinlabel {$\gamma_2$} at 180 59
  \pinlabel {$\gamma_{g-1}$} at 228 59
  \pinlabel {$\alpha_g$} at 264 66
  \pinlabel {$\beta_g$} at 250 14
  \normalsize \pinlabel {$\Sigma_g$} at -10 40
  \endlabellist
  \includegraphics[scale=.8]{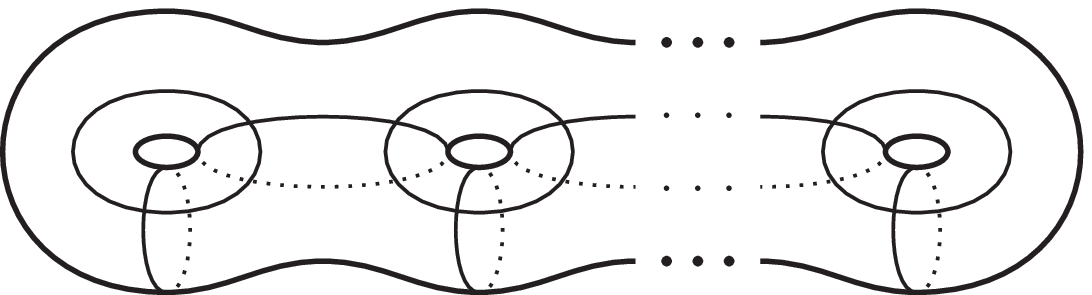}
  \caption{Lickorish's Dehn twist curves.}
  \label{figure:lickorish-generators}
\end{figure}

A natural way to measure its complexity is to consider the word length
of $h$ in the group~$\Mod(\Sigma_g)$.  It is well known that
$\Mod(\Sigma_g)$ is finitely generated by standard Dehn twists;
Lickorish showed that $\Mod(\Sigma_g)$ is generated by the $\pm1$ Dehn
twists about the $3g-1$ curves $\alpha_i$, $\beta_i$, and $\gamma_i$
shown in
Figure~\ref{figure:lickorish-generators}~\cite{Lickorish:1962-1}.

\begin{definition}
  \label{definition:heegaard-lickorish-complexity}
  The \emph{Heegaard-Lickorish complexity} of a closed $3$-manifold $M$
  is defined to be the minimal word length, with respect to the
  Lickorish generators, of a mapping class $h\in \Mod(\Sigma_g)$ which
  gives a Heegaard splitting of~$M$.
\end{definition}

The above geometric group theoretic data is related to the
Cheeger-Gromov bound by the following result, which we obtain by
combining Theorem~\ref{theorem:linear-universal-bound} with a result
in \cite{Cha:2015-1} (see
Section~\ref{subsection:bound-from-surgery}).

\begin{theorem}
  \label{theorem:universal-bound-for-heegaard-splitting}
  If $M$ is a closed $3$-manifold with Heegaard-Lickorish
  complexity~$\ell$, then
  \[
  |\rhot(M,\phi)| \le 251258280 \cdot \ell
  \]
  for any homomorphism $\phi\colon \pi_1(M) \to G$ to any group~$G$.
\end{theorem}

We also relate the Cheeger-Gromov bound to surgery presentations of
$3$-manifolds given as framed links.  For a framed link $L$ in $S^3$,
let $n_i(L)\in\Z$ be the framing on the $i$th component $L_i$, that
is, $n_i(L)=\lk(L_i^{\vphantom{\prime}}, L_i')$ where $L_i'$ is the
parallel copy of $L_i$ taken along the given framing.  We define
$f(L)=\sum_i |n_i(L)|$.  We denote by $c(L)$ the \emph{crossing
  number} of a link $L$ in $S^3$, that is, the minimal number of
crossings of a planar diagram of~$L$.

\begin{theorem}
  \label{theorem:universal-bound-for-surgery-presentation}
  Suppose $M$ is a $3$-manifold obtained by surgery along a framed
  link $L$ in~$S^3$.  Then
  \[
  |\rhot(M,\phi)| \le 69713280 \cdot c(L) + 34856640\cdot f(L)
  \]
  for any homomorphism $\phi\colon \pi_1(M) \to G$ to any group~$G$.
\end{theorem}

The proof is given in Section~\ref{subsection:bound-from-surgery}.

Similarly to Theorem~\ref{theorem:linear-bound-is-optimal}, we show
that the linear bounds in
Theorems~\ref{theorem:universal-bound-for-heegaard-splitting} and
\ref{theorem:universal-bound-for-surgery-presentation} are
asymptotically optimal.  For formal statements and proofs, see
Definition~\ref{definition:best-bounds-wrt-HL-surgery},
Theorem~\ref{theorem:heegaard-lickorish-dehn-surgery-bounds-are-optimal},
and related discussions in
Section~\ref{subsection:lower-bounds-of-lens-space-complexity}.

\begin{remark}
  \label{remark:optimal-coefficient}
  While the linear bounds in
  Theorems~\ref{theorem:linear-universal-bound},
  \ref{theorem:universal-bound-for-heegaard-splitting},
  and~\ref{theorem:universal-bound-for-surgery-presentation} are
  asymptotically optimal, it seems that the \emph{coefficients} in
  these linear bounds can be improved.  Although we do not address it
  in this paper, finding optimal or improved coefficients seems to be
  an interesting problem.
\end{remark}

As an application, our explicit universal bounds for the
Cheeger-Gromov invariants are useful in improving several recent
results in low dimensional topology related to knots, links,
$3$-manifolds, and their $4$-dimensional equivalence relations.  For
instance, by our results above, the proofs of numerous existence
results in \cite{Cochran-Teichner:2003-1, Cochran-Harvey-Leidy:2009-1,
  Kim:2006-1, Cochran-Harvey-Leidy:2008-1,
  Cochran-Harvey-Leidy:2009-3, Franklin:2013-1, Cha:2010-1,
  Cha:2012-1, Cha-Friedl-Powell:2012-1, Cha-Powell:2013-1} now give
explicit examples.  See
Remark~\ref{remark:applications-to-various-explicit-examples} for more
details.

\subsection{Applications to lower bounds of the complexity of
  $3$-manifolds}
\label{subsection:lower-bound-of-complexity}

The notion of the complexity of $3$-manifolds have been an intriguing
subject of study.  In the literature, the following variation of the
simplicial complexity is often considered: a \emph{pseudo-simplicial
  triangulation} of a $3$-manifold is defined to be a collection of
$3$-simplices whose faces are identified in pairs under affine
homeomorphisms to give the $3$-manifold as a quotient space.
Similarly to Definition~\ref{definition:complexity-of-manifold}, the
\emph{pseudo-simplicial complexity} $c(M)$ of a $3$-manifold $M$ is
defined to be the minimal number of $3$-simplices in a
pseudo-simplicial triangulation.  Following conventions in the
literature, we call $c(M)$ the \emph{complexity} of~$M$.  (cf.\ we use
the terminology \emph{simplicial complexity} in
Definition~\ref{definition:complexity-of-manifold} to avoid
confusion.)  In \cite{Matveev:1990-1}, Matveev defines the notion of
complexity using spines in $3$-manifolds, which turns out to be equal
to $c(M)$ for closed irreducible 3-manifolds $M$ except $M=S^3$, $\R P^3$,
and~$L(3,1)$, and develops some fundamental results.

Finding an efficient (pseudo-simplicial) triangulation is essential to
several aspects of $3$-manifold topology, from the normal surface theory
initiated in the 1920's by Kneser, to recent quantum invariants and
computational approaches.  Nonetheless, understanding the complexity
for the general case remains as a difficult problem.  While we easily
obtain an upper bound from a triangulation, finding a lower bound has
been recognized as a hard problem~\cite{Matveev:2003-1,
  Jaco-Rubinstein-Tillman:2013-1}.

We briefly overview known results on lower bounds of~$c(M)$.  In
\cite{Matveev-Pervova:2001-1}, Matveev and Pervova obtain basic lower
bounds of $c(M)$ from $H_1(M)$ and from the presentation length
of~$\pi_1(M)$ (see the end of
Section~\ref{subsection:lower-bounds-of-lens-space-complexity}).  We
remark that in most cases finding the presentation length of a group
is another hard problem.  In~\cite{Matveev-Petronio-Vesnin:2009-1},
Matveev, Petronio, and Vesnin observe and use that for a hyperbolic
$3$-manifold $M$, the Gromov norm $\operatorname{vol}(M)/v_3$ is a
lower bound for $c(M)$, where $v_3$ is the volume of a regular ideal
tetrahedron in~$\mathbb{H}^3$.
In a series of papers \cite{Jaco-Rubinstein-Tillman:2009-1,
  Jaco-Rubinstein-Tillman:2011-1, Jaco-Rubinstein-Tillman:2013-1},
Jaco, Rubinstein and Tillman develop remarkable techniques to
understand the complexity, particularly to find lower bounds, using
double covers and a $\Z_2$-version of the Thurston norm.

As an application of our results on the Cheeger-Gromov bound, we
present new lower bounds of the complexity of $3$-manifolds.  For the
simplicial complexity, note that
Theorem~\ref{theorem:linear-universal-bound} already told us that for
any homomorphism $\phi$ of~$\pi_1(M)$
\[
\frac{1}{363090}\cdot |\rhot(M,\phi)|
\]
is a lower bound.  Since the second barycentric subdivision of a
pseudo-simplicial triangulation is a simplicial complex and since each
tetrahedron in a pseudo-simplicial triangulation gives $(4!)^2=576$
tetrahedra in its second barycentric subdivision, we immediately
obtain the following corollary of
Theorem~\ref{theorem:linear-universal-bound}:

\begin{corollary}
  \label{corollary:lower-bound-of-complexity}
  If $M$ is a closed $3$-manifold, then for any homomorphism $\phi$ of
  $\pi_1(M)$,
  \[
  c(M) \ge \frac{1}{209139840} \cdot |\rhot(M,\phi)|.
  \]
\end{corollary}

Although the constant factor in the above inequality is small, the
Cheeger-Gromov $\rho$-invariants of $3$-manifolds are often so large
that they give interesting new results.  First, we have the following:

\begin{theorem}
  \label{theorem:rho-lower-bound-is-better}
  There are $3$-manifolds $M$ for which the lower bound for $c(M)$ in
  Corollary~\ref{corollary:lower-bound-of-complexity} is arbitrarily
  larger than the lower bound information from \textup{(i)} the
  fundamental group and first homology~\cite{Matveev-Pervova:2001-1},
  \textup{(ii)} the hyperbolic
  volume~\cite{Matveev-Petronio-Vesnin:2009-1}, and \textup{(iii)}
  double covers and $\Z_2$ Thurston
  norm~\cite{Jaco-Rubinstein-Tillman:2009-1,
    Jaco-Rubinstein-Tillman:2011-1, Jaco-Rubinstein-Tillman:2013-1}.
\end{theorem}

In fact, there are $3$-manifolds for which the lower bound in
Corollary~\ref{corollary:lower-bound-of-complexity} grows linearly
while the lower bounds in~\cite{Matveev-Pervova:2001-1},
\cite{Matveev-Petronio-Vesnin:2009-1},
\cite{Jaco-Rubinstein-Tillman:2009-1, Jaco-Rubinstein-Tillman:2011-1,
  Jaco-Rubinstein-Tillman:2013-1} vanish or have logarithmic or square
root growth.  More details are discussed in
Section~\ref{section:complexity-of-3-manifolds}.

As an infinite family of explicit examples, we consider lens spaces.
In~\cite{Jaco-Rubinstein-Tillman:2009-1,
  Jaco-Rubinstein-Tillman:2011-1}, Jaco, Rubinstein, and Tillman
determine the complexity of $L(p,q)$ in certain cases for which $p$ is
even, including the case of~$L(2k,1)$.  Nonetheless, for the general
case, current understanding of the complexity of lens spaces is far
from complete.  In particular, for $L(n,1)$ with $n$ odd, it turns out
that previously known lower bounds are not sharp even asymptotically.
(For more details, see the discussion at the end of
Section~\ref{subsection:lower-bounds-of-lens-space-complexity}.)
In~\cite{Matveev:1990-1} and~\cite{Jaco-Rubinstein:2006-1}, it was
conjectured that for $p>q>0$, $p>3$, if we write $p/q$ as a continued
fraction $[n_0,n_1,\ldots\,]$, then the complexity $c(L(p,q))$ is
equal to $\sum n_i-3$.  It specializes to the following:

\begin{conjecture}[\cite{Matveev:1990-1},~\cite{Jaco-Rubinstein:2006-1}]
  \label{conjecture:complexity-of-L(n,1)}
  For $n>3$, $c(L(n,1))=n-3$.
\end{conjecture}

In~\cite{Jaco-Rubinstein:2006-1}, Jaco and Rubinstein show that
$c(L(n,1))\le n-3$.  In~\cite{Jaco-Rubinstein-Tillman:2009-1}, Jaco,
Rubinstein, and Tillman prove
Conjecture~\ref{conjecture:complexity-of-L(n,1)} for even~$n$.  The
case of odd $n$ is still open.

In the following result, we give a new lower bound for $c(L(n,1))$ for
odd $n$, which tells us that $c(L(n,1))$ with an arbitrary $n$ is
asympotically linear.  Recall that we say $f(n)\in \Theta(g(n))$ if
the asymptotic growth of $f(n)$ and $g(n)$ are identical, that is,
there exist $C_1$, $C_2>0$ such that $C_1 |g(n)| \le |f(n)| \le C_2
|g(n)|$ for all sufficiently large~$n$.

\begin{theorem}
  \label{theorem:asymptotic-complexity-of-L(n,1)}
  $c(L(n,1)) \in \Theta(n)$.  In fact, for each $n>3$,
  \[
  \frac{1}{627419520} \cdot (n-3) \le c(L(n,1)) \le n-3. 
  \]
\end{theorem}

Theorem~\ref{theorem:asymptotic-complexity-of-L(n,1)} supports
Conjecture~\ref{conjecture:complexity-of-L(n,1)} by telling us that it
is asymptotically true.

The proof of Theorem~\ref{theorem:asymptotic-complexity-of-L(n,1)}
employs the Cheeger-Gromov invariants using
Corollary~\ref{corollary:lower-bound-of-complexity}.  More
applications of our results to the complexity of $3$-manifolds will
appear in subsequent papers.  For instance, in \cite{Cha:2015-2}, we
determine the asymptotic growth of the complexity of surgery manifolds
of knots.

\subsection{Efficient $4$-dimensional bordisms over a group}

One of the key ingredients of the proofs of
Theorems~\ref{theorem:linear-universal-bound},
\ref{theorem:universal-bound-for-heegaard-splitting},
and~\ref{theorem:universal-bound-for-surgery-presentation} is a new
result on the existence of an efficient $4$-dimensional bordism over a
group.  More precisely, we address the following problem, which looks
interesting on its own.

We consider manifolds over a group $G$, namely manifolds endowed with
a map to $BG$, the classifying space of~$G$.
As usual, we say that $W$ is a \emph{bordism over $G$ between $M$ and
  $N$} if $\partial W = M\sqcup -N$ as manifolds over~$G$.

\begin{question}
  \label{question:efficient-bordism-to-constant-end}
  Given a $3$-manifold $M$ over $G$, how efficiently can $M$ be bordant
  to a $3$-manifold which is over $G$ via a constant map?
\end{question}

To define the efficiency of a bordism rigorously, we consider the
following natural notion of complexity of a (co)bordism, which is
useful for the study of signature invariants.

\begin{definition}
  \label{definition:2-handle-cell-complexity}
  The \emph{$2$-handle complexity} of a $4$-dimensional smooth/PL
  (co)bordism is the minimal number of $2$-handles in a handle
  decomposition of~$W$.
\end{definition}

Although
Definition~\ref{definition:2-handle-cell-complexity} (as well as
Question~\ref{question:efficient-bordism-to-constant-end}) generalizes
to higher dimensions in an obvious way, in this paper we focus on the
low dimensional case only.

It is a standard fact that any $L^2$-signature of a $4$-manifold (in
particular the ordinary signature) is not greater that the $2$-handle
complexity.  

Suppose $M$ is a triangulated $3$-manifold endowed with a cellular map
$\phi\colon M\to BG$, and $\zeta_M\in C_3(M)$ is the sum of the
oriented $3$-simplices representing the fundamental class.  Then the
Atiyah-Hirzebruch bordism spectral sequence tells us that the
existence of a bordism $W$ from $M$ to another $3$-manifold which is
over $G$ via a constant map is equivalent to the existence of a chain
level analog: such $W$ exists if and only if there exists a $4$-chain
$u\in C_4(BG)$ satisfying $\partial u = \phi_{\#}(\zeta_M)$.  For the
reader's convenience we discuss details as
Lemma~\ref{lemma:null-bordism-and-H_3} in
Section~\ref{subsection:geometric-construction-of-null-cobordism}.

Our result (Theorem~\ref{theorem:existence-of-efficient-bordism}
stated below) concerning
Question~\ref{question:efficient-bordism-to-constant-end} is
essentially that if the chain level analog $u\in C_4(BG)$ of a desired
$W$ exists for $(M,\phi)$, then there exists a corresponding bordism
$W$ whose $2$-handle complexity is controlled \emph{linearly} in the
``size'' of $u$ and~$M$.  To measure the size of a chain, we define an
algebraic notion of diameter as follows:

\begin{definition}
  \label{definition:chain-diameter}
  Suppose $C_*$ is a based chain complex over $\Z$, and
  $\{e^k_\alpha\}$ is the given basis of~$C_k$.  The \emph{diameter}
  $d(u)$ of a $k$-chain $u=\sum_\alpha n^{\vphantom{k}}_\alpha
  e^k_\alpha \in C_k$ is defined to be the $L^1$-norm $d(u) =
  \sum_\alpha |n_\alpha|$.
\end{definition}

Note that the number of tetrahedra in a triangulation of a closed
$3$-manifold $M$ is equal to the diameter of the chain $\zeta_M\in
C_3(M)$ representing the fundamental class.

In order to use the notion of the diameter for a chain in $BG$
(particularly in Theorem~\ref{theorem:existence-of-efficient-bordism}
stated below), we need to fix a CW structure of~$BG$.  It is known
that we can obtain a $K(G,1)$ space $BG$ as the geometric realization
of the simplicial classifying space of $G$ (i.e., the nerve) which is
a simplicial set.  Due to Milnor~\cite{Milnor:1957-3}, this gives us
an explicit CW structure for~$BG$.  In addition, Milnor's geometric
realization tells us that each $n$-cell of $BG$ is naturally
identified with the standard $n$-simplex.  Another useful fact is that
any map of a simplicial complex to $BG$ is homotopic to a cellular
map which, roughly speaking, sends simplices to simplices affinely; we
call such a map \emph{simplicial-cellular}.  We give precise
definitions and provide more details in
Section~\ref{subsection:basic-simplicial-facts} and in the appendix
(in particular see
Definition~\ref{definition:simplicial-cell-complex}).

Now we can state our result about
Question~\ref{question:efficient-bordism-to-constant-end}.

\begin{theorem-named}
  [A special case of
  Theorem~\ref{theorem:existence-of-efficient-bordism}]

  Suppose $M$ is a triangulated closed $3$-manifold with $d(\zeta_M)$
  tetrahedra, and $M$ is over $G$ via a simplicial-cellular map
  $\phi\colon M \to BG$.  If there is a $4$-chain $u\in C_4(BG)$
  satisfying $\partial u = \phi_{\#}(\zeta_M)$, then there exists a
  smooth bordism $W$, between $M$ and a $3$-manifold which is over $G$
  via a constant map, whose $2$-handle complexity is at most $195\cdot
  d(\zeta_M) + 975\cdot d(u)$.
\end{theorem-named}

Our proof provides a geometric construction of a desired bordism $W$
using transversality and surgery arguments over~$G$.  It may be viewed
as a ``geometric realization'' of the algebraic idea of the
Atiyah-Hirzebruch bordism spectral sequence constructed from the exact
couple arising from skeleta.  To control the $2$-handle complexity of
$W$ carefully, we carry out transversality and surgery arguments
\emph{simplicially}.  Details can be found in
Section~\ref{section:chain-vs-bordism}.

We also show that the linear $2$-handle complexity in (the special
case of) Theorem~\ref{theorem:existence-of-efficient-bordism} is
asymptotically best possible.  For precise statements and detailed
discussions, see
Section~\ref{subsection:linear-2-handle-complexity-is-optimal},
particularly Definition~\ref{definition:2-handle-complexity-function}
and Theorem~\ref{theorem:linear-2-handle-complexity-is-best-possible}.

Our linear optimal bound of the $2$-handle complexity in
Theorem~\ref{theorem:existence-of-efficient-bordism} may be compared
with a result of Costantino and
Thurston~\cite{Constantio-Thurston:2008-1} that a closed $3$-manifold
(which is not over a group) of complexity $n$ bounds a $4$-manifold
whose complexity is bounded by~$O(n^2)$.

Theorem~\ref{theorem:existence-of-efficient-bordism} plays an
essential role in the proofs of the explicit estimates of the
Cheeger-Gromov bound in Theorems~\ref{theorem:linear-universal-bound},
\ref{theorem:universal-bound-for-heegaard-splitting},
and~\ref{theorem:universal-bound-for-surgery-presentation}.  Briefly,
we compute the Cheeger-Gromov invariants of a given $3$-manifold~$M$
by using bordism $W$ obtained by applying
Theorem~\ref{theorem:existence-of-efficient-bordism}, and by
controlling the $2$-handle complexity of $W$ efficiently, we obtain
the explicit universal bounds.  For this purpose, we need a chain
level analog $u$ of $W$ required in
Theorem~\ref{theorem:existence-of-efficient-bordism}, and more
importantly, we need to control the diameter of~$u$.  We do this by
applying a general algebraic topological idea discussed in the next
subsection.

\subsection{Controlled chain homotopy}
\label{subsection:intro-controlled-chain-homotopy}

The second key ingredient of the proofs of
Theorems~\ref{theorem:linear-universal-bound},
\ref{theorem:universal-bound-for-heegaard-splitting},
and~\ref{theorem:universal-bound-for-surgery-presentation} is a method
to estimate of the size of certain chain homotopies.  It is best
described using a notion of \emph{controlled chain homotopy}, which we
introduce in this subsection.  It seems to be an interesting algebraic
topological notion on its own, which may be compared with the
topological notion of controlled homotopy.  Readers primarily
interested in controlled chain homotopy may first read this subsection
and then proceed to Section~\ref{section:controlled-chain-homotopy}.

We begin with basic definitions.  Recall that the diameter $d(u)$ of a
chain $u$ is defined to be its $L^1$-norm (see
Definition~\ref{definition:chain-diameter}).  As a convention, we
assume that a chain complex $C_*$ is positive, namely $C_i = 0$ for
$i<0$.

\begin{definition}
  \label{definition:diameter-chain-homotopy}
  Suppose $C_*$ and $D_*$ are based chain complexes, and $P\colon C_*
  \to D_{*+1}$ is a chain homotopy.  We define the \emph{diameter
    function} $d_P\colon \Z\to \Z_{\ge 0}\cup\{\infty\}$ of $P$ by
  \[
  d_P(k) := \max\{d(P(c)) \mid \text{$c\in C_i$ is a basis element,
    $i\le k$}\}.
  \]
  For a partial chain homotopy $P$ defined on $C_i$ for $i\le N$ only,
  we define $d_P(k)$ for $k\le N$ exactly in the same way.

  Let $\delta$ be a function from the domain of $d_P$ to $\Z_{\ge 0}$.
  We say that $P$ is a \emph{$\delta$-controlled
    \textup{(}partial\textup{)} chain homotopy} if $d_P(k)\le
  \delta(k)$ for each $k$ in the domain of~$d_P$.
\end{definition}

Note that $d_P(k)$ may be infinity in general.  If $P$ is a (partial)
chain homotopy defined on a finitely generated chain complex, then
$d_P(k)$ is finite whenever defined.

\begin{definition}
  \label{definition:controlled-chain-homotopy}
  Suppose $\mathcal{S} = \{P_A \colon C^A_* \to D^A_{*+1}\}_{A\in
    \mathcal{I}}$ is a collection of chain homotopies, or a collection
  of partial chain homotopies defined in dimensions $\le n$ for some
  fixed~$n$.  We say that $\mathcal{S}$ is \emph{uniformly controlled
    by $\delta$} if each $P_A$ is a $\delta$-controlled (partial)
  chain homotopy.  The function $\delta$ is called a \emph{control
    function} for~$\mathcal{S}$.
\end{definition}

Our focus is to understand how various families of chain homotopies
can be uniformly controlled.  A few additional words might make it
clearer.  In many case the conclusion of a theorem on chain complexes
can be understood as the existence of a certain chain homotopy, and in
addition, such a theorem usually holds for a collection of objects, so
that it indeed gives a family of chain homotopies indexed by the
objects.  For example, the classical Eilenberg-Zilber theorem says
that $C_*(X\times Y)$ and $C_*(X)\otimes C_*(Y)$ are chain homotopy
equivalent, that is, for \emph{every} $(X,Y)$ there are \emph{chain
  homotopies} which tells us that the chain complexes are chain
homotopy equivalences.  Are these chain homotopies indexed by $(X,Y)$
uniformly controlled?

In general, we consider the following meta-question:

\begin{question}
  \label{question:uniformly-controlled-chain-homotopy}
  Pick a theorem about chain complexes or their homology.  In case of
  based chain complexes or their homology, can the theorem be
  understood in terms of uniformly controlled chain homotopies?  If
  so, find (an estimate of) a control function.
\end{question}

In this paper, we observe several interesting cases for which a family
of uniformly controlled chain homotopies exists, and we analyze the
control functions in detail.

Our first theorem concerns the acyclic model theorem of Eilenberg and
MacLane, which gives a family of functorial chain homotopies.  As a
fundamental observation, we show that if we use finitely many models
in each dimension, then there is a single control function $\delta$
such that all the resulting functorial chain homotopies obtained by an
acyclic model argument are uniformed controlled by~$\delta$.
This result, which we call
a \emph{controlled acyclic model theorem}, is stated as
Theorem~\ref{theorem:acyclic-model-theorem-with-control}.  We discuss
more details in
Section~\ref{subsection:controlled-acyclic-model-theorem}.

As an application, we apply the controlled acyclic model theorem to
products.  In
Section~\ref{subsection:controlled-eilenberg-zilber-theorem}, we
consider simplicial sets and the Moore complexes of the associated
freely generated simplicial abelian groups, as a general setup for
products and based chain complexes.  We present a \emph{controlled
  Eilenberg-Zilber theorem}, which essentially says that the chain
homotopy equivalence between the chain complex of a product and the
tensor products of chain complexes can be understood in terms of
uniformly controlled functorial chain homotopies.  See
Theorem~\ref{theorem:controlled-eilenberg-zilber-chain-homotopy} for
more details.

We also consider the context of group homology.  Recall that
conjugation on a group induces the identity on the homology with
integral coefficients.  We give a quantitative generalization of this
in terms of controlled chain homotopies.  For a precise statement and
related discussions, see
Theorem~\ref{theorem:controlled-conjugation-chain-homotopy} and
Section~\ref{subsection:controlled-chain-homotopy-for-conjugation}.

We give another uniformly controlled chain homotopy result, concerning
the result of Baumslag, Dyer, and
Heller~\cite{Baumslag-Dyer-Heller:1980-1} which was already mentioned
as a key ingredient of our topological proof of the existence of the
Cheeger-Gromov bound
(Theorem~\ref{theorem:existence-of-universal-bound-intro}): there is a
functorial embedding, say $i_G\colon G\hookrightarrow \cA(G)$, of a
group $G$ to an acyclic group $\cA(G)$ for each group~$G$.  From the
viewpoint of controlled chain homotopy, the following natural question
arises: for each $G$, is there a chain homotopy between the chain maps
induced by the identity $\id_{\cA(G)}$ and the trivial endomorphism of
$\cA(G)$, which forms a uniformly controlled family?

We give a partial answer.  In~\cite{Baumslag-Dyer-Heller:1980-1}, for
each $n\ge 1$, they constructed a functorial embedding that we denote
by $i^n_G\colon G \to \bA^n(G)$, which induces a zero map
$H_i(G;\mathds{k}) \to H_i(\bA^n(G);\mathds{k})$ for $1\le i\le n$ and
any field~$\mathds{k}$.  (See Definition~\ref{definition:mitosis} for
a precise description of~$\bA^n(G)$.)  This may be viewed as an
approximation of a functorial embedding into acyclic groups up to
dimension~$n$; in fact it turns out that $\varinjlim \bA^n(G)$ is
acyclic and $G$ embeds into it functorially.  The following result is
a controlled chain homotopy generalization of the homological property
of~$i^n_G$.

\begin{theorem-named}[Theorem~\ref{theorem:controlled-chain-homotopy-BDH}]
  For each $n$, there is a family $\{\Phi^n_G \mid G$ \textup{is a
    group}$\}$ of partial chain homotopies $\Phi^n_G$ defined in
  dimension $\le n$ between the chain maps induced by the trivial map
  $e\colon G\to \bA^n(G)$ and the embedding $i^n_G\colon G\to
  \bA^n(G)$, which is uniformly controlled by a function~$\dBDH$.  For
  $k\le 4$, the value of $\dBDH(k)$ is as follows.

  \smallskip
  \begin{center}
    \normalfont
    \begin{tabular}
      {c*{5}{p{\widthof{0000}}<{\centering}}}
      \toprule
      $k$ & 0 & 1 & 2 & 3 & 4\\
      $\dBDH(k)$ & 0 & 6 & 26 & 186 & 3410\\
      \bottomrule
    \end{tabular}
  \end{center}
\end{theorem-named}

Our proof of Theorem~\ref{theorem:controlled-chain-homotopy-BDH}
consists of a careful construction of the chain homotopy $\Phi^n_G$
and its diameter estimate, using the above results on the acyclic
model theorem and conjugation.  We provide more detailed discussions
and proofs in Section~\ref{section:chain-homotopy-for-mitosis}.

We remark that Theorem~\ref{theorem:controlled-chain-homotopy-BDH} for
$n=3$ (together with $\dBDH(3)=186$) is sufficient for our proofs of
the Cheeger-Gromov bound estimates for $3$-manifolds.  See
Section~\ref{universal-bound-computation} for more details.

\subsubsection*{Organization of the paper}

In Section~\ref{section:topological-proof}, we review the
$L^2$-signature approach to the Cheeger-Gromov $\rho$-invariant and
give a proof of
Theorem~\ref{theorem:existence-of-universal-bound-intro}.  In
Section~\ref{section:chain-vs-bordism}, we give a construction of
$4$-dimensional bordisms and estimate the $2$-handle complexity to
prove Theorem~\ref{definition:2-handle-cell-complexity}.  In
Section~\ref{section:controlled-chain-homotopy}, we develop the basic
theory of controlled chain homotopy, including a controlled acyclic
model theorem.  In Section~\ref{section:chain-homotopy-for-mitosis},
we present a chain level approach to the result of
Baumslag-Dyer-Heller.  In Section~\ref{universal-bound-computation},
we obtain explicit estimates for the Cheeger-Gromov universal bound by
proving Theorems~\ref{theorem:linear-universal-bound},
\ref{theorem:universal-bound-for-heegaard-splitting},
and~\ref{theorem:universal-bound-for-surgery-presentation}.  In
Section~\ref{section:complexity-of-3-manifolds}, we discuss the
application to the complexity of $3$-manifolds, and prove that our
linear Cheeger-Gromov bounds and geometric construction of efficient
bordisms are asymptotically optimal.  In the appendix, we discuss
basic definitions and facts on simplicial sets and simplicial
classifying spaces which we use in this paper, for the reader's
convenience.

\subsubsection*{Acknowledgements}

The author thanks an anonymous referee for helpful comments.  The
author thanks for the hospitality of Indiana University at
Bloomington, where part of this paper was written.  This work was
partially supported by NRF grants 2013067043 and 2013053914.

\section{Existence of universal bounds}
\label{section:topological-proof}

In this section we give a topological proof of the existence of a
universal bound for the Cheeger-Gromov invariant~$\rhot(M,\phi)$.

\subsection{A topological definition of the Cheeger-Gromov $\rho$-invariant}
\label{subsection:topological-definition-rho}

We begin by recalling a known topological definition
of~$\rhot(M,\phi)$.  We follow the approach introduced by Chang and
Weinberger~\cite{Chang-Weinberger:2003-1}; see also Harvey's
work~\cite{Harvey:2006-1}.

Suppose $M$ is a closed topological $(4k-1)$-manifold, and
$\phi\colon\pi_1(M) \to G$ is a homomorphism.  When $X$ is not path
connected, as a convention, we denote by $\pi_1(X)$ the free product
(= coproduct) $\coprod_\alpha \pi_1(X_\alpha)$ of the fundamental
groups of the path components $X_\alpha$ of~$X$.  Suppose $W$ is a
$4k$-manifold with $\partial W = rM$, $r$ disjoint copies of~$M$.
Suppose there are a monomorphism $G\hookrightarrow \Gamma$ and a
homomorphism $\pi_1(W)\to \Gamma$ which make the following diagram
commute:
\begin{equation}
  \label{equation:bounding-manifold-diagram}
  \begin{gathered}
    \hphantom{\hbox{\hss$\coprod\limits^r \pi_1(M)=\hbox{}$}} 
    \xymatrix@R=3.5em@C=3.5em@M=.5em{
      \hbox to 0mm{\hss$\coprod\limits^r \pi_1(M)=\hbox{}$}\pi_1(rM)
      \ar[r]^-{\coprod\phi} \ar[d]_{i_*} &
      G \ar@{^(.>}[d]
      \\
      \pi_1(W) \ar@{.>}[r] &
      \Gamma
    }
  \end{gathered}
\end{equation}

For a (discrete) group $\Gamma$, the \emph{group von Neumann algebra}
$\N\Gamma$ is defined as an algebra over $\C$ with involution.
L\"uck's book \cite{Lueck:2002-1} is a useful general reference
on~$\N\Gamma$; see also his paper~\cite{Lueck:1998-1}.  In this paper
we need the following known facts on $\N\Gamma$: (i) $\C\Gamma\subset
\N\Gamma$ as a subalgebra.  Consequently, in our case, $\N\Gamma$ is a
local coefficient system over $W$ via $\C[\pi_1(W)] \to \C\Gamma
\subset \N\Gamma$.  The homology $H_*(W;\N\Gamma)$ is defined as
usual, and by Poincar\'e duality, the intersection form
\[
\lambda\colon H_{2k}(W;\N \Gamma)\times H_{2k}(W;\N \Gamma)\to
\N\Gamma
\]
is defined.  (ii) $\N\Gamma$ is semihereditary, that is, any finitely
generated submodule of a finitely generated projective module over
$\N\Gamma$ is projective; consequently, in our case,
$H_{2k}(W;\N\Gamma)$ is a finitely generated module over~$\N\Gamma$.
(iii) For any hermitian form over a finitely generated
$\N\Gamma$-module, there is a spectral decomposition; in our case, for
the intersection form $\lambda$, we obtain an orthogonal direct sum
decomposition
\begin{equation}
  \label{equation:intersection-form-spectral-decomposition}
  H_{2k}(W;\N \Gamma) = V_+\oplus V_-\oplus V_0
\end{equation}
such that $\lambda$ is zero, positive definite, and negative definite
on $V_0$, $V_+$ and $V_-$ respectively; the positive definiteness
means that $\lambda(x,x)=a^*a$ for some nonzero $a \in \N G$ whenever
$x\in V_+$ is nonzero.  (iv) There is a dimension function
\[
\ldim_\Gamma\colon \{\mbox{finitely generated }\N\Gamma\mbox{-modules}\} \to
\R_{\ge 0}
\]
which is additive for short exact sequences and satisfies
$\ldim_\Gamma(\N\Gamma)=1$.

The \emph{$L^2$-signature} of $W$ over $\Gamma$ is defined to be
\[
\lsign_\Gamma W = \ldim_\Gamma V_+ - \ldim_\Gamma V_-.
\]
Now the \emph{$L^2$ $\rho$-invariant} of $(M,\phi)$ is defined to be
the signature defect
\begin{equation}
  \label{equation:rho-as-l2-signature-defect}
  \rhot(M,\phi) = \frac 1r \big(\lsign_\Gamma W - \sign W\big)
\end{equation}
where $\sign W$ denotes the ordinary signature of~$W$.

It is known that this topological definition of $\rhot(M,\phi)$ is
equivalent to the definition of Cheeger and Gromov given
in~\cite{Cheeger-Gromov:1985-1} in terms of $\eta$-invariants.  The
proof depends on the $L^2$-index theorem for manifolds with
boundary~\cite{Cheeger-Gromov:1985-1,Ramachandran:1993-1} and the fact
that various definitions of $L^2$-signatures are
equivalent~\cite{Lueck-Schick:2003-1}.  We remark that Cochran and
Teichner present an excellent introduction to the analytic definition
of $\rho(M,\phi)$ in \cite[Section~2]{Cochran-Teichner:2003-1}.

Although the $L^2$-signature defect definition involves the bounding
manifold $W$ (and the enlargement $\Gamma$ of the given $G$), it is
known that a topological argument using bordism theory shows that such
a $W$ always exists and that $\rhot(M,\phi)$ in
\eqref{equation:rho-as-l2-signature-defect} is independent of the
choice of $W$, without appealing to analytic index theory.  To the
knowledge of the author, this method for the $L^2$-case first appeared
in~\cite{Chang-Weinberger:2003-1}.  Since it is closely related to our
techniques for the universal bound of the $\rho$-invariants that will
be discussed in later sections, we give a proof below, without
claiming any credit.

For the existence of $W$, we use a result of Kan and
Thurston~\cite{Kan-Thurston:1976-1} that a group $G$ embeds into an
acyclic group, say~$\Gamma$.  Denote by $\Omega^\STOP_*$ and
$\Omega^\STOP_*(X)$ the oriented topological cobordism and bordism
groups.  By the foundational work of
Kirby-Siebenmann~\cite{Kirby-Siebenmann:1977-1} and
Freedman-Quinn~\cite{Freedman-Quinn:1990-1}, $\Omega^\STOP_*(X)$ is a
generalized homology theory.  Since $H_p(\Gamma)=0$ for $p\ne 0$, all
the $E^2$ terms of the Atiyah-Hirzebruch spectral sequence
\[
E^2_{pq} = H_p(\Gamma)\otimes \Omega^\STOP_q \Longrightarrow
\Omega^\STOP_n(B\Gamma)
\]
vanish except $E^2_{0,n}=\Omega_n^\STOP$.  It follows that the
inclusion $\{*\} \hookrightarrow B\Gamma$ induces an isomorphism
$\Omega^\STOP_{n} \cong \Omega^\STOP_{n}(B\Gamma)$.  Since
$\Omega^\STOP_{4k-1}\otimes\Q\cong \Omega^\SO_{4k-1}\otimes\Q = 0$ due
to Thom's classical work~\cite{Thom:1954-1}, it follows that $rM$
bounds a $4k$-manifold $W$ over $B\Gamma$ for some $r>0$.  This gives
us the diagram~\eqref{equation:bounding-manifold-diagram}.

For the independence of the choice of $W$, suppose the
diagram~\eqref{equation:bounding-manifold-diagram} is also satisfied
for $(W', r', \Gamma')$ in place of $(W, r, \Gamma)$.  By
$L^2$-induction (see, e.g.,
\cite[Equation~(2.3)]{Cheeger-Gromov:1985-1},
\cite[p.~253]{Lueck:2002-1},
\cite[Proposition~5.13]{Cochran-Orr-Teichner:1999-1}), $\lsign_\Gamma$
is left unchanged when $\Gamma$ is replaced by another group
containing $\Gamma$ as a subgroup.  Thus we may assume that
$\Gamma=\Gamma'$ by replacing $\Gamma$ and $\Gamma'$ with the
amalgamated product of them over $G$, and furthermore we may assume
that $\Gamma$ is acyclic using Kan-Thurston.  Let
$V=r'W\cup_{rr'M}-rW'$.  Then $V$ is a closed $4k$-manifold
over~$\Gamma$.  Since $\Gamma$ is acyclic, $\Omega^\STOP_{4k}\cong
\Omega^\STOP_{4k}(B\Gamma)$, and therefore $V$ is bordant to another
$V'$ which is over $B\Gamma$ via a constant map.  We have
$\lsign_\Gamma V' = \sign V'$.  Using Novikov additivity and that
$\lsign$ and $\sign$ are bordism invariants, we obtain
\begin{multline*}
\frac 1r \big(\lsign_\Gamma W-\sign W\big) -
\frac 1{r'} \big(\lsign_\Gamma W'-\sign W'\big)
\\
= \frac{1}{rr'}(\lsign_\Gamma V-\sign V) 
= \frac{1}{rr'}(\lsign_\Gamma V' - \sign V') = 0.
\end{multline*}

We remark that we may assume the codomain $G$ of $\phi\colon
\pi_1(M)\to G$ is countable.  In fact, by $L^2$-induction,
$\rhot(M,\phi)$ is left unchanged when $G$ is replaced by the
countable group $\phi(\pi_1(M))$.

\subsection{Existence of a universal bound}

In this subsection we give a new proof of the existence of the
Cheeger-Gromov universal bound, which applies directly to topological
manifolds.  Recall
Theorem~\ref{theorem:existence-of-universal-bound-intro} from the
introduction: \emph{for any closed topological $(4k-1)$-manifold $M$,
  there is a constant $C_M$ such that $|\rhot(M,\phi)|\le C_M$ for any
  homomorphism $\phi$ of~$\pi_1(M)$.}

In proving this using the topological definition of the Cheeger-Gromov
invariants in Section~\ref{subsection:topological-definition-rho}, it
is crucial to understand the ``size'' of the bounding $4k$-manifold
$W$, since $\rhot(M,\phi)$ is given by the $L^2$-signature defect of
$W$ as in~\eqref{equation:rho-as-l2-signature-defect}.  The key
difficulty which is well known to experts is that the $4k$-manifold
$W$ in Section~\ref{subsection:topological-definition-rho} depends on
$\phi\colon \pi_1(M) \to G$ in general, since $W$ is obtained by
appealing to bordism theory over an acyclic group $\Gamma$, which
depends on the group~$G$.

We resolve this difficulty by employing the following
\emph{functorial} embedding of groups into acyclic groups, which was
given by Baumslag, Dyer, and Heller.

\begin{theorem}[{Baumslag-Dyer-Heller~\cite[Theorem~5.5]{Baumslag-Dyer-Heller:1980-1}}]
  \label{theorem:BDH-acyclic-enlargement}
  There exists a functor $\cA\colon \Gp\to \Gp$ on the category $\Gp$
  of groups and a natural transformation $\iota\colon \id_\Gp \to \cA$
  such that $\cA(G)$ is acylic and $\iota_G\colon G \to \cA(G)$ is
  injective for any group~$G$.
\end{theorem}

We remark that $\cA(G)$ given in \cite{Baumslag-Dyer-Heller:1980-1}
has the same cardinality as $G$ if $G$ is infinite, and is generated
by $(n+5)$ elements if $G$ is generated by $n$ elements.

\begin{proof}[Proof of
  Theorem~\ref{theorem:existence-of-universal-bound-intro}]
  Consider $\iota_{\pi_1(M)}\colon \pi_1(M) \to \cA(\pi_1(M))$ given
  by Theorem~\ref{theorem:BDH-acyclic-enlargement}.  Since
  $\cA(\pi_1(M))$ is acyclic, there is a $4k$-manifold $W$ bounded by
  $rM$ over $\cA(\pi_1(M))$ for some $r>0$, by the bordism argument in
  Section~\ref{subsection:topological-definition-rho}.  Suppose
  $\phi\colon \pi_1(M) \to G$ is arbitrarily given.  Let
  $\Gamma:=\cA(G)$.  Then we have the following commutative diagram,
  by the functoriality of~$\cA$:
  \[
  \xymatrix@M=.7ex@R=2.5em@C=2.5em{
    {\displaystyle\coprod^r \pi_1(M)} \ar[dd]_{i_*} \ar[rr]^-{\coprod \phi}
    \ar@{^{(}->}[rd]^-{\coprod\iota_{\pi_1(M)}} &
    &
    G \ar@{^{(}->}[rd]^-{\iota_G}
    \\
    &
    \cA(\pi_1(M)) \ar[rr]^-{\cA(\phi)} &
    &
    \cA(G) = \Gamma
    \\
    \pi_1(W) \ar[ru]
  }
  \]
  From this it follows that we can define $\rhot(M,\phi)$ as the
  $L^2$-signature defect of $W$ over $\Gamma$, as
  in~\eqref{equation:rho-as-l2-signature-defect}.  Note that our $W$
  is now independent of the choice of~$\phi$.

  Recall that $W$ has the homotopy type of a finite CW complex.  Let
  $C_*(W;\N\Gamma)$ be the cellular chain complex defined using this
  CW structure.  We have $C_{2k}(W;\N\Gamma)\cong (\N\Gamma)^N$ where
  $N$ is the number of the $2k$-cells.  By the additivity of the
  $L^2$-dimension under short exact sequences, we have
  \begin{align*}
    |\lsign_{\Gamma}W| &\le \ldim_\Gamma V_+ + \ldim_\Gamma V_- \\
    &\le \ldim_\Gamma H_{2k}(W;\N\Gamma) \le \ldim_\Gamma
    C_{2k}(W;\N\Gamma)=N.
  \end{align*}
  A similar argument shows that $|\sign W|\le N$.
  By~\eqref{equation:rho-as-l2-signature-defect}, it follows that
  $|\rhot(M,\phi)|\le 2N$.  This completes the proof, since $W$, and
  consequently $N$, are independent of the choice of $\phi$ and~$G$.
\end{proof}

\section{Construction of bordisms and $2$-handle complexity}
\label{section:chain-vs-bordism}

In this section, we introduce a general geometric construction which
relates chain level algebraic data to a $4$-dimensional bordism of a
given $3$-manifold.  It may be viewed as a geometric incarnation of
the Atiyah-Hirzebruch bordism spectral sequence.  Furthermore, we give
a more thorough analysis to obtain an explicit relationship between
the complexity of the given algebraic data and the number of the
$2$-handles of an associated $4$-dimensional bordism.

The results in this section will be used to reduce the problem of
finding a universal bound for the $\rho$-invariants to a study of
algebraic topological chain level information.

\subsection{Geometric construction of bordisms}
\label{subsection:geometric-construction-of-null-cobordism}

We begin with a straightforward observation on the Atiyah-Hirzebruch
bordism spectral sequence, which is stated as
Lemma~\ref{lemma:null-bordism-and-H_3} below.  In this and following
sections, we consider the category of spaces $X$ endowed with a map
$\phi\colon X\to K$, where $K$ is a fixed connected CW complex.  We
say that $X$ \emph{is over}~$K$.  If $K=B\Gamma$ for a group $\Gamma$,
we say that $X$ \emph{is over}~$\Gamma$.  In this case we often view
$\phi\colon X\to K$ as $\phi\colon\pi_1(X)\to \Gamma$ and vice versa.

We say that $X$ is \emph{trivially over~$K$} if $X$ is endowed with a
constant map to~$K$.

\begin{definition}
  A bordism $W$ with $\partial W = M\sqcup -N$ over $K$ is called a
  \emph{bordism between $M$ and a trivial end} if $N$ is trivially
  over~$K$.
\end{definition}

\begin{lemma}
  \label{lemma:null-bordism-and-H_3}
  For a closed $3$-manifold $M$ endowed with $\phi\colon M\to K$, the
  following are equivalent.
  \begin{enumerate}
  \item $M$ bounds a smooth $4$-manifold $V$ over~$K$.
  \item There is a smooth bordism $W$ over $K$ between $M$ and a
    trivial end.
  \item\label{lemma-item:fundamental-class-vanishes} The image
    $\phi_*[M]$ of the fundamental class $[M]\in H_3(M)$ is zero in
    $H_3(K)$.
  \end{enumerate}
\end{lemma}

\begin{proof}
  (1)~implies (2) obviously.  (2)~implies (1) since $N:=\partial W
  \setminus M$ bounds a $4$-manifold which can be used to cap off~$W$.
  From the Atiyah-Hirzebruch spectral sequence
  \[
  E^2_{p,q}=H_p(K)\otimes \Omega^\SO_q \Longrightarrow
  \Omega^\SO_n(K)
  \]
  and from that $\Omega^\SO_0=\Z$, $\Omega^\SO_1 = \Omega^\SO_2 =
  \Omega^\SO_3= 0$, it follows that $\Omega^\SO_3(K)\cong H_3(K)$
  under the isomorphism sending the bordism class of $\phi\colon M\to
  K$ to $\phi_*[M]\in H_3(K)$.  This shows that (1) is equivalent
  to~(3).
\end{proof}

\begin{remark}
  \label{remark:rho-from-bordism}
  If $(M,\phi)$ is as in Lemma~\ref{lemma:null-bordism-and-H_3} and
  $K=B\Gamma$, then $\rhot(M,\phi)$ can be defined as the
  $L^2$-signature defect of the bordism $W$ in
  Lemma~\ref{lemma:null-bordism-and-H_3}~(2), as well as $V$ in
  Lemma~\ref{lemma:null-bordism-and-H_3}~(1).  For, if $N$ is over
  $\Gamma$ via $\psi$ and $\partial W = M\sqcup -N$ over $\Gamma$,
  then $\rhot(M,\phi)-\rhot(N,\psi)$ is the $L^2$-signature defect
  of~$W$ by~\eqref{equation:rho-as-l2-signature-defect}, and since the
  $L^2$-signature over a trivial map is equal to the ordinary
  signature, we have $\rhot(N,\psi)=0$ if $\psi$ is trivial.
\end{remark}

Suppose $M$ is a closed $3$-manifold equipped with a CW structure,
whose $3$-cells are oriented so that the sum $\zeta_M$ of the
$n$-cells is a cycle representing the fundamental class $[M]\in
H_n(M)$.  We may assume that $\phi\colon M\to K$ is cellular by
appealing to the cellular approximation theorem.  Let $\phi_{\#}$ be
the chain map on the cellular chain complex $C_*(-)$ induced
by~$\phi$.  Then we can restate
Lemma~\ref{lemma:null-bordism-and-H_3}~(3) as follows:

\begin{theorem-named}[Addendum to Lemma~\ref{lemma:null-bordism-and-H_3}]
  $(3)'$ $\phi_{\#}(\zeta_M) = \partial u$ for some $4$-chain $u$
  in~$C_4(K)$.
\end{theorem-named}

The goal of this section is to discuss a more explicit relationship of
the $4$-dimensional bordism $W$ in
Lemma~\ref{lemma:null-bordism-and-H_3}~$(2)$ and the $4$-chain $u$ in
Lemma~\ref{lemma:null-bordism-and-H_3}~$(3)'$.

As an easier direction, if $W$ is a bordism between $M$ and a trivial
end $N$, then for the sum $\zeta_W$ of oriented $4$-cells of $W$ which
represent the fundamental class of $(W,\partial W)$, we have $\partial
\zeta_W = \zeta_M - \zeta_N$.  Since the image of $\zeta_N$ in
$C_3(K)$ is zero, the image $u\in C_4(K)$ of $\zeta_W$ satisfies
$\partial u = \phi_{\#}(\zeta_M)$.

For the converse, for a given $4$-chain $u\in C_4(K)$ satisfying
Lemma~\ref{lemma:null-bordism-and-H_3}~$(3)'$, we will present a
construction of a bordism $W$ between $M$ and a trivial end.  The rest
of this subsection is devoted to this.  This will tell us how the
Atiyah-Hirzebruch spectral sequence is reinterpreted as a geometric
construction, and provide us the foundational idea of the more
sophisticated analysis accomplished in
Section~\ref{subsection:complexity-of-bordism}.

\begin{step-named}[Preparation and strategy]

  As above, suppose a given closed $3$-manifold $M$ has a fixed CW
  complex structure, and $\phi\colon M \to K$ is cellular.  Suppose
  $\phi_{\#}(\zeta_M) = \partial u$ for some $u \in C_4(K)$.

  Our construction of $W$ is based on the following observation.  Let
  $K^{(i)}$ be the $i$-skeleton of~$K$.  By Atiyah-Hirzebruch,
  $\Omega^\SO_3(K)$ is filtered by
  \[
  \Omega^\SO_3(K) = J_{3} \supset J_{2} \supset J_{1} \supset
  J_{0} \supset J_{-1}=0
  \]
  where $J_{i}=\Im\{\Omega^\SO_3(K^{(i)})\to \Omega^\SO_3(K)\}$, and as
  in the proof of Lemma~\ref{lemma:null-bordism-and-H_3}, we have
  \begin{equation}
    \label{equation:spectral-squence-filtration}
    J_{i}/J_{i-1} \cong E^\infty_{i,3-i} \cong E^2_{i,3-i} =
    H_i(K)\otimes \Omega^\SO_{3-i} =
    \begin{cases}
      H_3(K) &\text{if }i=3\\
      0 &\text{if }i=0,1,2.
    \end{cases}    
  \end{equation}
  Let $M_3:=M$.  Obviously $\phi$ maps $M_3$ to~$K^{(3)}$.  For
  $i=3$, \eqref{equation:spectral-squence-filtration} tells us that
  the existence of $u$ implies that the bordism class of $(M_3,\phi)$
  in $\Omega^\SO_3(K^{(3)})$ lies in the image of
  $\Omega^\SO_3(K^{(2)})$, that is, there is a bordism $W_3$ over $K$
  between $M_3$ and another $3$-manifold, say $M_2$, such that $M_2$
  maps to~$K^{(2)}$.  Similarly, for $i=2$ and then for $i=1$,
  \eqref{equation:spectral-squence-filtration} tells us that
  $\Omega^{\SO}_{3-i}=0$ implies that $M_i$ over $K^{(i)}$ admits a
  bordism $W_i$ over $K$ to another $3$-manifold $M_{i-1}$ that maps
  to~$K^{(i-1)}$.

  Once we have the bordisms $W_i$ for $i=3,2,1$, by concatenating
  them, we obtain a bordism $W$ between the given $M$ and the
  $3$-manifold $N:=M_0$.  Since $K$ is a connected CW complex, $N\to
  K^{(0)}$ is homotopic to a constant map.  By altering the map $W\to
  K$ on a collar neighborhood of $N$ using the homotopy, we may assume
  that $N$ is over $K$ via a constant map.  This gives a desired
  bordism~$W$ between the given $M$ and a trivial end~$N$.

  In Steps~1, 2, and~3 below, we present how to actually construct
  $W_3$, $W_2$, and $W_1$, using the given $u$ and the facts
  $\Omega^\SO_{3-i}=0$, respectively.
\end{step-named}

\begin{step-named}[Step~1: Reduction to the $2$-skeleton~$K^{(2)}$]

  We will construct $W_3$ using the given $4$-chain~$u$.  Denote the
  characteristic map of a $4$-cell $e^4_\alpha$ of $K$ by
  $\phi_\alpha\colon D^4_\alpha\to K^{(4)}$ where $D^4_\alpha$ is a
  $4$-disk.  We may assume that the center of each $3$-cell of $K$ is
  a regular value of $\phi\colon M\to K^{(3)}$ and a regular value of
  each attaching map $\phi_\alpha|_{\partial
    D^4_\alpha}\colon \partial D^4_\alpha \to K^{(3)}$.  Write the
  $4$-chain $u$ as $u=-\sum_\alpha n_\alpha^{\vphantom{4}}
  e_\alpha^4$, and consider the $4$-manifold $X = M\times[0,1] \sqcup
  \bigsqcup_\alpha n^{\vphantom{1}}_\alpha D^4_\alpha$.  View $X$ as a
  bordism over $K$ between $M\times 0$ and $M' := \partial X \setminus
  M\times 0$, via the map $X\to K$ induced by $\phi$ composed with the
  projection $M\times[0,1]\to M$ and the maps~$\phi_\alpha$.  Let
  $\psi\colon M' \to K$ be its restriction.  The relation
  $\phi_{\#}(\zeta_M)-\partial u=0$ implies that for the center $y$ of
  each $3$-cell of $K$, the points in $\psi^{-1}(y)\in M'$ signed by
  the local degree are cancelled in pairs.  For each cancelling pair,
  attach to $X$ a $1$-handle joining these; the attaching $0$-sphere
  is framed by pulling back a fixed framing at the regular value~$y$,
  as usual.  Let $W_3$ be the resulting cobordism, which is from
  $M=M\times 0$ to another $3$-manifold, say~$M_2$.  The map $\psi$
  induces a map $W_3 \to K^{(4)}$ which maps $M \sqcup M_2$
  to~$K^{(3)}$.  In addition, the image of $M_2$ is disjoint from the
  centers of $3$-cells in~$K^{(3)}$.  It follows that by a homotopy on
  a collar neighborhood, we may assume that $M_2$ is mapped
  to~$K^{(2)}$.  This completes Step~1, as summarized in the following
  diagram:
  \[
  \xymatrix@M=.5em@R=3.5em@C=2em{
    &&
    \hbox to 0mm{\hss$M=$ }M_3 \ar@{^(->}[r] \ar[d]^(.55){\phi}
    &
    W_3 \ar@{..>}[llld]|(.365)\hole
    &
    M_2 \ar@{_(->}[l] \ar@{->}[d]^{\phi_2}
    \\
    K^{(4)} 
    &&
    K^{(3)} \ar@{_(->}[ll]
    &
    &
    K^{(2)} \ar@{_(->}[ll]
  }
  \]
\end{step-named}

\begin{step-named}[Step~2: Reduction to the $1$-skeleton $K^{(1)}$]

  For the map $\phi_2\colon M_2 \to K^{(2)}$ obtained above, we may
  assume that the center $y$ of a $2$-cell of $K^{(2)}$ is a regular
  value.  Then $\phi_2^{-1}(y)$ is a disjoint union of framed circles
  in~$M_2$.  Take $M_2\times[0,1]$, and attach $2$-handles along the
  components of the framed $1$-manifold $\phi_2^{-1}(y)\times 1\subset
  M_2\time 1$.  This gives a $4$-dimensional cobordism $W_2$ from
  $M_2=M_2\times 0$ to another $3$-manifold $M_1$, and $\phi_2$
  extends to $W_2 \to K^{(2)}$.  By the construction, the image of
  $M_1$ in $K^{(2)}$ is disjoint from the centers of $2$-cells.
  Therefore by a homotopy we may assume that $W_2 \to K^{(2)}$
  restricts to a map $\phi_1\colon M_1\to K^{(1)}$.

  We remark that in the above argument, $\Omega^\SO_1=0$ appears as
  the fact that a circle bounds a disk so that we can attach a
  $2$-handle along a circle.

\end{step-named}

\begin{step-named}[Step~3: Reduction to the $0$-skeleton~$K^{(0)}$]
  For the map $\phi_1\colon M_1\to K^{(1)}$, we may assume that the
  center of each $1$-cell of $K^{(1)}$ is a regular value of~$\phi_1$.
  Then $S := \phi_1^{-1}(\{\text{centers of $1$-cells}\})$ is a framed
  $2$-submanifold in~$M$.  Since there is a union of handlebodies, say
  $R$, bounded by $S$, we can do ``surgery'' along~$S$.  More
  precisely, we obtain the trace of surgery by attaching $R\times
  [-1,1]$ to $M_1\times[0,1]$ along $S\times[-1,1] =$ normal bundle
  of~$S$ in $M_1\times 1$.  Performing this for each $1$-cell of
  $K^{(1)}$, we obtain a cobordism $W_1$ from $M_1=M_1\times 0$ to
  another $3$-manifold $M_0$, which is endowed with an induced map
  $W_1\to K^{(1)}$.  Similarly to the above, since the image of $M_0$
  in $K^{(1)}$ under this map is away from the centers of 1-cells, we
  may assume that $M_0$ is mapped to $K^{(0)}$, by a homotopy.

  We remark that in the above argument $\Omega^\SO_2=0$ is used as
  that the $2$-manifold $S$ bounds a $3$-manifold~$R$.

  The following diagram summarizes the above construction:
  \[
  \xymatrix@M=.5em@R=3.5em@C=1.2em{
    &&
    M_3 \ar@{^(->}[r] \ar[d]^(.55){\phi}
    &
    W_3 \ar@{..>}[llld]|(.375)\hole
    &
    M_2 \ar@{_(->}[l] \ar@{^(->}[r] \ar@{->}[d]_-{\phi_2}
    &
    W_2 \ar@{..>}[ld]
    &
    M_1 \ar@{_(->}[l] \ar@{^(->}[r] \ar@{->}[d]_-{\phi_1}
    &
    W_1 \ar@{..>}[ld]
    &
    M_0 \ar@{_(->}[l] \ar@{->}[d]_-{\phi_0}
    \\
    K^{(4)} 
    &&
    K^{(3)} \ar@{_(->}[ll]
    &
    &
    K^{(2)} \ar@{_(->}[ll]
    &
    &
    K^{(1)} \ar@{_(->}[ll]
    &
    &
    K^{(0)} \ar@{_(->}[ll]
  }
  \]
  
\end{step-named}

\begin{remark}
  \label{remark:surgery-along-surface-as-handles}
  The operation of ``surgery along a surface $S$'' in Step~3 above can
  be translated to standard handle attachments as follows.  Let $g_i$
  be the genus of a component $S_i$ of $S = \phi_1^{-1}(\{$centers of
  $1$-cells$\})$, and $R_i$ be a handlebody bounded by~$S_i$.  Viewing
  $R_i$ as a $0$-handle $D^3$ with $g_i$ $1$-handles $D^2_{ij}
  \times[-1,1]$ ($1\le j\le g_i$) attached, and then turning it
  upside-down, we see that attaching $R_i\times[-1,1]$ along
  $S_i\times[-1,1]$ is equivalent to attaching $D^2_{ij}
  \times[-1,1]^2$ along $\partial D^2_{ij}\times[-1,1]^2$ as
  $2$-handles, and then attaching $D^3\times[-1,1]$ along $\partial
  D^3\times[-1,1]$ as a $3$-handle.  It follows that the bordism $W_1$
  in Step~3 above consists of $(g_1+\cdots+g_r)$ $2$-handles and $r$
  $3$-handles, where $r$ is the number of components of~$S$.  This
  observation will be useful in
  Section~\ref{subsection:complexity-of-bordism}.
\end{remark}

\begin{remark}
  \label{remark:why-not-cellular-argument}
  From Steps~1, 2, and~3 above and from
  Remark~\ref{remark:surgery-along-surface-as-handles}, we obtain a
  handle decomposition of the bordism~$W$.  However, the above
  construction which uses CW complexes does not give bounds on the
  number of handles of~$W$.  For instance, regarding $2$-handles, if
  we write $s =$ the number of components of $\phi_2^{-1}(\{$centers
  of $2$-cells$\})$, and if $r$ and the $g_i$ are as in
  Remark~\ref{remark:surgery-along-surface-as-handles}, then our $W$
  has $s+(g_1+\cdots+g_r)$ $2$-handles.  Transversality arguments do
  not provide any control on the number of components $s$ and $r$ and
  the genera $g_i$ of the pre-image; in fact, a homotopy can increase
  $s$, $r$ and $g_i$ arbitrarily.  In order to provide an efficient
  control, we will use a simplicial setup and perform a more
  sophisticated analysis in
  Section~\ref{subsection:complexity-of-bordism}.
\end{remark}

\subsection{Simplicial-cellular approximations of maps to classifying
  spaces}
\label{subsection:basic-simplicial-facts}

In this subsection, we discuss some geometric ideas that arise from
elementary simplicial set theory, for readers not familiar with
simplicial sets.  (We present a short brief review of basic necessary
facts on simplicial sets in the appendix, for the reader's
convenience.)  These will be used in the next subsection, in order to
control the $2$-handle complexity of a bordism~$W$.

We first formally state a generalization of simplicial complexes and
simplicial maps, by extracting geometric properties of simplicial sets
(and their geometric realizations) that we need.

\begin{definition}
  \label{definition:simplicial-cell-complex}
  Let $\Delta^n$ be the standard $n$-simplex.
  \begin{enumerate}
  \item A CW complex $X$ is a \emph{pre-simplicial-cell complex} if
    each $n$-cell is endowed with a characteristic map of the form
    $\Delta^n \to X$.  In particular, an open $n$-cell is identified
    with the interior of $\Delta^n$.  Often we call an $n$-cell an
    \emph{$n$-simplex}.  Note that a simplicial complex is a
    pre-simplicial-cell complex in an obvious way.
  \item A cellular map $X\to Y$ between pre-simplicial-cell complexes
    $X$ and $Y$ is called a \emph{simplicial-cellular map} if its
    restriction on an open $k$-simplex of $X$ is a surjection onto an
    open $\ell$-simplex of $Y$ ($\ell\le k$) which extends to an
    affine surjection $\Delta^k\to \Delta^\ell$ sending vertices to
    vertices.
  \item A pre-simplicial-cell complex $X$ is a \emph{simplicial-cell
      complex} if the attaching map $\partial\Delta^k\to X^{(k-1)}$ of
    every $k$-cell is simplicial-cellular.  Here we view the
    simplicial complex $\partial\Delta^k$ as a pre-simplicial-cell
    complex.
  \end{enumerate}
  By abuse of terminology, we do not distinguish a simplicial-cell
  complex from its underlying space.  Similarly for simplicial and CW
  complexes.
\end{definition}

We note that the composition of simplicial-cellular maps is
simplicial-cellular.

As an example, a simplicial complex is a simplicial-cell complex, and
a simplicial map between simplicial complexes is a simplicial-cellular
map.  More generally, simplicial sets give us simplicial-cell
complexes.  More precisely, a simplicial set has the geometric
realization, which is a CW complex due to Milnor~\cite{Milnor:1957-3};
in fact, his proof shows that the geometric realization is a
simplicial-cell complex in the sense of
Definition~\ref{definition:simplicial-cell-complex}.  See the appendix
(\ref{appendix:simplicial-sets}) for a more detailed
discussion.

The following special case will play a key role in the next
subsection.  It is well known that for a group $G$ a $K(G,1)$ space is
obtained as the geometric realization of the simplicial classifying
space, that is, the nerve of $G$ (for example see
\cite[p.~6]{Goerss-Jardine:1999-1}, \cite[p.~257]{Weibel:1994-1}).
From now on, we denote this $K(G,1)$ space by~$BG$.  By the above,
$BG$ is a simplicial-cell complex.  We remark that $BG$ is not
necessarily a simplicial complex.

\begin{theorem}[Simplicial-cellular approximation of maps to $BG$]
  \label{theorem:simplicial-cellular-approximation}
  Suppose $X$ is the geometric realization of a simplicial set.  Then
  any map $X\to BG$ is homotopic to a simplicial-cellular map.
\end{theorem}

In this paper, we will apply
Theorem~\ref{theorem:simplicial-cellular-approximation} to a
simplicial complex~$X$; we note that a simplicial complex gives rise
to a simplicial set (by ordering the vertices).

Since the author did not find it in the literature, a proof of
Theorem~\ref{theorem:simplicial-cellular-approximation} is given in
the appendix; see
Proposition~\ref{proposition:simplicial-realization}.

\begin{remark}
  Theorem~\ref{theorem:simplicial-cellular-approximation} may be
  compared with the standard simplicial and cellular approximation
  theorems.  The simplicial approximation respects the simplicial
  structure but requires a subdivision of the domain.  On the other
  hand, the cellular approximation does not require a subdivision but
  does not respect simplicial structures.
  Theorem~\ref{theorem:simplicial-cellular-approximation} respects the
  simplicial structures and requires no subdivision.  The latter is an
  important feature too, since controlling the number of simplicies is
  essential for our purpose.
\end{remark}

\subsection{Estimating the $2$-handle complexity}
\label{subsection:complexity-of-bordism}

In this subsection we present a simplicial refinement of the
transversality-and-surgery arguments used in
Section~\ref{subsection:geometric-construction-of-null-cobordism}, and
find an upper bound of the $2$-handle complexity of the resulting
bordism.

We define the \emph{complexity of a triangulated $3$-manifold} to be
the number of $3$-simplices.  (Note that this is different from the
notion of the (simplicial) complexity of a $3$-manifold.)  Recall from
the introduction that the $2$-handle complexity of a $4$-dimensional
bordism $W$ is the minimal number of $2$-handles in a handle
decomposition of~$W$.

For a triangulated closed $3$-manifold $M$, let $\zeta_M$ be the sum
of oriented $3$-simplices of $M$ which represents the fundamental
class, as we did for a CW complex structure.  Recall that the diameter
$d(\zeta_M)$ is equal to the complexity of the triangulation.

The main result of this subsection is the following.  

\begin{theorem}
  \label{theorem:existence-of-efficient-bordism}
  Suppose $M$ is a closed triangulated $3$-manifold with
  complexity~$d(\zeta_M)$.  Suppose $M$ is over a simplicial-cellular
  complex $K$ via a simplicial-cellular map $\phi\colon M \to K$.  If
  there is a $4$-chain $u\in C_4(K)$ satisfying $\partial u =
  \phi_{\#}(\zeta_M)$, then there exists a smooth bordism $W$ between
  $M$ and a trivial end whose $2$-handle complexity is at most $195\cdot
  d(\zeta_M) + 975\cdot d(u)$.
\end{theorem}

We remark that when $K=B\Gamma$, any map $\phi\colon M\to K$ may be
assumed to be a simplicial-cellular map up to homotopy, by
Theorem~\ref{theorem:simplicial-cellular-approximation}.

Recall that in
Section~\ref{subsection:geometric-construction-of-null-cobordism} we
constructed a bordism $W$ between $M$ and a trivial end by stacking
bordisms $W_3$, $W_2$, and $W_1$ such that $\partial W_i = M_i \sqcup
-M_{i-1}$ over $K$, where $M_3:=M$ is the given $3$-manifold, and $M_i$
is over $K$ via a map $\phi_i\colon M_i \to K^{(i)}$ to the
$i$-skeleton for each~$i$.  The main strategy of our proof of
Theorem~\ref{theorem:existence-of-efficient-bordism} is to refine the
construction of the $W_i$ carefully to control the number of
$2$-handles.  For this purpose, we will triangulate $M_i$ and make
$\phi_i$ simplicial-cellular.  For the initial case, $M_3=M$ is
triangulated and $\phi_3 = \phi$ is simplicial-cellular by the
hypothesis of Theorem~\ref{theorem:existence-of-efficient-bordism}.
Arguments for $W_i$ and $M_{i-1}$ for $i=3, 2, 1$ are given as the
three propositions below.

\begin{proposition}[Step~1: Reduction to $K^{(2)}$ and complexity
  estimate]
  \label{proposition:reduction-to-2-skeleton}
  Suppose $M$, $\phi$, $u$ are as in
  Theorem~\ref{theorem:existence-of-efficient-bordism}.  Then there is
  a triangulated $3$-manifold $M_2$ with complexity at most
  $n_2:=18\cdot d(\zeta_M)+90\cdot d(u)$, which is over $K$ via a
  simplicial-cellular map $\phi_2\colon M_2\to K^{(2)}$, and there is
  a bordism $W_3$ over $K$ between $M$ and $M_2$ which has no
  $2$-handles.
\end{proposition}

\begin{proof}
  Following Step~1 in
  Section~\ref{subsection:geometric-construction-of-null-cobordism},
  we write $u=-\sum_\alpha n^{\vphantom{4}}_\alpha \sigma^4_\alpha$,
  where the $\sigma^4_\alpha$ are $4$-simplices of $K$ with attaching
  maps $\phi_\alpha\colon \partial\Delta^4_\alpha \to K^{(3)}$.  Here
  $\Delta^4_\alpha$ is a standard $4$-simplex.  Let $X :=
  (M\times[0,1]) \sqcup (\bigsqcup_\alpha n^\nstrut_\alpha
  \Delta^4_\alpha)$.  The $4$-manifold $X$ is a bordism over $K$
  between $M = M\times 0$ and $M' := (M\times 1) \sqcup
  (\bigsqcup_\alpha n^\nstrut_\alpha \partial\Delta^4_\alpha)$, via
  the map $X\to K$ induced by $\phi$ and the~$\phi_\alpha$.  Let
  $\psi\colon M'\to K$ be the restriction.  The $3$-manifold $M'$ is
  triangulated using the given triangulation of $M$ and the standard
  triangulation of~$\partial\Delta^4_\alpha$.  The map $\psi$ is
  simplicial-cellular since $\phi$ and the $\phi_\alpha$ are
  simplicial-cellular.  From the relation $\phi_{\#}(\zeta_M)-\partial
  u = 0$, it follows that the $3$-simplices of $M'$ whose image under
  $\psi$ is nonzero in $C_3(K)$ are canceled in pairs in the image
  under~$\psi$.  For each canceling pair of $3$-simplices of $M'$, we
  attach a $1$-handle to $X$ which joins their barycenters.  To do it
  simplicially, we subdivide relevant $3$-simplices as follows.

  Recall that the product $\Delta^2\times [0,1]$ is triangulated by a
  prism decomposition; see Figure~\ref{figure:prism-decomposition}.
  More precisely, ordering vertices of $\Delta^2$ as $\{u_0,u_1,u_2\}$
  and vertices of $[0,1]$ as $\{w_0,w_1\}$ and letting
  $v_{ij}=(u_i,w_j)\in \Delta^2\times[0,1]$, the standard prism
  decomposition has $3$-simplices $[v_{00},v_{10},v_{20},v_{21}]$,
  $[v_{00},v_{10},v_{11},v_{21}]$, and
  $[v_{00},v_{01},v_{11},v_{21}]$.  We note that we obtain several
  different prism decompositions by reordering vertices of $\Delta^2$
  and $[0,1]$.

  \begin{figure}[H]
    \labellist\footnotesize
    \pinlabel{$v_{00}$} at -10 48
    \pinlabel{$v_{10}$} at 75 9
    \pinlabel{$v_{20}$} at 108 58
    \pinlabel{$v_{01}$} at -10 138
    \pinlabel{$v_{11}$} at 75 101
    \pinlabel{$v_{21}$} at 108 148
    \endlabellist
    \includegraphics[scale=.65]{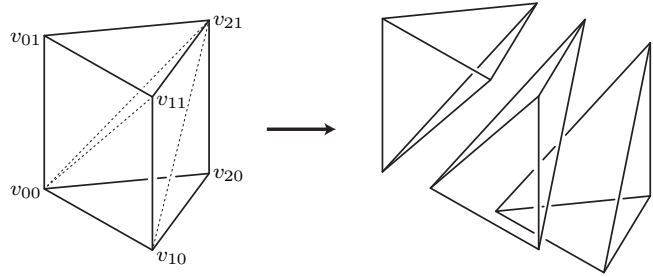}
    \caption{The standard prism decomposition of $\Delta^2\times[0,1]$.}
    \label{figure:prism-decomposition}
  \end{figure}

  \begin{figure}[H]
    \labellist
    \small
    \pinlabel{\begin{tabular}{@{}c@{}}$4\cdot
        3+1=13$\\$3$-simplices\end{tabular}} at 450 85
    \endlabellist
    \includegraphics[scale=.7]{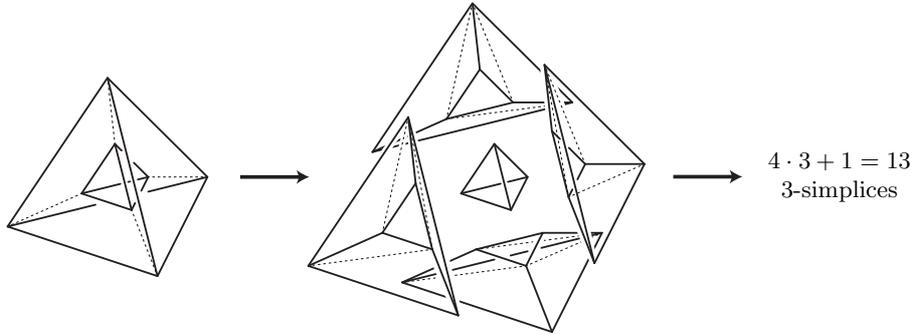}
    \caption{A subdivision of a $3$-simplex for 1-handle attachment.}
    \label{figure:subdiv-for-1-handle}
  \end{figure}

  Take a $3$-simplex $\Delta'$ embedded in the interior of a standard
  $3$-simplex $\Delta^3$, and subdivide $\partial \Delta^3 \times
  [0,1]\cong \Delta^3\setminus \inte\Delta'$ by taking a prism
  triangulation of $\tau\times [0,1]$ for each face $\tau$
  of~$\Delta^3$.  As in Figure~\ref{figure:subdiv-for-1-handle}, one
  can choose prime decompositions appropriately in such a way that
  they agree on the intersections.  This gives us a subdivision of
  $\Delta^3$, which contains $\Delta'$ as a simplex.  We call
  $\Delta'$ the \emph{inner subsimplex} of this subdivision.  We apply
  this subdivision to each $3$-simplex of $M'$ whose image under
  $\psi$ is nonzero in $C_3(K)$, and then attach $1$-handles
  $\Delta^3\times[0,1]$ to $X$ by identifying $\Delta^3\times 0$ and
  $\Delta^3\times 1$ with inner subsimplices of a canceling pair of
  $3$-simplices.  This gives a cobordism $W_3$ between $M=M_3$ and a
  new $3$-manifold $M_2$ obtained from $M'$ by surgery.  By
  triangulating the belt tube $\partial\Delta^3 \times [0,1]$ of each
  $1$-handle using a prism decomposition of $($each face of
  $\Delta^3)\times[0,1]$, and by combining it with the subdivision on
  $M'$, we obtain a triangulation of~$M_2$.

  We want to show that there is a simplicial-cellular map $\phi_2
  \colon M_2\to K^{(2)}$ such that $\phi_3\sqcup \phi_2\colon
  M_3\sqcup M_2 \to K$ extends to~$W_3$.  To do this explicitly, first
  observe that there is a map $\Delta^3\to \Delta^3$ which is (i)
  simplicial with respect to the subdivision in
  Figure~\ref{figure:subdiv-for-1-handle}, (ii) collapses the collar
  $\Delta^3-\inte\Delta'$ onto $\partial\Delta^3$, (iii) stretches the
  inner subsimplex onto $\Delta^3$, and (iv) is homotopic to the
  identity rel $\partial\Delta^3$.  Composing it with the map
  $\psi\colon M'\to K$ on each subdivided $3$-simplex on $M$, we
  obtain a simplicial-cellular map $\psi'\colon M'\to K$ with respect
  to the subdivision.  Note that $\psi'$ is homotopic to $\psi$.  Thus
  we may assume that the 4-manifold $X$ is over $K$ via a map $X\to K$
  that restricts to $\psi'$ on~$M'$.  Then $X\to K$ extends to the
  $1$-handles, and induces a map $W_3\to K$, since the restrictions of
  $\psi'$ on two inner subsimplices joined by a $1$-handle are the
  same.  Let $\phi_2\colon M_2 \to K$ be the restriction.  Since
  $\psi'$ is simplicial-cellular, $\phi_2$ is simplicial-cellular.
  Since $\psi'$ sends $M'\setminus \bigsqcup$(inner simplices) to
  $K^{(2)}$, it follows that $\phi_2$ sends $M_2$ to~$K^{(2)}$.  This
  completes the construction of the desired $W_3$, $M_2$ and
  $\phi_2\colon M_2\to K^{(2)}$.

  Now we estimate the complexity of the triangulation of~$M_2$.  Let
  $n=d(\zeta_M)$, the complexity of the given triangulation of~$M$.
  Since $u$ has diameter $d(u)=\sum_\alpha |n_\alpha|$, the initial
  triangulation of $M' = (M\times 1)\sqcup(\bigsqcup_\alpha
  n^\nstrut_\alpha \partial\Delta^4_\alpha)$ has complexity $n+5d(u)$.
  Since our subdivision in Figure~\ref{figure:subdiv-for-1-handle}
  produces 13 $3$-simplices from one $3$-simplex, the complexity of
  the new subdivision of $M'$ is at most $13(n+5d(u))$.  The number of
  $1$-handles attached is at most $(n+5d(u))/2$, and each $1$-handle
  attachment removes two $3$-simplices (inner subsimplices) and adds
  $4\cdot 3=12$ $3$-simplices (those in the belt tube).  Therefore, as
  claimed, the complexity of the triangulation of $M_2$ is at most
  \[
  n_2 := 12(n+5d(u))+12\cdot\frac{n+5d(u)}{2} = 18n+90d(u).
  \]
  From our construction, it is obvious that $W$ has no $2$-handles.
\end{proof}

\begin{proposition}[Step~2: Reduction to $K^{(1)}$ and
  complexity estimate]
  \label{proposition:reduction-to-1-skeleton}
  Suppose $M_2$ is a closed triangulated $3$-manifold with complexity
  $n_2$, which is over $K$ via a simplicial-cellular map $\phi_2\colon
  M_2 \to K^{(2)}$.  Then there is another triangulated $3$-manifold
  $M_1$ with complexity at most $n_1 := 21n_2$, which is over $K$ via
  a simplicial-cellular map $\phi_1\colon M_1 \to K^{(1)}$, and there
  is a bordism $W_2$ over $K$ between $M_2$ and $M_1$ with $2$-handle
  complexity at most~$\lfloor n_2/3 \rfloor$.
\end{proposition}

\begin{proof}
  To obtain $W_2$, we will attach $2$-handles to $M_2\times[0,1]$
  along the inverse image of the barycenter of each $2$-simplex of $K$
  under~$\phi_2$, similarly to Step~2 of
  Section~\ref{subsection:geometric-construction-of-null-cobordism}.
  Fix a $2$-simplex of $K$ and denote its barycenter by~$b$.  If the
  interior of a $3$-simplex of $M_2$ meets $\phi_2^{-1}(b)$, then
  since $\phi_2$ is a simplicial-cellular map, it follows that
  $\phi_2$ on the $3$-simplex is an affine projection $\Delta^3\to
  \Delta^2$ onto the $2$-simplex sending vertices to vertices; see
  Figure~\ref{figure:proj-to-2-simplex}, which illustrates the case
  $[0,1,2,3] \mapsto [0,1,2,2]$.
  Figure~\ref{figure:proj-to-2-simplex} also shows the pre-image
  $\phi_2^{-1}(b)$ in the $3$-simplex.

  \begin{figure}[t]
    \labellist
    \footnotesize
    \pinlabel{$\phi_2$} at 160 73
    \pinlabel{$b$} at 230 66
    \pinlabel{$\phi_2^{-1}(b)$} at 47 66
    \pinlabel{0} at 55 128
    \pinlabel{1} at 82 0
    \pinlabel{2} at 132 40
    \pinlabel{3} at -5 43
    \pinlabel{0} at 218 128
    \pinlabel{1} at 193 0
    \pinlabel{2} at 274 40
    \endlabellist
    \includegraphics[scale=.65]{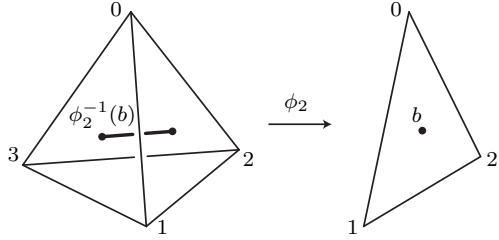}
    \caption{A simplicial projection $\Delta^3 \to \Delta^2$.}
    \label{figure:proj-to-2-simplex}
  \end{figure}

  We take a sufficiently thin tubular neighborhood $U\cong
  \phi_2^{-1}(b)\times \Delta^2$ of $\phi_2^{-1}(b)$ in $M_2$ in such
  a way that the intersection of $U$ and a $3$-simplex of $M_2$ is a
  triangular prism or empty.  We triangulate the exterior $M_2
  \setminus \inte(U)$ by subdividing each 3-simplex with a triangular
  prism removed as in Figure~\ref{figure:subdiv-for-2-handle}; we
  first decompose it into one $3$-simplex, one triangular prism, and
  $4$ quadrangular pyramids, and then divide the triangular prism and
  quadrangular pyramids along the dashed lines to obtain a subdivision
  with $1+3+4\cdot 2 = 12$ $3$-simplices.  Since the subdivision of
  the two front faces of the original $3$-simplex shown in the left of
  Figure~\ref{figure:subdiv-for-2-handle} are identical and the two
  back faces are not subdivided, our subdivisions agree on the
  intersection of any two such $3$-simplices.  Observe that
  $\partial(M_2\setminus \inte(U))=\partial U$ meets a $3$-simplex of
  $M_2$ in three squares forming a cylinder as in
  Figure~\ref{figure:subdiv-for-2-handle}, where each square has been
  triangulated into two $2$-simplices.  For later use, we note that we
  can alter the triangulation of these squares by changing the
  subdivisions of the quadrangular pyramids and the triangular prism
  in Figure~\ref{figure:subdiv-for-2-handle}.

  \begin{figure}[H]
    \labellist
    \small
    \endlabellist
    \includegraphics[scale=.7]{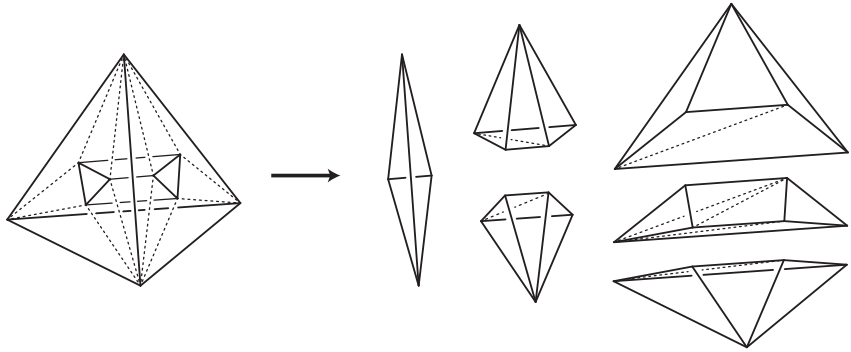}
    \caption{A subdivision of a $3$-simplex with a triangular prism
      removed.}
    \label{figure:subdiv-for-2-handle}
  \end{figure}

  Now we consider $2$-handle attachments.  The pre-image
  $\phi_2^{-1}(b)\subset M_2$ is a disjoint union of piecewise linear
  circles.  Suppose $C$ is a circle component of $\phi_2^{-1}(b)$.
  Let $r$ be the number of $3$-simplices of $M_2$ which $C$ passes
  through as in the local picture shown in
  Figure~\ref{figure:proj-to-2-simplex}, that is, $C$ is an $r$-gon.
  Take a $2$-handle $D\times \Delta^2$, where $D$ is a $2$-disk.
  Triangulate $D$ into $r$ triangles by drawing $r$ line segments from
  the center to the perimeter, and then triangulate $D\times ($each
  face of $\Delta^2) \cong D\times [0,1]$ by ordering the
  $0$-simplices of $D$ and then taking the prism decomposition of
  (each $2$-simplex of $D)\times [0,1]$.  Gluing these, we obtain a
  triangulation of the belt tube $D\times \partial\Delta^2$ of the
  $2$-handle.  We attach the $2$-handle $D\times \Delta^2$ to
  $M_2\times[0,1]$ by identifying the neighborhood $C\times \Delta^2
  \subset M_2 = M_2\times 1$ with the attaching tube $\partial D\times
  \Delta^2$.  We may assume that the triangulation of $\partial
  D\times \partial\Delta^2$ agrees with that of $\partial (M\setminus
  \inte(U))$, by altering the latter as mentioned above if necessary.
  We note that our triangulation of the belt tube of this $2$-handle
  has $3\cdot 3r=9r$ $3$-simplices.

  Attaching $2$-handles for each $2$-simplex of $K$ in this way, we
  obtain a cobordism $W_2$ between $M_2$ and another $3$-manifold~$M_1$,
  together with a triangulation of~$M_1$.

  We make $W_2$ a bordism over $K$ similarly to Step~1 above: observe
  that there is a piecewise linear endomorphism of the $3$-simplex
  $\Delta^3$ shown in the left of
  Figure~\ref{figure:subdiv-for-2-handle} which restricts to a
  simplicial-cellular map of the exterior $\Delta^3\setminus\inte(U)$
  onto $A:=\partial\Delta^3\setminus ($interior of the two faces of
  $\Delta^3$ meeting $\phi^{-1}(b))$, and is homotopic to the identity
  rel~$A$.  From this it follows that the map $\phi_2\colon M_2\to
  K^{(2)}$ is homotopic to a map, which restricts to a
  simplicial-cellular map $M_2\setminus\inte(U) \to K^{(1)}$ and
  extends to $W_2\to K^{(2)}$.  Also, $W_2\to K^{(2)}$ restricts to a
  simplicial-cellular map $\phi_1\colon M_1\to K^{(1)}$.  In
  particular $W_2$ is a bordism over $K$ between $(M_2,\phi_2)$ and
  $(M_1,\phi_1)$.

  Now we estimate the complexity of~$M_1$.  Recall the hypothesis that
  $M_2$ has $n_2$ $3$-simplices.  Our subdivision of $M\setminus
  \inte(U)$ has at most $12n_2$ $3$-simplices, since each $3$-simplex that
  meets an attaching circle contributes $12$ $3$-simplices as observed
  above (see Figure~\ref{figure:subdiv-for-2-handle}).  Suppose we
  attach $s$ $2$-handles and the $i$th $2$-handle is attached along an
  $r_i$-gon.  As observed above, the belt tube of the $i$th $2$-handle
  has $9r_i$ $3$-simplices.  Therefore our triangulation of $M_1$ has
  complexity at most $12n_2 + 9(r_1+\cdots+r_s)$.  Since each
  $3$-simplex of $M_2$ can contribute at most one line segment to the
  attaching circles, we have $r_1+\cdots+r_s \le n_2$.  It follows
  that $M_2$ has complexity at most~$21n_2$.  Since $r_i \ge 3$, we
  also obtain that $3s\le n_2$ as claimed.
\end{proof}

\begin{proposition}[Step~3: Reduction to $K^{(0)}$
    and complexity estimate]
    \label{proposition:reduction-to-0-skeleton}

    Suppose $M_1$ is a closed triangulated $3$-manifold with
    complexity $n_1$, which is over $K$ via a simplicial-cellular map
    $\phi_1\colon M_1\to K^{(1)}$.  Then there is another $3$-manifold
    $M_0$ which is over $K$ via a map $\phi_0\colon M_0 \to K^{(0)}$
    and there is a bordism $W_1$ over $K$ between $M_1$ and $M_0$
    whose $2$-handle complexity is at most $\lfloor n_1/2 \rfloor$.
  \end{proposition}
  
\begin{proof}
  We construct the bordism $W_1$ similarly to Step~3 of
  Section~\ref{subsection:geometric-construction-of-null-cobordism},
  namely by attaching $R_i\times [0,1]$ to $M_1\times[0,1]$, where
  $R_i$ is a handlebody bounded by a component $S_i$ of the pre-image
  of the barycenter of a $1$-simplex of $K$ under~$\phi_1$.  Recall from
  Remark~\ref{remark:why-not-cellular-argument} that if $S_i$ has
  genus $g_i$, then attaching $R_i$ is equivalent to attaching $g_i$
  $2$-handles and one $3$-handle.

  Since $\phi_1$ is simplicial-cellular, the pre-image
  $\phi_1^{-1}(b)$ of a barycenter $b$ of a $1$-simplex of $K$
  intersects a $3$-simplex $\Delta^3$ of $M_1$ as shown in
  Figure~\ref{label:proj-to-1-simplex}; we have two possibilities,
  where $\phi^{-1}(b) \cap \Delta^3$ is either a triangle or a
  quadrangle.  By dividing each quadrangle in $\phi^{-1}(b)$ into two
  triangles, we obtain a triangulation of the $2$-manifold
  $\phi_1^{-1}(b)$.  Since $M_1$ has $n_1$ $3$-simplices and each
  $3$-simplex can contribute at most two triangles to $\phi^{-1}(b)$,
  it follows that the $2$-manifold $\bigsqcup_i S_i$ has a
  triangulation with at most $2n_1$ $2$-simplices.

  \begin{figure}[H]
    \labellist
    \footnotesize
    \pinlabel{0} at 54 125
    \pinlabel{1} at -4 30
    \pinlabel{2} at 80 -2
    \pinlabel{3} at 131 40
    \pinlabel{0} at 214 125
    \pinlabel{1} at 214 -2
    \pinlabel{$b$} at 214 62
    \normalsize\pinlabel{$[0,1,2,3]\longmapsto [0,1,1,1]$} at 115 -28
    \footnotesize
    \pinlabel{0} at 325 100
    \pinlabel{1} at 409 125
    \pinlabel{2} at 320 28
    \pinlabel{3} at 423 -2
    \pinlabel{0} at 504 125
    \pinlabel{1} at 504 -2
    \pinlabel{$b$} at 504 62
    \normalsize\pinlabel{$[0,1,2,3]\longmapsto [0,0,1,1]$} at 420 -28
    \endlabellist
    \vspace*{1ex}
    \includegraphics[scale=.65]{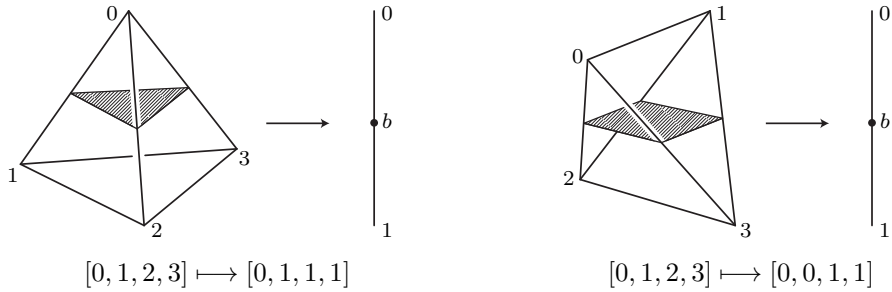}
    \vspace*{4ex}
    \caption{Simplicial projections $\Delta^3 \to \Delta^1$.}
    \label{label:proj-to-1-simplex}
  \end{figure}

  To estimate the genera, we invoke the following observation:

  \begin{lemma}
    \label{lemma:surface-genus-and-complexity}
    A connected closed surface admitting a triangulation with $n$
    $2$-simplices has genus at most~$\lfloor \frac{n-2}{4} \rfloor$.
  \end{lemma}

  \begin{proof}
    Since there are $\frac{3n}{2}$ $1$-simplices, the Euler
    characteristic $2-2g$ is equal to $n-\frac{3n}{2}+v$, where $v$ is
    the number of $0$-simplices.  Since $v\ge 3$, it follows that $g \le
    \frac{n-2}{4}$.
  \end{proof}

  Returning to the proof of
  Proposition~\ref{proposition:reduction-to-0-skeleton}, suppose the
  inverse image of the union of the barycenters of $1$-simplices of
  $K$ under $\phi_1$ has $r$ components $S_1,\ldots,S_r$, and suppose
  $S_i$ has $m_i$ $2$-simplices in its triangulation.  By
  Lemma~\ref{lemma:surface-genus-and-complexity}, the genus $g_i$ of
  $S_i$ is at most $m_i/4$.  Since $m_1+\cdots+m_r\le 2n_1$, it
  follows that $g_1+\cdots+g_r \le n_1/2$.  Therefore, the $2$-handle
  complexity of $W_1$ is at most $n_1/2$ as claimed.
\end{proof}

Now we combine the above three propositions to give a proof of
Theorem~\ref{theorem:existence-of-efficient-bordism}.

\begin{proof}[Proof of Theorem~\ref{theorem:existence-of-efficient-bordism}]
  Let $M_3=M$ and $\phi_3=\phi$, and apply
  Propositions~\ref{proposition:reduction-to-2-skeleton},
  \ref{proposition:reduction-to-1-skeleton},
  and~\ref{proposition:reduction-to-0-skeleton} to obtain bordisms
  $W_3$, $W_2$, and~$W_1$ together with $(M_2,\phi_2)$,
  $(M_1,\phi_1)$, and $(M_0,\phi_0)$.  Concatenating $W_3$, $W_2$, and
  $W_1$, we obtain a bordism $W$ over $K$ between $M$ and $N:=M_0$.
  Since $\phi_0$ is to $K^{(0)}$, $\phi_0$ is homotopic to a
  constant map, and so we may assume that $N$ is trivially over $K$.
  By Propositions~\ref{proposition:reduction-to-2-skeleton},
  \ref{proposition:reduction-to-1-skeleton},
  and~\ref{proposition:reduction-to-0-skeleton}, $M_2$ and $M_1$ have
  complexity at most $n_2:=18n+90d(u)$ and $n_1 := 21n_2 =
  378n+1890d(u)$, respectively.  Also, $W_3$ has no $2$-handles, $W_2$
  has at most $n_2/3 = 6n+30d(u)$ $2$-handles, and $W_1$ has at most
  $n_1/2 = 189n+945d(u)$ $2$-handles.  It follows that the $2$-handle
  complexity of $W$ is not greater than
  \[
  6n+30d(u)+189n+945d(u) = 195n+975d(u). \qedhere
  \]
\end{proof}

In light of Theorem~\ref{theorem:existence-of-efficient-bordism},
finding a 4-chain $u$ with controlled diameter $d(u)$ is essential in
constructing an efficient 4-dimensional bordism to a trivial end.
This will be done by using the results developed in the next section.

\section{Controlled chain homotopy}
\label{section:controlled-chain-homotopy}

In this section we develop some useful results on controlled chain
homotopy.  We recall basic definitions from the introduction.  In this
paper we assume that chain complexes are always positive.  We also
assume that chain complexes are over $\Z$, although everything holds
over a ring $R$ endowed with a norm~$|\cdot|$.  The \emph{diameter}
$d(u)$ of a chain $u$ in a based chain complex is defined to be its
$L^1$-norm, that is, if $u=\sum_\alpha n_\alpha e_\alpha$ where
$\{e_\alpha\}$ is the given basis, then $d(u)=\sum_\alpha |n_\alpha|$.
For a chain homotopy $P\colon C_*\to D_{*+1}$ between based chain
complexes $C_*$ and $D_*$, the \emph{diameter function} $d_P$ of $P$
is defined by
\[
d_P(k):=\max\{d(P(c))\mid \text{$c\in C_i$ is a basis element, $i\le k$}\}.
\]
If $P$ is a partial chain homotopy which is defined on $C_i$ for $i\le
N$ only, then $d_P(k)$ is defined for $k\le N$.  Note that $d_P(k)$
may not be finite if $\bigoplus_{i\le k}C_i$ is not finitely
generated.

For a function $\delta$ from the domain of $d_P$ to $\Z_{\ge 0}$, we
say that $P$ is a \emph{$\delta$-controlled \textup(partial\textup)
  chain homotopy} if $d_P(k) \le \delta(k)$ for each~$k$.

Similarly to the chain homotopy case, the
\emph{diameter function} $d_\phi(k)$ of a chain map $\phi\colon C_*\to
D_*$ is defined by
\[
d_\phi(k) = \max\{d(\phi(u))\mid \text{$u\in C_i$ is a basis element,
  $i\le k$}\}.
\]
We say that a chain map $f\colon C_*\to D_*$ between based chain
complexes $C_*$ and $D_*$ is \emph{based} if $f$ takes a basis element
to a basis element.  A based chain map $\phi$ has $d_\phi(k)= 1$.

For a chain homotopy or a chain map $P$, $d(P(z)) \le d_P(k)\cdot
d(z)$ for any chain $z$ of dimension at most~$k$.  We state a few more
basic facts for later use:

\begin{lemma}
  \leavevmode\Nopagebreak
  \label{lemma:basic-properties-of-diameter}
  \begin{enumerate}
  \item\label{lemma-item:diameter-sum} \textup{(Sum)} If $P\colon
    \phi\simeq\psi$ and $Q\colon \zeta\simeq\xi$ for $\phi$, $\psi$,
    $\zeta$, $\xi\colon C_*\to D_*$, then $P+Q\colon \phi+\zeta \simeq
    \psi+\xi$ and $d_{P+Q}(k) \le d_P(k)+d_Q(k)$.
  \item\label{lemma-item:diameter-composition}
    \textup{(Composition)} If $P\colon \phi\simeq\psi$ and $Q\colon
    \zeta\simeq\xi$ for chain maps $\phi$, $\psi\colon C_*\to D_*$ and
    $\zeta$, $\xi\colon D_*\to E_*$, then $\zeta P+ Q\psi\colon
    \zeta\phi \simeq \xi\psi$ and $d_{\zeta P+Q\psi}(k) \le
    d_\zeta(k)\cdot d_P(k)+d_Q(k)\cdot d_\psi(k)$.
  \item\label{lemma-item:diameter-tensor-product} \textup{(Tensor
      product)} If $P\colon \phi\simeq\psi$ and $Q\colon
    \zeta\simeq\xi$ for chain maps $\phi$, $\psi\colon C_*\to D_*$ and
    $\zeta$, $\xi\colon C'_*\to D'_*$, then
    \[
    \Phi(\sigma\otimes\tau):=(P\otimes \zeta +
    (-1)^{|\sigma|}\psi\otimes Q)(\sigma\otimes\tau)
    \]
    is a chain homotopy $\Phi\colon \phi\otimes\zeta \simeq
    \psi\otimes\xi$, and $d_{\Phi}(k) \le d_P(k)\cdot d_\zeta(k) +
    d_\psi(k)\cdot d_Q(k)$.
    
  \end{enumerate}
  The analogs for partial chain homotopies hold too.
\end{lemma}

The proof of Lemma~\ref{lemma:basic-properties-of-diameter} is
straightforward.  We omit details.

From Definition~\ref{definition:controlled-chain-homotopy} in the
introduction, we recall the notion of a uniformly control family of
chain homotopies: suppose $\mathcal{S} = \{P_A \colon C^A_* \to
D^A_{*+1}\}_{A\in \mathcal{I}}$ is a collection of chain homotopies or
a collection of partial chain homotopies defined in dimensions $\le n$
for some fixed~$n$.  We say that $\mathcal{S}$ is \emph{uniformly
  controlled by $\delta$} if each $P_A$ is a $\delta$-controlled chain
homotopy.

In many cases a family of chain homotopies comes with functoriality,
in the following sense.  Let $\Chp$ be the category of positive chain
complexes over~$\Z$; morphisms are degree zero chain maps as usual.
Suppose $\mathbf{C}$ is a category, $F$, $G\colon \mathbf{C}\to
\Chp$ are functors, and $\phi$, $\psi\colon F\to G$ are natural
transformations, that is, for each $A\in \mathbf{C}$ we have chain
complexes $F(A)$, $G(A)$ and chain maps $\phi_A$, $\psi_A\colon
F(A)\to G(A)$ which are functorial in~$A$.  We say that $\{P_A\colon
\phi_A\simeq \psi_A\}_{A\in \mathbf{C}}$ is a \emph{family of natural
  chain homotopies between $\phi$ and $\psi$} if $P_A\colon F(A)_*\to
G(A)_{*+1}$ is functorial in $A$ and $P_A\partial + \partial P_A =
\psi_A-\phi_A$ for each $A\in\mathbf{C}$.  The partial chain homotopy
analog is defined similarly.

We denote by $\Chpb$ the category of positive based chain complexes
and (not necessarily based) chain maps.  The above paragraph applies
to $\Chpb$ similarly.

\subsection{Controlled acyclic model theorem}
\label{subsection:controlled-acyclic-model-theorem}

Our first source of a uniformly controlled family of natural chain
homotopies is the classical acyclic model theorem of Eilenberg and
MacLane~\cite{Eilenberg-MacLane:1953-1}.

We recall two basic definitions used to state the standard acyclic
model theorem.  We say that $F\colon \mathbf{C} \to \Chp$ (or $\Chpb$)
is \emph{acyclic} with respect to a collection $\mathcal{M}$ of
objects in $\mathbf{C}$ if the chain complex $F(A)$ is acyclic for
each $A$ in~$\mathcal{M}$.  Also, we say that $F$ is \emph{free} with
respect to $\mathcal{M}$ if for each $i$ there is a collection
$\mathcal{M}_i = \{(A_\lambda,c_\lambda)\}_\lambda$ with $A_\lambda\in
\mathcal{M}$ and $c_\lambda\in F(A_\lambda)_i$ such that for any
object $B$ in $\mathbf{C}$, $F(B)_i$ is a free abelian group and the
elements $F(f)(c_\lambda)\in F(B)_i$ for $f\in \Mor(A_\lambda,B)$ are
distinct and form a basis.  We define analogs for based chain
complexes:

\begin{definition}
  \begin{enumerate}
  \item A functor $F\colon \mathbf{C} \to \Chpb$ is \emph{based} if
    for any $f\in \Mor_{\mathbf{C}}(A,B)$, $F(f) \in
    \Mor_{\Chpb}(F(A),F(B))$ is a based chain map.  Also, $F$ is
    \emph{based-acyclic} if $F$ is based and acyclic.
  \item A functor $F\colon \mathbf{C} \to \Chpb$ is \emph{based-free}
    with respect to $\mathcal{M}$ if for each $i$ there is a
    collection $\mathcal{M}_i = \{(A_\lambda,c_\lambda)\}_\lambda$
    with $A_\lambda\in \mathcal{M}$ and $c_\lambda\in F(A_\lambda)_i$
    such that for any $A\in\mathbf{C}$, the elements
    $F(f)(c_\lambda)\in F(A)_i$ for $f\in \Mor(A_\lambda,A)$ are
    distinct and form the preferred basis of the based free abelian
    group~$F(A)_i$.  In addition, if $\mathcal{M}_i$ is finite for
    each $i$, then we say that $F$ is \emph{finitely based-free}.
  \end{enumerate}
\end{definition}

Observe that $F$ is automatically based if $F$ is based-free.

\begin{theorem}[Controlled acyclic model theorem]
  \label{theorem:acyclic-model-theorem-with-control}
  Suppose $F$, $G\colon \mathbf{C} \to \Chpb$ are functors, $F$ is
  finitely based-free with respect to~$\mathcal{M}$, and $G$ is
  based-acyclic with respect to~$\mathcal{M}$.  Then the following
  hold.
  \begin{enumerate}
  \item Any natural transformation $\phi_0\colon H_0\circ F \to
    H_0\circ G$ extends to a natural transformation $\phi\colon F\to
    G$.
  \item Suppose $\phi$, $\psi\colon F\to G$ are natural
    transformations that induce the same transformation $H_0\circ F
    \to H_0\circ G$.  Then there exist a function $\delta\colon \Z\to
    \Z_{\ge0}$ and a family of natural chain homotopies $\{P_A\colon
    \phi_A\simeq\psi_A\}$ which is uniformly controlled by~$\delta$.
  \end{enumerate}
\end{theorem}

The key is that that even when the rank of the chain complexes is
unbounded, we have a uniform control $\delta$ if there are only
finitely many models in each dimension.

\begin{proof}
  Recall that (1) is a conclusion of a standard acyclic model
  argument.

  For (2), recall the construction of a family of chain homotopies
  \[
  P_A = \{(P_{A})_i\colon F(A)_{i-1} \to G(A)_i\}, \quad  A\in
  \mathbf{C}
  \]
  from the standard acyclic model argument: assume $(P_A)_{i-1}$ has
  been defined.  Using that $G(A_\lambda)$ is acyclic for each
  $(A_\lambda,c_\lambda) \in \mathcal{M}_i$, we obtain a chain, which
  we denote by $(P_{A_\lambda})_i(c_\lambda)\in G(A_\lambda)_{i+1}$ as
  abuse of notation for now, that makes the equation
  $P_{A_\lambda}\partial +\partial P_{A_\lambda} =
  \psi_{A_\lambda}-\phi_{A_\lambda}$ satisfied at $c_\lambda\in
  F(A_\lambda)_i$; then for an arbitrary $A\in\mathbf{C}$, using that
  $F$ is free, we define $(P_A)_i$ on a basis element by
  $(P_A)_i(F(f)(c_\lambda)) := G(f)((P_{A_\lambda})_i(c_\lambda))$ and
  extend it linearly.

  Since $G(f)$ is based, the diameter of $(P_A)_i(F(f)(c_\lambda))$ is
  equal to that of~$(P_{A_\lambda})_i(c_\lambda)$.  Since $F(A)_i$ is
  based by $\{F(f)(c_\lambda)\}$, it follows that for any $A\in
  \mathbf{C}$ the diameter function $d_{P_A}$ of $P_A$ is equal to
  the function $\delta$ defined by
  \[
  \delta(k) := \max \{d((P_{A_\lambda})_i(c_\lambda)) \mid i \le k, \,
  (A_\lambda,c_\lambda)\in \mathcal{M}_i \}.
  \]
  The value $\delta(k)$ is finite for any $k$, since $\mathcal{M}_i$
  is a finite collection for any~$i$.
\end{proof}

The proof of Theorem~\ref{theorem:acyclic-model-theorem-with-control}
tells us that the control function $\delta(k)$ is obtained from the
diameter of the chain homotopy on the models.  Using this, we can
often compute $\delta(k)$ explicitly, at least for small~$k$.  We deal
with an example in the next subsection.

\subsection{Controlled Eilenberg-Zilber theorem}
\label{subsection:controlled-eilenberg-zilber-theorem}

In this subsection, we investigate uniform control for the chain
homotopies of the Eilenberg-Zilber theorem for products.  Our result
is best described using simplicial sets.  Readers not familiar with
simplicial sets may refer to our quick review of basic definitions in
the appendix.

We first state a theorem, and then recall the terminologies used in
the statement for the reader's convenience.

\begin{theorem}[Controlled Eilenberg-Zilber Theorem]
  \label{theorem:controlled-eilenberg-zilber-chain-homotopy}
  For simplicial sets $X$ and $Y$, let
  \begin{gather*}
    \Delta_{X,Y} \colon C_*(X\times Y) \longrightarrow C_*(X) \otimes
    C_*(Y)
    \\
    \nabla_{X,Y} \colon C_*(X) \otimes C_*(Y) \longrightarrow
    C_*(X\times Y)
  \end{gather*}
  be the Alexander-Whitney map and the shuffle map.  Then there is a
  natural family of chain homotopies
  \[
  \{P_{X,Y}\colon \nabla_{X,Y}\circ\Delta_{X,Y} \simeq
  \id_{C_*(X\times Y)}\mid \text{$X$ and $Y$ are simplicial sets}\}
  \]
  which is uniformly controlled by a function~$\dEZ(k)$.
  Furthermore, the value of $\dEZ(k)$ for $k\le 4$ is as
  follows.

  \smallskip
  \begin{center}
    \normalfont
    \begin{tabular}
      {c*{5}{p{\widthof{00}}<{\centering}}}
      \toprule
      $k$ & 0 & 1 & 2 & 3 & 4\\
      $\dEZ(k)$ & 0 & 1 & 4 & 11 & 26\\
      \bottomrule
    \end{tabular}
  \end{center}
\end{theorem}

\begin{remark}
  \label{remark:controlled-eilenberg-zilber}
  \begin{enumerate}
  \item\label{remark-item:what-is-new-in-controlled-eilenberg-zilber}
    Of course the existence of the chain homotopy $P_{X,Y}$ is due to
    Eilenberg-Zilber~\cite{Eilenberg-MacLane:1953-1}.  What
    Theorem~\ref{theorem:controlled-eilenberg-zilber-chain-homotopy}
    newly gives is that $\{P_{X,Y}\}$ is uniformly controlled, and
    that the values of the control function $\dEZ$ are as above.
  \item\label{remark-item:why-only-for-k<=3} In our applications,
    explicit values of $\dEZ(k)$ for $k\le 3$ are sufficient, since we
    are interested in chains arising from $3$-manifolds.
  \end{enumerate}
\end{remark}

Recall, for instance from the appendix, that a simplicial set $X$
consists of sets $X_n$ ($n=0,1,\ldots$), face maps $d_i\colon X_n \to
X_{n-1}$, and degeneracy maps $s_i\colon X_n \to X_{n+1}$
($i=0,1,\ldots,n$).  We call $\sigma\in X_n$ an \emph{$n$-simplex}
of~$X$.  Let $\Z X$ be the simplicial abelian group generated by $X$,
and denote its (unnormalized) Moore complex by $\Z X_*$.  In other
words, $\Z X_n$ is the free abelian group generated by $X_n$, and the
boundary map $\partial\colon \Z X_n \to \Z X_{n-1}$ is defined by
$\partial\sigma=\sum_i (-1)^i d_i\sigma$ for $\sigma\in X_n$.  We
always view $\Z X_*$ as a based chain complex; each $\Z X_n$ is based
by the $n$-simplices.  We denote the homology by $H_*(X):=H_*(\Z
X_*)$.

For two simplicial sets $X$ and $Y$, the product $X\times Y$ is
defined by $(X\times Y)_n := X_n\times Y_n$; writing $\sigma\times\tau
:= (\sigma,\tau)\in X_n\times Y_n$, the face and degeneracy maps are
defined by $d_i(\sigma\times\tau) = d_i\sigma\times d_i\tau$ and
$s_i(\sigma\times\tau)=s_i\sigma\times s_i\tau$.

The \emph{Alexander-Whitney map}
\[
\Delta = \Delta_{X,Y} \colon \Z(X\times Y)_* \longrightarrow \Z X_*
\otimes \Z Y_*
\]
is defined by
\begin{equation}
  \label{equation:alexander-whitney-definition}
  \Delta(\sigma\times\tau) = \sum_{i=0}^n d_{i+1}\cdots d_n\sigma
  \otimes (d_0)^i \tau
\end{equation}
for $\sigma\times\tau \in X_n\times Y_n$.  To define its chain
homotopy inverse, we use the following notation.  A
\emph{$(p,q)$-shuffle}
$(\mu,\nu)=(\mu_1,\ldots,\mu_p,\nu_1,\ldots,\nu_q)$ is a permutation
of $(1,\ldots,p+q)$ such that $\{\mu_i\}$, $\{\nu_i\}$ are both
increasing.  Let $\epsilon(\mu,\nu)$ be the sign of the permutation,
and $S_{p,q}$ be the set of $(p,q)$-shuffles.  Then the \emph{shuffle
  map} (or the Eilenberg-Zilber map or the Eilenberg-MacLane map)
\[
\nabla=\nabla_{X,Y} \colon \Z X_* \otimes \Z Y_* \longrightarrow
\Z(X\times Y)_*
\]
is defined by
\begin{equation}
  \label{equation:shuffle-definition}
  \nabla(\sigma\otimes\tau) = \sum_{(\mu,\nu)\in S_{p,q}}
  (-1)^{\epsilon(\mu,\nu)} (s_{\nu_q}\cdots s_{\nu_1}\sigma) \times
  (s_{\mu_p}\cdots s_{\mu_1}\tau)
\end{equation}
for $\sigma\otimes\tau \in \Z X_{p}\otimes \Z Y_{q}$.

It is verified straightforwardly that $\Delta$ and $\nabla$ are chain
maps and $\Delta\circ\nabla=\id$ on $\Z X_*\otimes \Z Y_*$.  It is
known that $\nabla\circ\Delta$ is chain homotopic to $\id$ on
$\Z(X\times Y)_*$, by an acyclic model argument with
$\mathcal{M}=\{\Delta^n\times\Delta^n\mid n\ge 0\}$ as models.  By
using our controlled version of the acyclic model theorem
(Theorem~\ref{theorem:acyclic-model-theorem-with-control}), we can
obtain the additional conclusions on the chain homotopy
$\nabla\circ\Delta \simeq \id$ as stated in
Theorem~\ref{theorem:controlled-eilenberg-zilber-chain-homotopy}.  We
describe details below.

\begin{proof}[Proof of Theorem~\ref{theorem:controlled-eilenberg-zilber-chain-homotopy}]

  We follow the standard acyclic model argument for a product.  Let
  $\sSet$ be the category of simplicial sets, and define a functor
  $F\colon\sSet\times\sSet \to \Chpb$ by $F(X,Y) := \Z(X\times Y)_*$.
  By definition, $F$ is based.  Let $\Delta^n$ be the standard
  $n$-simplex as a simplicial set; we write a $k$-simplex of
  $\Delta^n$ as a sequence $[v_0,\ldots,v_k]$ of integers $v_i$ such
  that $0\le v_0\le\cdots\le v_k\le n$.  Let
  $\mathcal{M}=\{(\Delta^n,\Delta^n)\mid n\ge 0\}$.  Then $F$ is
  acyclic with respect to $\mathcal{M}$, since
  $\Delta^n\times\Delta^n$ is contractible.  Also, $F$ is finitely
  based-free with respect to $\mathcal{M}$ since $\Z(X\times Y)_n$ is
  freely generated by
  \[
  \{f[0,\ldots,n]\times g[0,\ldots,n]\in (X\times Y)_n \mid
  f\colon\Delta^n\rightarrow X,\, g\colon \Delta^n \rightarrow Y
  \text{ are morphisms}\}.
  \]
  Note that there is only one model $(\Delta^n,\Delta^n)$ in each
  dimension~$n$.

  By Theorem~\ref{theorem:acyclic-model-theorem-with-control}, it
  follows that there is a function $\dEZ(k)$ and a natural family of
  chain homotopies $P_{X,Y}\colon \Z(X\times Y)_* \to \Z(X\times
  Y)_{*+1}$ between $\nabla_{X,Y}\circ\Delta_{X,Y}$ and $\id$, which is
  uniformly controlled by~$\dEZ$.

  We will explicitly compute the value $\dEZ(k)$ for small~$k$.  For
  convenience, denote
  \[
  P_k := (P_{\Delta^k,\Delta^k})_k\colon \Z(\Delta^k\times \Delta^k)_k
  \to \Z(\Delta^k\times \Delta^k)_{k+1}.
  \]
  The proof of
  Theorem~\ref{theorem:acyclic-model-theorem-with-control} tells us
  that $\dEZ(k)$ is exactly the diameter of the chain
  $P_k([0,\ldots,k]\times[0,\ldots,k])$, where
  $P_k([0,\ldots,k]\times[0,\ldots,k])$ is defined inductively as
  follows: assuming that $P_{k-1}([0,\ldots,k-1]\times[0,\ldots,k-1])$
  has been defined, $P_{k-1}$ is determined by naturality and
  $P_k([0,\ldots,k]\times [0,\ldots,k])\in
  \Z(\Delta^k\times\Delta^k)_{k+1}$ is defined to be a solution $x$ of
  the system of linear equations
  \begin{equation}
    \label{equation:equation-for-chain-homotopy}
    \partial x = (-P_{k-1}\partial+\nabla\circ\Delta-\id)([0,\ldots,k]\times[0,\ldots,k])
  \end{equation}
  where $\partial\colon \Z(\Delta^k\times\Delta^k)_{k+1}\to
  \Z(\Delta^k\times\Delta^k)_{k}$ is viewed as a linear map.

  We remark that
  \[
  \rank \Z(\Delta^k\times\Delta^k)_{k+1}=\binom{2k+2}{k}
  \quad\text{and} \quad \rank
  \Z(\Delta^k\times\Delta^k)_{k}=\binom{2k+1}{k},
  \]
  that is, the system \eqref{equation:equation-for-chain-homotopy}
  consists of $\binom{2k+1}{k}$ linear equations in $\binom{2k+2}{k}$
  variables.  It can be seen that the ranks grow exponentially, by
  using Stirling's formula.  Fortunately for small $k$ we can still
  find (or at least verify) solutions.  We describe details below.

  For $k=0$, $P_0([0]\times[0]) = 0$ satisfies
  \eqref{equation:equation-for-chain-homotopy} since
  $\nabla\circ\Delta=\id$ on $\Z(\Delta^0\times\Delta^0)_{0}$.  From
  this it follows that $\dEZ(0)=0$.

  For $k=1$, straightforward computation shows that
  \[
  \nabla\Delta([0,1]\times[0,1]) = \nabla([0]\otimes[0,1] +
  [0,1]\otimes[1]) = [0,0]\times[0,1] + [0,1]\times[1,1].
  \]
  Since it is equal to $\partial([0,0,1]\times[0,1,1])$,
  $P_1([0,1]\times[0,1]):=[0,0,1]\times[0,1,1]$ is a solution
  of~\eqref{equation:equation-for-chain-homotopy}.  Since this is a
  chain of diameter one, we have $\dEZ(1)=1$.

  For $k=2$, we have that
  \begin{align*}
    \nabla\Delta([0,1,2]\times[0,1,2]) &=
    \nabla([0]\otimes[0,1,2]+[0,1]\otimes[1,2]+[0,1,2]\otimes[2])
    \\
    &= [0,0,0]\times[0,1,2]-[0,0,1]\times[1,2,2]
    \\
    &\qquad+[0,1,1]\times[1,1,2]+[0,1,2]\times[2,2,2]
  \end{align*}
  and that
  \begin{align*}
    P_1\partial([0,1,2]\times[0,1,2]) &=
    P_1([1,2]\times[1,2]-[0,2]\times[0,2]+[0,1]\times[0,1])
    \\
    &= [1,1,2]\times[1,2,2]-[0,0,2]\times[0,2,2]+[0,0,1]\times[0,1,1].
  \end{align*}
  Using these, it is straightforward to verify that
  \begin{align*}
    P_2([0,1,2]\times[0,1,2]) &= -[0,0,0,1]\times[0,1,2,2] +
    [0,0,1,1]\times[0,1,1,2]
    \\
    &\qquad + [0,0,1,2]\times[0,2,2,2] - [0,1,1,2]\times[0,1,2,2]
  \end{align*}
  is a solution of~\eqref{equation:equation-for-chain-homotopy}.
  Since its diameter is 4, we have $\dEZ(2)=4$.

  For $k=3$, \eqref{equation:equation-for-chain-homotopy} is a system
  of $1225$ linear equations in $3136$ variables.  Aided by a
  computer, we found the following solution
  of~\eqref{equation:equation-for-chain-homotopy}:
  \begin{align*}
    P_3([0,1,2,3]\times[0,1,2,3])
    &= [0,0,0,0,1]\times[0,1,2,3,3] - [0,0,0,1,1]\times[0,1,2,2,3] \\
    &\qquad + [0,0,0,1,2]\times[0,2,3,3,3] +
    [0,0,1,1,1]\times[0,1,1,2,3] \\
    &\qquad - [0,0,1,1,2]\times[0,2,2,3,3] +
    [0,0,1,2,2]\times[0,2,2,2,3] \\
    &\qquad + [0,0,1,2,3]\times[0,3,3,3,3] +
    [0,1,1,1,2]\times[0,1,2,3,3] \\
    &\qquad - [0,1,1,2,2]\times[0,1,2,2,3] -
    [0,1,1,2,3]\times[0,1,3,3,3] \\
    &\qquad + [0,1,2,2,3]\times[0,1,2,3,3].
  \end{align*}
  We remark that we can verify by hand that it is a solution
  of~\eqref{equation:equation-for-chain-homotopy}.  From this it
  follows that $\dEZ(3)=d( P_3([0,1,2,3]\times[0,1,2,3])) = 11$.

  For $k=4$, our computation fully depends on a computer.  A solution
  of the system \eqref{equation:equation-for-chain-homotopy}, which
  has $15876$ equations in $44100$ variables in this case, is given by
  \begin{multline*}
    P_4([0,1,2,3,4]\times[0,1,2,3,4]) =
    \\
    \begin{aligned}
      & -[0,0,0,0,0,1]\times[0,1,2,3,4,4] +
      [0,0,0,0,1,1]\times[0,1,2,3,3,4]
      \\
      & + [0,0,0,0,1,2]\times[0,2,3,4,4,4] -
      [0,0,0,1,1,1]\times[0,1,2,2,3,4]
      \\
      & - [0,0,0,1,1,2]\times[0,2,3,3,4,4] +
      [0,0,0,1,2,2]\times[0,2,3,3,3,4]
      \\
      & - [0,0,0,1,2,3]\times[0,3,4,4,4,4] +
      [0,0,1,1,1,1]\times[0,1,1,2,3,4]
      \\
      & + [0,0,1,1,1,2]\times[0,2,2,3,4,4] -
      [0,0,1,1,2,2]\times[0,2,2,3,3,4]
      \\
      & + [0,0,1,1,2,3]\times[0,3,3,4,4,4] +
      [0,0,1,2,2,2]\times[0,2,2,2,3,4]
      \\
      & - [0,0,1,2,2,3]\times[0,3,3,3,4,4] +
      [0,0,1,2,3,3]\times[0,3,3,3,3,4]
      \\
      & + [0,0,1,2,3,4]\times[0,4,4,4,4,4] -
      [0,1,1,1,1,2]\times[0,1,2,3,4,4]
      \\
      & + [0,1,1,1,2,2]\times[0,1,2,3,3,4] -
      [0,1,1,1,2,3]\times[0,1,3,4,4,4]
      \\
      & - [0,1,1,2,2,2]\times[0,1,2,2,3,4] +
      [0,1,1,2,2,3]\times[0,1,3,3,4,4]
      \\
      & - [0,1,1,2,3,3]\times[0,1,3,3,3,4] -
      [0,1,1,2,3,4]\times[0,1,4,4,4,4]
      \\
      & - [0,1,2,2,2,3]\times[0,1,2,3,4,4] +
      [0,1,2,2,3,3]\times[0,1,2,3,3,4]
      \\
      & + [0,1,2,2,3,4]\times[0,1,2,4,4,4] -
      [0,1,2,3,3,4]\times[0,1,2,3,4,4].
    \end{aligned}
  \end{multline*}
  It follows that $\dEZ(4)=26$.
\end{proof}

\begin{remark}
  \label{remark:general-formula-for-eilenberg-zilber}
  In spite of
  Remark~\ref{remark:controlled-eilenberg-zilber}~(\ref{remark-item:why-only-for-k<=3}),
  it would be nicer if we had an explicit closed formula for
  $P_k([0,\ldots,k]\times[0,\ldots,k])$ for general~$k$; this would
  give a general formula for the chain homotopy $P_{X,Y}$ for any
  $X$,~$Y$, and possibly a closed formula for $\dEZ(k)$.  The author
  does not know the answer.
\end{remark}

\subsection{Conjugation on groups}
\label{subsection:controlled-chain-homotopy-for-conjugation}

Recall that for a group $G$, the (unnormalized) Moore complex $\Z
BG_*$ associated to the simplicial classifying space $BG$ (which is a
simplicial set) can be used to compute the group homology $H_*(G)$
with integral coefficients.  For example, see the appendix
(\ref{appendix:chain-complexes}
and~\ref{appendix:simplicial-classifying-spaces}).  In fact $\Z BG_*$
is equal to the unnormalized bar resolution tensored with~$\Z$.  An
explicit description of $\Z BG_*$ is as follows: $\Z BG_n$ is the free
abelian group generated by $BG_n:=\{[g_1,\ldots,g_n]\mid g_i\in G\}$,
and the boundary map $\partial\colon \Z BG_n\to \Z BG_{n-1}$ is given
by $\partial c = \sum_{i=0}^n (-1)^i d_ic$, where $d_i$ is defined by
\[
d_i[g_1,\ldots,g_n] = 
\begin{cases}
  [g_2,\ldots,g_n] &\text{if }i = 0,\\
  [g_1,\ldots, g_{i-1},g_i g_{i+1}, g_{i+2},\ldots,g_n] &\text{if }0<i<n,\\
  [g_1,\ldots,g_{n-1}] &\text{if }i = n.
\end{cases}
\]

As abuse of notation, for a group homomorphism $f$, we denote by $f$
the induced based chain map on $\Z B(-)_*$, that is,
$f[g_1,\ldots,g_n] = [f(g_1),\ldots,f(g_n)]$.

It is well known that for any group $G$ and $g\in G$, the conjugation
homomorphism $\mu_g\colon G\to G$ defined by $\mu_g(h)=h^g:=ghg^{-1}$
induces the identity map on~$H_*(G)$.  For example, see \cite[p.~191,
Theorem~6.7.8]{Weibel:1994-1}.  In the following theorem, we give a
chain level statement in terms of controlled chain homotopies, from
which the homological statement is immediately obtained.

\begin{theorem}
  \label{theorem:controlled-conjugation-chain-homotopy}
  There is a family of chain homotopies
  \[
  \{S_{G,g}\colon \id_{\Z BG_*} \simeq \mu_g \mid \text{\textup{$G$ is a
    group, $g\in G$}}\}
  \]
  which is uniformly controlled by the function
  $\dconj(k):=k+1$.  The chain homotopy $S_{G,g}$ is natural
  with respect to $(G,g)$, in the sense that $f S_{G,g} =
  S_{\Gamma,f(g)} f$ for any homomorphism $f\colon G\to \Gamma$.
\end{theorem}

To motivate our chain homotopy construction for
Theorem~\ref{theorem:controlled-conjugation-chain-homotopy}, we recall
a geometric interpretation of an $n$-simplex $[g_1,\ldots,g_n]$
of~$BG$ that arises from the nerve construction for~$G$: there is
exactly one $0$-simplex $[\,]$ in $BG$ which is the basepoint, and for
$n>0$, $[g_1,\ldots,g_n]\in BG_n$ corresponds to an $n$-simplex
$[v_0,\ldots,v_n]$ (which is possibly degenerate) in the geometric
realization of $BG$ whose edge $[v_{i-1},v_{i}]$ is a loop
representing $g_i \in \pi_1(BG)=G$.

Consider a prism $\Delta^n\times[0,1]$.  For convenience, we write
$\Delta^n = [e_0\ldots,e_n]$, and denote the vertices of
$\Delta^n\times[0,1]$ by $v_{ij} = (e_i,j)$, $i=0,\ldots,n$, $j=0,1$.
If there is a geometric homotopy from $\id_{BG}$ to the conjugation
$\mu_g$, then the restriction on a simplex $[g_1,\ldots,g_n]$ should
give a map of $\Delta^n\times[0,1]$ that sends the edges
$[v_{(i-1)0},v_{i0}]$ and $[v_{(i-1)1},v_{i1}]$ to $g_i$ and
$\mu_g(g_i) = g_i^g$ respectively.  This tells us what the restriction
$\Delta^n\times\{0,1\}\to BG$ should be.  The standard prism
decomposition divides the product $\Delta^n\times[0,1]$ into $n+1$
simplices.  It turns out that, for instance as illustrated in
Figure~\ref{figure:conjugation-decomposition} for $n=2$, we can label
edges of the resulting simplices in such a way that the prescribed
$\Delta^n\times\{0,1\}\to BG$ extends to $\Delta^n\times[0,1]$
simplicially.  Note that in
Figure~\ref{figure:conjugation-decomposition} each path
$e_i\times[0,1]$ is sent to the loop $g^{-1}$, so that the basepoint
change effect of the homotopy is exactly the conjugation by $g$ on
$\pi_1(BG) = G$.

\begin{figure}[ht]
  \medskip
  $\hbox{}\quad
  \vcenter{\hbox{
  \labellist\small
  \pinlabel{$v_{00}$} at -3 30
  \pinlabel{$v_{10}$} at 102 -2
  \pinlabel{$v_{20}$} at 156 70
  \pinlabel{$v_{01}$} at -3 116
  \pinlabel{$v_{11}$} at 102 70
  \pinlabel{$v_{21}$} at 156 147
  \pinlabel{$g_1$} at 45 10
  \pinlabel{$g_2$} at 130 35
  \pinlabel{$g^{-1}$} at 160 110
  \pinlabel{$g^{-1}$} at 80 38
  \pinlabel{$g^{-1}$} at -9 74
  \pinlabel{$g_1^g$} at 52 100
  \pinlabel{$g_2^g$} at 117 94
  \endlabellist
  \includegraphics[scale=.8]{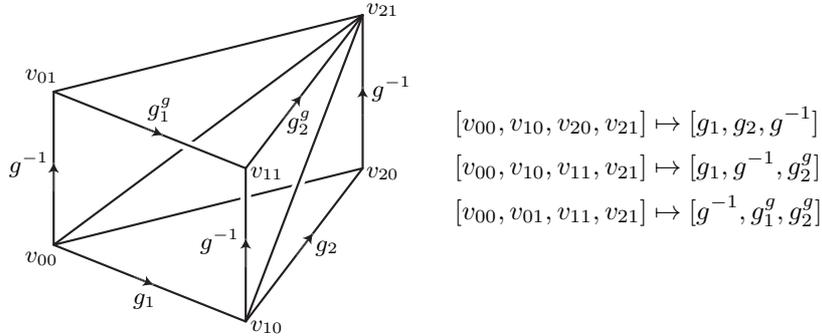}}}
  \qquad\quad
  \begin{aligned}\relax
    [v_{00},v_{10},v_{20},v_{21}] &\mapsto [g_1,g_2,g^{-1}]
    \\
    [v_{00},v_{10},v_{11},v_{21}] &\mapsto [g_1,g^{-1},g_2^g]
    \\
    [v_{00},v_{01},v_{11},v_{21}] &\mapsto [g^{-1},g_1^g, g_2^g]
  \end{aligned}
  $
  \caption{Prism decomposition of a homotopy for conjugation.} 
  \label{figure:conjugation-decomposition}
\end{figure}

Generalizing Figure~\ref{figure:conjugation-decomposition} to an
arbitrary dimension $n$, we obtain the chain homotopy formula used in
the formal proof of
Theorem~\ref{theorem:controlled-conjugation-chain-homotopy} given
below.

\begin{proof}[Proof of Theorem~\ref{theorem:controlled-conjugation-chain-homotopy}]
  For a group $G$ and an element $g\in G$, we define a chain homotopy
  \[
  S = S_{G,g} \colon \Z BG_* \to \Z BG_{*+1}
  \]
  by
  \[
  S[g_1,\ldots,g_n] = \sum_{i=0}^n (-1)^i [g_1,\ldots,g_i, g^{-1},
  g_{i+1}^g, \ldots, g_n^g].
  \]
  By a straightforward computation it is verified that $S\partial
  + \partial S = \mu_g - \id$.  From the defining formula, it follows
  that $S_{G,g}$ is natural and that $d_{S_{G,g}}(k) \le k+1$.
\end{proof}

\section{Chain homotopy for embeddings into mitoses}
\label{section:chain-homotopy-for-mitosis}

We begin by recalling a definition of Baumslag, Dyer, and Heller to
set up notations.  As before, we write $g^h := hgh^{-1}$.

\begin{definition}[\cite{Baumslag-Dyer-Heller:1980-1}]
  \label{definition:mitosis}
  Suppose $G$ is a group.  A group $M$ endowed with an embedding
  $\imath\colon G\to M$ is a \emph{mitosis} of $G$ if there are
  elements $u$, $t\in M$ such that $M$ is generated by $\imath(G)\cup
  \{u,t\}$ and $g^t = gg^u$, $[h,g^u] = e$ for any $g$, $h\in
  \imath(G)$.  In particular, define
  \[
  m(G) := \langle G, u, t \mid [h, g^u] = e,\, g^t = gg^u \text{ for
    any $g$, $h\in G$}\rangle.
  \]
  Then $m(G)$ together with the natural embedding $k_G\colon G \to
  m(G)$ is a mitosis of~$G$.

  Define $\bA^0(G)=G$, $\bA^n(G):=m(\bA^{n-1}(G))$ for $n\ge 1$
  inductively.  We denote by $i^n_G\colon G \to \bA^n(G)$ the
  composition $k_{\bA^{n-1}(G)}\circ \cdots \circ k_{\bA^1(G)}\circ
  k_G$.
\end{definition}

As observed in \cite{Baumslag-Dyer-Heller:1980-1}, it is verified
straightforwardly that (i) $m\colon \Gp \to \Gp$ is a functor of the
category $\Gp$ of groups, (ii) $k_G$ is a natural transformation
$\id_{\Gp} \to m$ which is injective for each~$G$, and (iii)
$m(f)\colon m(G)\to m(\Gamma)$ is injective whenever $f\colon G\to
\Gamma$ is an injective group homomorphism.  Consequently (i), (ii),
and (iii) hold for $(\bA^n, i^n_G)$ in place of~$(m,k_G)$.

In~\cite{Baumslag-Dyer-Heller:1980-1}, they showed that if
$\mathds{k}$ is a field, then the map $H_i(G;\mathds{k})\to
H_i(\bA^n(G);\mathds{k})$ induced by $i^n_G$ is zero for
$i=1,\ldots,n$.  Our main aim of this section is to prove the
following chain level result
(Theorem~\ref{theorem:controlled-chain-homotopy-BDH}), which
particularly gives this homological result
of~\cite{Baumslag-Dyer-Heller:1980-1} as an immediate consequence.

We denote the trivial group homomorphism by $e_{\pi,G}\colon \pi\to
G$.  When the groups $\pi$ and $G$ are understood from the context, we
write $e=e_{\pi,G}$ by dropping $\pi, G$ from the notation.  Recall
that we denote by $f\colon \Z BG_*\to \Z B\Gamma_*$ the chain map
induced by a group homomorphism $f\colon G\to \Gamma$.

\begin{theorem}
  \label{theorem:controlled-chain-homotopy-BDH}
  For each $n$, there is a family
  \[
  \{\Phi^n_G\colon e\simeq i^n_G \mid \text{\textup{$G$ is a group}}\}
  \]
  of partial chain homotopies $\Phi^n_G$ defined in dimension $\le n$,
  between the chain maps $e$, $i^n_G\colon \Z BG_* \to \Z B\bA^n(G)_*$,
  which is uniformly controlled by a function~$\dBDH$.  For $k\le 4$,
  the value of $\dBDH(k)$ is as follows:

  \smallskip
  \begin{center}
    \normalfont
    \begin{tabular}
      {c*{5}{p{\widthof{0000}}<{\centering}}}
      \toprule
      $k$ & 0 & 1 & 2 & 3 & 4\\
      $\dBDH(k)$ & 0 & 6 & 26 & 186 & 3410\\
      \bottomrule
    \end{tabular}
  \end{center}
\end{theorem}

A precise definition of $\dBDH$ will be given in
Definition~\ref{definition:control-function-for-BDH}.  Note that the
control function $\dBDH$ is independent of~$n$.  The values of
$\dBDH(k)$ for $k\le 3$ will be essential in proving
Theorem~\ref{theorem:linear-universal-bound} stated in the
introduction.

The remainder of this section is devoted to the proof of
Theorem~\ref{theorem:controlled-chain-homotopy-BDH}.  As a
preliminary, we make some observations on the product of groups.  From
the definition, for groups $G$ and~$H$, we have $B(G\times H) =
BG\times BH$ as simplicial sets.  Let
\[
\Delta = \Delta_{BG,BH}\colon \Z(BG\times BH)_*\to \Z BG_*\otimes
\Z BH_*
\]
be the Alexander-Whitney map.  We define
\[
\Lambda_G,\, \Lambda_H, \Lambda \colon \Z(BG\times BH)_*\to \Z
BG_*\otimes \Z BH_*
\]
by
\begin{align*}
  \Lambda_G(\sigma\times\tau) & := \sigma\otimes (d_0)^n \tau =
  \sigma\otimes [\,]
  \\
  \Lambda_H(\sigma\times\tau) & := d_1\cdots d_n\sigma \otimes\tau =
  [\,] \otimes \tau,
\end{align*}
for $\sigma\times\tau\in (BG\times BH)_n$, and by $\Lambda :=
\Delta-\Lambda_G-\Lambda_H$.  Note that if $n\ge 1$, $\Lambda_H$ and
$\Lambda_G$ are the first and last term of the defining formula
\eqref{equation:alexander-whitney-definition} of $\Delta$
respectively.  Consequently, $\Lambda$ is the sum of the remaining
terms.

\begin{lemma}
  \label{lemma:parts-of-alexander-whitney-are-chain-maps}
  The maps $\Lambda_G$, $\Lambda_H$, and $\Lambda$ are chain maps.
\end{lemma}

\begin{proof}
  Since
  \[
  \Lambda_H\partial (\sigma\times\tau) = \Lambda_H \Big(\sum_i (-1)^i
  d_i\sigma\times d_i\tau\Big) = \sum_i (-1)^i ([\,]\otimes d_i \tau) =
  [\,]\otimes \partial\tau =
  \partial \Lambda_H (\sigma\times\tau),
  \]
  we have that $\Lambda_H$ is a chain map.  A similar argument works
  for~$\Lambda_G$.  Since $\Delta$ is a chain map, it follows that
  $\Lambda = \Delta-\Lambda_G-\Lambda_H$ is a chain map.
\end{proof}

For the next lemma, recall that $\dEZ(k)$ is the control
function in
Theorem~\ref{theorem:controlled-eilenberg-zilber-chain-homotopy}.

\begin{lemma}
  \label{lemma:product-chain-contraction}
  Suppose $f\colon G\to K$ and $g\colon H \to L$ are group
  homomorphisms.  Suppose $Q\colon e\simeq f$ is a partial chain
  homotopy defined in dimension $\le n-1$ between $e$, $f\colon \Z
  BG_* \to \Z BK_*$, that is, $Q\partial + \partial Q = f-e$ on $\Z
  BG_i$ for $i\le n-1$.  Suppose $Q_0=0$ on $\Z BG_0$.  Consider the
  product homomorphisms
  \[
  f\times g,\, f\times e,\, e\times g\colon G\times H \to K\times L
  \]
  and the induced chain maps $\Z(BG\times BH)_* \to \Z(BK\times
  BL)_*$.  Let $P=P_{BK,BL}\colon \nabla\Delta\simeq\id$ be the chain
  homotopy in
  Theorem~\ref{theorem:controlled-eilenberg-zilber-chain-homotopy}.
  Then
  \[
  T:=P(f\times g - e\times g) + \nabla(Q\otimes g)\Lambda\colon
  \Z(BG\times BH)_*\rightarrow \Z(BK\times BL)_{*+1}
  \]
  is a partial chain homotopy
  \[
  T\colon (f\times e-e\times e) + (e\times g-e\times e) \simeq (f\times
  g-e\times e)
  \]
  defined in dimension $\le n$.  Furthermore it satisfies that $T_0=0$
  on $C_0(BK\times BL)$, that is, $d_T(0)=0$, and
  \[
  d_T(k) \le 2\cdot \dEZ(k) + (k-1)\binom{k}{\lfloor k/2
    \rfloor}\cdot d_Q(k-1) \quad \textup{for }k\ge 1.
  \]
\end{lemma}

We remark that $\Delta$, $\Lambda$, and $\nabla$ in the above
statements are those for the product of $BK$ and~$BL$.

\begin{proof}
  By
  Lemma~\ref{lemma:basic-properties-of-diameter}~(\ref{lemma-item:diameter-tensor-product}),
  we have that $Q\otimes g: e\otimes g\simeq f\otimes g$ is a partial
  chain homotopy.  More precisely, on $\sum_{i<n} \Z BG_i\otimes \Z
  BH_*$,
  \begin{equation}
    \begin{aligned}
      (Q\otimes g)\partial + \partial (Q\otimes g) 
      &= Q\partial
      \otimes g \pm Q\otimes g\partial + \partial Q \otimes g \mp
      Q\otimes \partial g
      \\
      & = (Q\partial+\partial Q)\otimes g = f\otimes g - e\otimes g.
    \end{aligned}
    \label{equation:tensor-chain-homotopy}
  \end{equation}

  By the definitions, for any $f$ and~$g$, the following diagram
  commutes:
  \begin{equation}
    \vcenter{
      \xymatrix{
        \Z(BG\times BH)_* \ar[r]^{f\times g} \ar[d]_\Delta
        &
        \Z(BK\times BL)_* \ar[d]^\Delta
        \\
        \Z BG_*\otimes \Z BH_* \ar[r]_{f\otimes g}
        &
        \Z BK_*\otimes \Z BL_*
      }
    }.
    \label{equation:alexander-whitney-commutative}
  \end{equation}

  We also have
  \begin{equation}
    \nabla(f\otimes g)\Lambda_G(\sigma\times\tau) =
    \nabla(f\otimes g)(\sigma\otimes[\,]) = \nabla(f\sigma
    \otimes [\,]) = (f\times e)(\sigma\times\tau),
    \label{equation:lambda-G-equation}
  \end{equation}
  for any $f$ and $g$.  Similarly
  \begin{equation}
    \nabla(f\otimes g)\Lambda_H = e\times g.
    \label{equation:lambda-H-equation}
  \end{equation}

  Now, on $\Z (BG\times BH)_k$ with $1\le k\le n$, we have
  \begin{equation}
    \label{equation:product-chain-homotopy-verification}
    \begin{aligned}
      f\times g - e\times g
      & \simeq \nabla\Delta (f\times g - e\times g)
      &&\text{by Theorem~\ref{theorem:controlled-eilenberg-zilber-chain-homotopy}}
      \\
      & = \nabla(f\otimes g - e \otimes g)\Delta
      && \text{by \eqref{equation:alexander-whitney-commutative}} 
      \\
      & = \nabla(f\otimes g - e\otimes g)(\Lambda_G+\Lambda_H+\Lambda)
      && \text{by definitions}
      \\
      & = (f\times e - e\times e) + (e\times g - e\times g)
      && \text{by \eqref{equation:lambda-G-equation}, \eqref{equation:lambda-H-equation},}
      \\
      & \qquad + \nabla((Q\otimes g)\partial+\partial
      (Q\otimes g))\Lambda
      && \text{\hphantom{by }and \eqref{equation:tensor-chain-homotopy}}
      \\
      & = (f\times e-e\times e) + \nabla(Q\otimes g)\Lambda \partial +
      \partial \nabla (Q\otimes g)\Lambda
      && \text{by
        Lemma~\ref{lemma:parts-of-alexander-whitney-are-chain-maps}.}
    \end{aligned}
  \end{equation}
  Note that in \eqref{equation:product-chain-homotopy-verification} we
  can apply \eqref{equation:tensor-chain-homotopy} since the image of
  $\Z (BG\times BH)_k$ under $\Lambda$ lies in $\sum_{i=1}^{k-1}
  \Z BG_i\otimes \Z BH_*$.

  On $\Z (BG\times BH)_0$, we have $f\times g-e\times g = 0 = f\times
  e-e\times e$.

  Let $P = P_{BK,BL}$ be the chain homotopy given by
  Theorem~\ref{theorem:controlled-eilenberg-zilber-chain-homotopy},
  and let
  \[
  T := P(f\times g - e\times g) + \nabla(Q\otimes g)\Lambda.
  \]
  Note that $T_0 = 0$ on $\Z(BG\times BH)_0$ since $Q_0=0$.
  From~\eqref{equation:product-chain-homotopy-verification} and
  Lemma~\ref{lemma:basic-properties-of-diameter}
  (\ref{lemma-item:diameter-sum}),~(\ref{lemma-item:diameter-composition}),
  it follows that $T$ is a partial chain homotopy between $(f\times
  e-e\times e) + (e\times g - e\times e)$ and $f\times g - e\times e$
  in dimension $\le n$.

  Now we estimate the diameter $d_T(k)$ of~$T$.  The chain maps
  $f\times g$ and $e\times g$ have diameter function $\equiv 1$.
  Observe that the defining formula
  \eqref{equation:shuffle-definition} for $\nabla$ has
  $\binom{p+q}{p}$ summands, since the number of $(p,q)$-shuffles is
  $\binom{p+q}{p}$.  It follows that $d_\nabla(k) \le
  \binom{k}{\lfloor k/2 \rfloor}$.  Similarly, from the defining
  formula \eqref{equation:alexander-whitney-definition} for $\Delta$,
  it follows that $d_\Lambda(k) \le k-1$.  Note that $d_{(Q\otimes
    g)\Lambda}(k) \le d_Q(k-1)\cdot d_\Lambda(k)$ since the $Q$ factor
  in the expression $(Q\otimes g)\Lambda$ is applied to only chains of
  dimension at most $k-1$.  Combining the above observations using
  Lemma~\ref{lemma:basic-properties-of-diameter}, we obtain the
  claimed estimate for~$d_T(k)$.
\end{proof}

\begin{remark}
  A \emph{reduced simplicial set} is defined to be a simplicial set
  with a unique $0$-simplex.
  Lemmas~\ref{lemma:parts-of-alexander-whitney-are-chain-maps}
  and~\ref{lemma:product-chain-contraction} hold for reduced
  simplicial sets, although we stated and proved them for classifying
  spaces of groups only.  The proofs are identical.
\end{remark}

We use the above results to show a key property of the mitosis
embedding $k_G\colon G\to m(G)$ on the chain level.

\begin{theorem}
  \label{theorem:mitosis-chain-homotopy}
  Suppose $\phi\colon \pi\to G$ is a group homomorphism and $Q\colon e
  \simeq \phi$ is a partial chain homotopy defined in dimension $\le
  n-1$ between $e$, $\phi\colon \Z B\pi_* \to \Z BG_*$ such that
  $Q_0=0$ on~$\Z B\pi_0$.  Then there is a partial chain homotopy
  $R\colon e \simeq k_G\circ\phi$ defined in dimension $\le n$ between
  $e$, $k_G\circ\phi\colon \Z B\pi_* \to \Z Bm(G)_*$.  In addition,
  $R_0 = 0$ on $\Z B\pi_0$, that is, $d_R(0)=0$, and
  \[
  d_R(k) \le 2(k+1) + 2\cdot \dEZ(k) + (k-1)\binom{k}{\lfloor
    k/2 \rfloor} \cdot d_Q(k-1) \quad\textup{for } k\ge 1.
  \]
\end{theorem}

\begin{proof}
  Recall that
  \[
  m(G) = \langle G, u, t \mid [h,g^u] = e,\, g^t = gg^u\text{ for any
  }g,\,h\in G\rangle.
  \]
  Define inclusions $i,\, j,\, D\colon \pi\to \pi\times\pi$ by $i(g) =
  (g,e)$, $j(g) = (e,g)$, and $D(g) = (g,g)$.  Define $f\colon G\times
  G \to m(G)$ by $f(g,h) = gh^u$.  Recall $\mu_g(h)=h^g$ denotes the
  conjugation by~$g$.  Consider the following diagram:
  \[
  \xymatrix@C=2.5em{
    & \Z(B\pi\times B\pi)_* \ar[r]^-{\phi\times\phi}
    & \Z(BG\times BG)_* \ar[r]^-{f}
    & \Z Bm(G)_*
    \\
    \Z B\pi_* \ar[ur]^{j} \ar[r]^-{i} \ar[dr]_{D}
    & \Z(B\pi\times B\pi)_* \ar[r]^-{\phi\times\phi} 
    & \Z(BG\times BG)_* \ar[r]^-{f}
    & \Z Bm(G)_* \ar[u]_{\mu_u} \ar[d]^{\mu_t}
    \\
    & \Z(B\pi\times B\pi)_* \ar[r]^-{\phi\times\phi}
    & \Z(BG\times BG)_* \ar[r]^-{f}
    & \Z Bm(G)_*
  }
  \]
  It commutes since it is obtained from a commutative diagram of
  group homomorphisms.

  For $g\in m(G)$, let $S_g:=S_{m(G),g} \colon \id\simeq \mu_g$ be the
  chain homotopy in
  Theorem~\ref{theorem:controlled-conjugation-chain-homotopy}.  Then
  we obtain a chain homotopy
  \begin{equation}
    \label{equation:i=j}
    S_u f (\phi\times\phi) i\colon f(\phi\times\phi) i \simeq
    \mu_u f (\phi\times\phi) i  = f(\phi\times\phi) j
  \end{equation}
  by
  Lemma~\ref{lemma:basic-properties-of-diameter}~(\ref{lemma-item:diameter-composition}).
  Similarly we obtain a chain homotopy
  \begin{equation}
    \label{equation:i=D}
    S_t f (\phi\times\phi) i\colon f(\phi\times\phi) i \simeq
    f(\phi\times\phi) D.
  \end{equation}

  Since $Q\colon e\simeq \phi$,
  Lemma~\ref{lemma:product-chain-contraction} gives us a partial chain
  homotopy
  \[
  T\colon (\phi\times e - e\times e) + (e\times\phi - e\times e)
  \simeq \phi\times\phi - e\times e
  \]
  in dimension $\le n$.  From this we obtain a
  partial chain homotopy
  \[
  f T D\colon f(\phi\times e + e\times\phi - e\times e)D \simeq f
  (\phi\times\phi)D
  \]
  in dimension $\le n$, by
  Lemma~\ref{lemma:basic-properties-of-diameter}~(\ref{lemma-item:diameter-composition}).
  Since
  \[
  f (\phi\times e) D = f(\phi\times\phi) i,\quad f (e\times\phi)
  D = f(\phi\times\phi) j,\quad f(e\times e)D = e,
  \]
  it follows that $fTD$ is indeed a chain homotopy
  \begin{equation}
    \label{equation:i+j=D}
    fTD \colon f(\phi\times\phi)i + f(\phi\times\phi)j - e
    \simeq f (\phi\times\phi)D.
  \end{equation}
  Now we have
  \begin{align*}
    k_G\circ\phi - e = f(\phi\times\phi)i - e
    & \simeq f(\phi\times\phi) D - f(\phi\times\phi)j
    && \text{by \eqref{equation:i+j=D}}
    \\
    & \simeq f(\phi\times\phi)i - f(\phi\times\phi)j
    && \text{by \eqref{equation:i=D}}
    \\
    & \simeq f(\phi\times\phi)j - f(\phi\times\phi)j = 0
    && \text{by \eqref{equation:i=j}}.
  \end{align*}
  Also, Lemma~\ref{lemma:basic-properties-of-diameter}
  (\ref{lemma-item:diameter-sum}) tells us that
  \[
  R := fTD-S^tf(\phi\times\phi)i+S^uf(\phi\times\phi)i
  \]
  is a chain homotopy $R\colon e \simeq k_G\circ\phi$.  Since $Q_0=0$
  by the hypothesis, we have $T_0=0$ by
  Lemma~\ref{lemma:product-chain-contraction}.  From this it follows
  that $R_0=0$, that is, $d_R(0)=0$.  Also, by
  Lemma~\ref{lemma:basic-properties-of-diameter}
  (\ref{lemma-item:diameter-sum}) and by the estimates in
  Theorem~\ref{theorem:controlled-conjugation-chain-homotopy} and
  Lemma \ref{lemma:product-chain-contraction}, we obtain
  \begin{align*}
    d_R(k) &\le d_{S^t}(k)+d_{S^u}(k) + d_T(k)
    \\
    &\le 2(k+1) + 2\cdot \dEZ(k) + (k-1)\binom{k}{\lfloor k/2 \rfloor}
    \cdot d_Q(k-1) \quad\text{for }k \ge 1.  \qedhere
  \end{align*}
\end{proof}

Applying Theorem~\ref{theorem:mitosis-chain-homotopy} repeatedly, we
obtain the following result for $i^n_G\colon G \to \bA^n(G)$.

\begin{definition}
  \label{definition:control-function-for-BDH}
  Let $\dBDH \colon \{0,\ldots,n\} \to \Z_{\ge 0}$ be the
  function defined inductively by the initial condition
  $\dBDH(0) = 0$ and the recurrence relation
  \[
  \dBDH(k) = 2(k+1) + 2\cdot \dEZ(k) +
  (k-1)\binom{k}{\lfloor k/2\rfloor} \cdot \dBDH(k-1)
  \]
  for $k\ge 1$.
\end{definition}

\begin{corollary}
  \label{corollary:chain-homotopy-for-BDH}
  For each integer $n\ge 0$, there is a family 
  \[
  \{\Phi^n_G\colon e\simeq i^n_G \mid \text{$G$ is a group}\}
  \]
  of partial chain homotopies in dimension $\le n$ between $e$,
  $i^n_G\colon \Z BG_* \to \Z B\bA^n(G)_*$, which is uniformly
  controlled by~$\dBDH$.
\end{corollary}

\begin{proof}
  For $n=0$, the zero map $\Phi_G := 0$ is a partial chain homotopy
  $\Phi^G\colon e\simeq \id_G=i^0_G$ in dimension $\le 0$.  So the
  claimed conclusion holds.

  Suppose the conclusion holds for $n-1$.  Applying
  Theorem~\ref{theorem:mitosis-chain-homotopy} to $\phi :=
  i^{n-1}_G\colon G\to \bA^{n-1}(G)$ and $Q := \Phi^{n-1}_G\colon
  e\simeq i^{n-1}_G$, it follows that there is a partial chain
  homotopy
  \[
  \Phi^n_G\colon e\simeq k_{\bA^{n-1}G}\circ i^{n-1}_G = i^n_G
  \]
  in dimension $\le n$ which satisfies $d_{\Phi^n_G}(0)=0$ and
  \[
  d_{\Phi^n_G}(k) \le 2(k+1) + 2\cdot \dEZ(k) +
  (k-1)\binom{k}{\lfloor k/2\rfloor} \cdot d_{\Phi^{n-1}_G}(k-1)
  \quad\textup{for }k \ge 1.
  \]
  Since $\{\Phi^{n-1}_G\}$ is uniformly controlled by
  $\dBDH$, the conclusion for $n$ follows.
\end{proof}

Now we are ready to complete the proof of
Theorem~\ref{theorem:controlled-chain-homotopy-BDH} stated in the
beginning of this section.

\begin{proof}[Proof of
  Theorem~\ref{theorem:controlled-chain-homotopy-BDH}]

  The existence of the desired uniformly controlled family of chain
  homotopies in Theorem~\ref{theorem:controlled-chain-homotopy-BDH} is
  no more than Corollary~\ref{corollary:chain-homotopy-for-BDH}.  For
  $k\le 4$, the values of $\dBDH(k)$ are obtained by an
  inductive straightforward computation, using
  Definition~\ref{definition:control-function-for-BDH} and the values
  of $\dEZ(k)$ given in
  Theorem~\ref{theorem:controlled-eilenberg-zilber-chain-homotopy}.
\end{proof}

\section{Explicit universal bounds from presentations of
  $3$-manifolds}
\label{universal-bound-computation}

In this section we obtain explicit estimates of the Cheeger-Gromov
universal bound from fundamental presentations of $3$-manifolds.

\subsection{Bounds from triangulations}
\label{subsection:bound-from-simplicial-complexity}

The goal of this subsection is to give a proof of
Theorem~\ref{theorem:linear-universal-bound}: \emph{suppose $M$ is a
  $3$-manifold with simplicial complexity~$n$.  Then for any
  $\phi\colon\pi_1(M)\to G$,}
\[
|\rhot(M,\phi)| \le 363090\cdot n.
\]
Recall that the \emph{simplicial complexity} of a $3$-manifold $M$ is
the minimal number of $3$-simplices in a triangulation (i.e., a
simplicial complex structure) of~$M$. 

In the proof, we will use the results developed in
Sections~\ref{section:chain-vs-bordism},
\ref{section:controlled-chain-homotopy},
and~\ref{section:chain-homotopy-for-mitosis}, as well as the idea of
the existence proof of
Theorem~\ref{theorem:existence-of-universal-bound-intro} given in
Section~\ref{section:topological-proof}.  First we state a corollary
of Theorem~\ref{theorem:existence-of-efficient-bordism} and
Corollary~\ref{corollary:chain-homotopy-for-BDH}.  Recall that we
defined the functorial embedding $i^n_G\colon G\to \bA^n(G)$ in
Definition~\ref{definition:mitosis}.

\begin{theorem}
  \label{theorem:efficient-bordism-over-A^3}
  Suppose $M$ is a $3$-manifold with simplicial complexity~$n$.  View
  $M$ as a manifold over $\bA^3(\pi_1(M))$ via the embedding
  $i^3_{\pi_1(M)}\colon \pi_1(M)\to \bA^3(\pi_1(M))$.  Then there is a
  smooth bordism $W$ over $\bA^3(\pi_1(M))$ between $M$ and a trivial
  end, whose $2$-handle complexity is at most $181545\cdot d(\zeta_M)$.
\end{theorem}

In the proof of Theorem~\ref{theorem:efficient-bordism-over-A^3} given
below, there is a small technicality which arises from that we use two
chain complexes for a simplicial set~$X$: the cellular chain complex
$C_*(X)$ of its geometric realization, which was used in
Section~\ref{section:chain-vs-bordism}, and the Moore complex $\Z X_*$
of the simplicial abelian group $\Z X$ associated to~$X$, which was
used in Sections~\ref{section:controlled-chain-homotopy}
and~\ref{section:chain-homotopy-for-mitosis}.  It is known that if we
denote by $D_*(X)$ the subgroup of $\Z X_*$ generated by degenerate
simplices of $X$, then $D_*(X)$ is indeed a subcomplex, $C_*(X)\cong
\Z X_* / D_*(X)$, and the projection $p\colon \Z X_* \to C_*(X)$ is a
chain homotopy equivalence~\cite[p.~236]{MacLane:1995-1}.  See the
appendix (\ref{appendix:chain-complexes}) for more details.

\begin{proof}[Proof of Theorem~\ref{theorem:efficient-bordism-over-A^3}]
  For brevity, we write $\pi:=\pi_1(M)$, $\Gamma := \bA^3(\pi_1(M))$,
  and $i:=i^3_{\pi_1(M)}\colon \pi\to \Gamma$.  Choose a simplicial
  complex structure of $M$ with minimal number of $3$-simplices.  By
  abuse of notation, we denote by $M$ the simplicial set obtained from
  this simplicial complex structure.  As before, let $\zeta_M\in
  C_*(M)$ be the sum of oriented $3$-simplices of $M$ that represents
  the fundamental class $[M]\in H_3(M)$.  Since $M$ is a simplicial
  complex, $C_*(X)$ is a subcomplex of $\Z_* X$, and the projection
  $p\colon \Z X_* \to C_*(X)$ is a left inverse of the inclusion.  In
  particular $\zeta_M$ lifts to a cycle $\xi_M\in \Z M_3$.  We have
  $d(\xi_M)=d(\zeta_M)$.

  By Theorem~\ref{theorem:simplicial-cellular-approximation} (see also
  Proposition~\ref{proposition:simplicial-realization} in the
  appendix), the identity map $\pi_1(M)\to\pi=\pi_1(B\pi)$ induces a
  simplicial-cellular map $j\colon M\to B\pi$.  Let $\phi = i\circ
  j\colon M\to B\pi\to B\Gamma$.  By
  Theorem~\ref{theorem:controlled-chain-homotopy-BDH}, there is a
  partial chain homotopy $\Phi\colon e\simeq i$ defined in dimension
  $\le 3$.  (Using our convention, here $e$ and $i$ designate the
  induced chain maps $\Z B\pi_* \to \Z B\Gamma_*$.)  Since $\xi_M$ is
  a cycle, we have
  \begin{equation}
    \begin{aligned}
      \phi(\xi_M) = i(j(\xi_M)) &=
      e(j(\xi_M))+\Phi\partial(j(\xi_M)) + \partial\Phi(j(\xi_M))
      \\
      &=e(j(\xi_M))+\partial \Phi (j(\xi_M)))
    \end{aligned}
    \label{equation:computing-bounding-chain}
  \end{equation}
  in~$\Z B\Gamma_3$.  Note that the image of $e\colon \Z B\Gamma_i \to
  \Z B\Gamma_i$ lies in $D_i(B\Gamma)$ for $i>0$.  By applying the
  projection $p\colon \Z B\Gamma_* \to C_*(B\Gamma)$ to
  \eqref{equation:computing-bounding-chain}, it follows that the
  $4$-chain $u := p\Phi(j(\xi_M))$ satisfies $\phi_{\#}(\zeta_M)
  = \partial u$ in the cellular chain complex $C_*(B\Gamma)$.  Here we
  use that $p\phi=\phi_{\#}p$ for a morphism $\phi$ of simplicial
  sets.

  Theorem~\ref{theorem:controlled-chain-homotopy-BDH} also tells us
  that $d_{\Phi}(3) \le \dBDH(3) = 186$.  We have $d_j(k)=d_p(k)=1$
  since $j$ is (induced by) a simplicial map and $p$ is a projection
  sending a basis element to a basis element or zero.  From this it
  follows that
  \[
  d(u) = d(p(\Phi(j(\xi_M)))) \le d_p(3) \cdot d_{\Phi}(3) \cdot
  d_j(3) \cdot d(\xi_M) = 186\cdot d(\zeta_M).
  \]

  Now we apply Theorem~\ref{theorem:existence-of-efficient-bordism} to
  $(M$, $\phi$, $u)$.  This gives us a smooth bordism $W$ over
  $\Gamma$ between $M$ and another $3$-manifold $N$ which is trivially
  over $B\Gamma$, where 
  \[
  (\text{$2$-handle complexity of $W$}) \le 195\cdot d(\zeta_M) +
  975\cdot d(u) \le 181545\cdot d(\zeta_M).  \qedhere
  \]
\end{proof}

\begin{proof}[Proof of Theorem~\ref{theorem:linear-universal-bound}]
  Suppose $M$ is a closed $3$-manifold with simplicial complexity~$n$,
  and $\phi\colon \pi_1(M)\to G$ is a homomorphism.  By
  Theorem~\ref{theorem:efficient-bordism-over-A^3}, there is a smooth
  bordism $W$ with $\partial W = M\sqcup -N$ over $\bA^3(\pi_1(M))$,
  where $N$ is trivially over~$\bA^3(\pi_1(M))$ and the $2$-handle
  complexity of $W$ is at most $181545\cdot n$.  Let
  $\Gamma:=\bA^3(G)$.  Similarly to the proof of
  Theorem~\ref{theorem:existence-of-universal-bound-intro}, we
  consider the following commutative diagram:
  \[
  \xymatrix@M=1.1ex{
    \pi_1(M) \ar[dd] \ar[rr]^{\phi}
    \ar@{^{(}->}[rd]^{i^3_{\pi_1(M)}} &
    &
    G \ar@{^{(}->}[rd]^{i^3_G}
    \\
    &
    \bA^3(\pi_1(M)) \ar[rr]^{\bA(\phi)} &
    &
    \bA^3(G) = \Gamma.
    \\
    \pi_1(W) \ar[ru]
  }
  \]
  By $L^2$-induction and Remark~\ref{remark:rho-from-bordism}, we can
  compute the $\rho$-invariant as the $L^2$-signature defect of $W$ as
  follows:
  \[
  \rhot(M,\phi) = \rhot(M, i^3_G\circ \phi) = \lsign_{\Gamma}W - \sign
  W.
  \]
  Since both $|\lsign_{\Gamma}W|$ and $|\sign W|$ are not greater than
  the $2$-handle complexity of $W$, it follows that
  \[
  |\rhot(M,\phi)| \le 2\cdot 181545\cdot n = 363090\cdot n. \qedhere
  \]
\end{proof}

\subsection{Bounds from Heegaard splittings and surgery presentations }
\label{subsection:bound-from-surgery}

In this subsection, we first prove
Theorem~\ref{theorem:universal-bound-for-heegaard-splitting} which
says the following: \emph{if $M$ is a closed $3$-manifold with
  Heegaard-Lickorish complexity~$\ell$, then for any~$\phi$,
  \[
  |\rhot(M,\phi)| \le 251258280 \cdot \ell.
  \]
}%

Our proof relies on Theorem~\ref{theorem:linear-universal-bound} and a
result from~\cite{Cha:2015-1}:

\begin{theorem}[{\cite[Theorem~A]{Cha:2015-1}}]
  \label{theorem:simplicial-and-HL-complexity}
  Suppose $M$ is a closed 3-manifold with simplicial complexity $n$
  and Heegaard-Lickorish complexity~$\ell$.  If $M\ne S^3$, then $n\le
  692\ell$.
\end{theorem}

\begin{proof}[Proof of
  Theorem~\ref{theorem:universal-bound-for-heegaard-splitting}]
  If $M=S^3$, then since $M$ is simply connected, $\rhot(S^3,\phi)=0$
  for any~$\phi$.  It follows that the conclusion holds in this case.
  Suppose $M\ne S^3$ has Heegaard-Lickorish complexity~$\ell$.  Then
  by Theorems~\ref{theorem:simplicial-and-HL-complexity}
  and~\ref{theorem:linear-universal-bound}, it follows that
  \[
  |\rhot(M,\phi)| \le 363090\cdot 692\cdot \ell = 251258280 \cdot \ell
  \]
  for any~$\phi$.
\end{proof}

In the rest of this subsection, we prove
Theorem~\ref{theorem:universal-bound-for-surgery-presentation}.
Recall that $c(L)$ denotes the crossing number of a link~$L$.  Also
recall that for a framed link $L$, we define $f(L)=\sum_i |n_i|$ where
$n_i\in \Z$ is the framing on the $i$th component of~$L$.
Theorem~\ref{theorem:universal-bound-for-surgery-presentation} says
the following: \emph{suppose $M$ is a $3$-manifold obtained by surgery
  along a framed link $L$ in~$S^3$.  Then for any~$\phi$,
  \[
  |\rhot(M,\phi)| \le 69713280 \cdot c(L) + 34856640\cdot f(L).
  \]
}%

For the proof of
Theorem~\ref{theorem:universal-bound-for-surgery-presentation}, we
need the following result proven in~\cite{Cha:2015-1}.

\begin{theorem}[{\cite[Theorem~B and Definition~1.3]{Cha:2015-1}}]
  \label{theorem:simplicial-complexity-bound-from-surgery}
  Suppose $M\ne S^3$ is a 3-manifold obtained by surgery along a
  framed link $L$ in $S^3$ which has no split unknotted zero framed
  component.  Then the simplicial complexity of $M$ is not greater
  than $192\cdot c(L)+ 96\cdot f(L)$.
\end{theorem}

\begin{proof}[Proof of
  Theorem~\ref{theorem:universal-bound-for-surgery-presentation}]

  If $M$ is $S^3$, then $\rhot(M,\phi)=0$ for any~$\phi$.  Therefore
  we may assume that $M\ne S^3$.

  Suppose $L$ is a framed link in $S^3$ that gives $M$ by surgery.  We
  claim that we may assume that $L$ does not have any split unknotted
  zero framed component.  To show the claim, suppose $L$ has $k$ split
  unknotted zero framed components, and let $L'$ be the sublink
  consisting of the other components.  Let $M$ and $M'$ be the
  $3$-manifolds obtained by surgery on $L$ and $L'$, respectively.
  Then $M$ is the connected sum of $M'$ and $k$ copies of $S^1\times
  S^2$.  Since $S^1\times S^2=\partial(S^1\times D^3)$ over
  $\pi_1(S^1\times S^2)=\Z$ and $S^1\times D^3$ has no $2$-handles,
  $\rhot(S^1\times S^2, \psi)=0$ for any~$\psi$.  Since $\rhot$ is
  additive under connected sum, we have
  $\rhot(M,\phi)=\rhot(M',\phi')$ where $\phi'\colon \pi_1(M')\to G$
  is the homomorphism induced by $\phi\colon\pi_1(M) \to G$.  Since
  $c(L)=c(L')$, $f(L)=f(L')$, and since we are interested in a
  universal bound, it follows that we may assume $L=L'$ as claimed.

  By the claim and by
  Theorem~\ref{theorem:simplicial-complexity-bound-from-surgery}, the
  simplicial complexity of $M$ is at most $192\cdot c(L)+96\cdot
  f(L)$.  By Theorem~\ref{theorem:linear-universal-bound}, it follows
  that
  \[
  |\rhot(M,\phi)| \le 363090(192\cdot c(L)+96\cdot f(L)) = 69713280
  \cdot c(L) + 34856640\cdot f(L)
  \]
  for any homomorphism $\phi\colon \pi_1(M) \to G$.
\end{proof}

The following theorem gives a similar but better estimate for a
special case:

\begin{theorem}
  \label{theorem:universal-bound-for-blackboard-framing-surgery}
  Suppose $D$ is a planar diagram of a link $L$ with $c$ crossings, in
  which each component is involved in a crossing.  Let $M$ be the
  $3$-manifold obtained by surgery on $L$ along the blackboard framing
  of~$D$. Then
  \[
  |\rhot(M,\phi)|\le 34856640\cdot c
  \]
  for any homomorphism $\phi\colon \pi_1(M) \to G$.
\end{theorem}

\begin{proof}
  We proceed similarly to the proof of
  Theorem~\ref{theorem:universal-bound-for-surgery-presentation};
  instead of
  Theorem~\ref{theorem:simplicial-complexity-bound-from-surgery}, we
  apply~\cite[Lemma~2.1]{Cha:2015-1} to our case, to obtain that the
  simplicial complexity of $M$ is at most~$96c$.  The conclusion
  follows from Theorem~\ref{theorem:linear-universal-bound}.
\end{proof}

\begin{example}
  \label{example:universal-bound-for-6_1}
  Consider the stevedore knot, which is $6_1$ in the table in
  Rolfsen~\cite{Rolfsen:1976-1}, or
  KnotInfo~\cite{cha-livingston:knotinfo}.  It is the simplest
  nontrivial ribbon knot.  Since it has an 8-crossing diagram with
  writhe zero, it follows that the zero surgery manifold $M$ of $6_1$
  satisfies $|\rhot(M,\phi)| \le 34856640\cdot 8 = 278853120$ for
  any~$\phi$, by
  Theorem~\ref{theorem:universal-bound-for-blackboard-framing-surgery}.
\end{example}

\begin{remark}
  \label{remark:applications-to-various-explicit-examples}
  In light of
  Theorem~\ref{theorem:universal-bound-for-surgery-presentation} and
  Theorem~\ref{theorem:universal-bound-for-blackboard-framing-surgery},
  now the proofs of the following existence results of various authors
  can give us explicit examples of
  \begin{enumerate}  
  \item knots of infinite order in the graded quotient of the
    Cochran-Orr-Teichner $n$-solvable filtration, and similarly for the
    grope filtration \cite[Theorems~1.4 and
    4.2]{Cochran-Teichner:2003-1}, \cite[Theorems~9.1 and~9.5 and
    Corollary~9.7]{Cochran-Harvey-Leidy:2009-1}; 
  \item slice knots which are algebraically doubly slice but
    nontrivial in the graded quotient of the double $n$-solvable
    filtration (and consequently not doubly slice)
    \cite[Theorem~1.1]{Kim:2006-1};
  \item knots whose iterated Bing doubles are $n$-solvable but not
    $(n+1)$-solvable (and consequently not slice)
    \cite[Corollaries~5.2 and~5.3 and
    Theorem~5.16]{Cochran-Harvey-Leidy:2008-1};
  \item $2$-torsion knots generating $(\Z_2)^\infty$ in the graded
    quotients of the $n$-solvable filtration~\cite[Theorems~5.5
    and~5.7 and Corollary~5.6]{Cochran-Harvey-Leidy:2009-3};
  \item non-concordant knots obtained from the same knots by infection
    using distinct curves \cite[Theorem~3.1 and Corollaries~3.2
    and~3.3]{Franklin:2013-1};
  \item knots which generate $\Z^\infty$ in the graded quotients of
    the $n$-solvable filtration and have vanishing Cochran-Orr-Teichner
    PTFA signature obstructions \cite[Theorems~1.4 and~4.11]{Cha:2010-1};
  \item links which are height $n$ grope concordant to but not height
    $n.5$ Whitney tower concordant to the Hopf link
    \cite[Theorem~4.1]{Cha:2012-1};
  \item non-concordant $m$-component links with the same arbitrarily
    given multivariable Alexander polynomial $\Delta$, if $m>2$ or
    $\Delta\ne 1$ \cite[Theorems~A,~B, 3.1,
    and~4.1]{Cha-Friedl-Powell:2012-1};
  \item non-concordant links admitting a homology cobordism between
    their zero surgery manifolds in which the meridians are homotopic
    \cite[Theorems~1.1 and~1.2]{Cha-Powell:2013-1}.
\end{enumerate}
\end{remark}

\section{Complexity of $3$-manifolds}
\label{section:complexity-of-3-manifolds}

In this section, we present applications of our Cheeger-Gromov bounds
to the complexity of $3$-manifolds.  We will also show that the
Cheeger-Gromov bounds in
Theorems~\ref{theorem:linear-universal-bound},
\ref{theorem:universal-bound-for-heegaard-splitting},
and~\ref{theorem:universal-bound-for-surgery-presentation} and the
$2$-handle complexity of the $4$-dimensional bordism in
Theorem~\ref{theorem:existence-of-efficient-bordism} are
asymptotically optimal.

\subsection{Lower bounds of the complexity of lens spaces}
\label{subsection:lower-bounds-of-lens-space-complexity}

Recall that Theorem~\ref{theorem:asymptotic-complexity-of-L(n,1)} in
the introduction says the following: \emph{$c(L(n,1)) \in \Theta(n)$.
  In fact, for each $n>3$,}
\[
\frac{n-3}{627419520} \le c(L(n,1)) \le n-3.
\]
The upper bound in
Theorem~\ref{theorem:asymptotic-complexity-of-L(n,1)} is due to Jaco
and Rubinstein~\cite{Jaco-Rubinstein:2006-1}.  In this section we give
a proof of the lower bound.

For the proof, we need the value of the Cheeger-Gromov invariant
of~$L(n,1)$.  For the finite fundamental group case, the
Cheeger-Gromov invariants are determined by the Atiyah-Singer
$G$-signatures~\cite{Atiyah-Singer:1968-3}, or the
Atiyah-Patodi-Singer
$\rho$-invariants~\cite{Atiyah-Patodi-Singer:1975-2}.  In particular,
for lens spaces, the computation in~\cite{Atiyah-Patodi-Singer:1975-2}
can be reinterpreted as a computation of the Cheeger-Gromov invariant.
We state a special case as a lemma, for the use in this and next
subsections.

\begin{lemma}
  \label{lemma:rho-for-L(n,1)}
  $\displaystyle \rhot(L(n,1),\id_{\pi_1(L(n,1))}) = \frac{n}{3} +
  \frac{2}{3n} - 1$.
\end{lemma}

\begin{proof}
  Atiyah, Patodi, and Singer computed their $\rho$-invariant for
  general (including high dimensional) lens
  spaces~\cite[Proposition~2.12]{Atiyah-Patodi-Singer:1975-2}.  For
  $L(n,1)$ and the regular representation $\alpha$ of
  $\pi_1(L(n,1))=\Z_d$, their formula gives the following:
  \[
  \rho_{\alpha}(L(n,1))=\sum_{k=1}^{n-1}\cot^2\Big(\frac{\pi k}{n}\Big).
  \]
  By the cotangent formula for the Dedekind sum (e.g., see
  \cite{Rademacher-Grosswald:1972-1}), the above sum is equal to $4n
  \sum_{k=1}^{n-1} (\!(k/n)\!)^2$, where $(\!(k/n)\!)$ denotes the
  sawtooth function, whose value is $k/n-1/2$ in our case.  From this
  we obtain
  \[
  \rho_{\alpha}(L(n,1)) = \frac{n^2}3 + \frac{2}{3} - n.
  \]
  Since $\frac{1}{n}\dim_\C = \ldim_{\Z_n}$, we have
  $\rhot(L(n,1),\id_{\pi_1(L(n,1))}) = \frac 1n \rho_\alpha(L(n,1))$.
  From this Lemma~\ref{lemma:rho-for-L(n,1)} follows.
\end{proof}

\begin{proof}[Proof of the lower bound in
  Theorem~\ref{theorem:asymptotic-complexity-of-L(n,1)}]
  We may assume $n>0$, by reversing the orientation if $n<0$. By
  Lemma~\ref{lemma:rho-for-L(n,1)} and
  Corollary~\ref{corollary:lower-bound-of-complexity}, it follows that
  \[
  c(L(n,1)) \ge \frac{1}{627419520}\Big( n+\frac{2}{n}-3 \Big) \ge
  \frac{n-3}{627419520}. \qedhere
  \]
\end{proof}

As discussed below in detail, it turns out that for odd $n$, the lower
bound in Theorem~\ref{theorem:asymptotic-complexity-of-L(n,1)} can be
arbitrarily larger than lower bounds from previously known methods.
Recall that for two functions $f(n)$ and $g(n)$, we say $g(n)$ is
\emph{dominated by} $f(n)$ and write $g(n)\in o(f(n))$ if
$\limsup_{n\to\infty} |g(n)/f(n)| = 0$.

\begin{enumerate}
\item Since $L(n,1)$ is a Seifert fibered space, the lower bound from
  the hyperbolic volume~\cite{Matveev-Petronio-Vesnin:2009-1} does not
  apply to~$L(n,1)$.
\item When $n$ is odd, since $H_1(L(n,1);\Z_2)=0$, the methods of
  Jaco-Rubinstein-Tillman \cite{Jaco-Rubinstein-Tillman:2009-1,
    Jaco-Rubinstein-Tillman:2011-1, Jaco-Rubinstein-Tillman:2013-1}
  using double covers and the $\Z_2$-Thurston norm do not give any
  nonzero lower bound.
\item In \cite{Matveev-Pervova:2001-1}, Matveev and Pervova proved the
  following:
  \[
  c(M) \ge 2\log_5|tH_1(M)| + \rank_{\Z} H_1(M),
  \] 
  where $|tH_1(M)|$ denotes the order of the torsion subgroup of
  $H_1(M)$.  For $M=L(n,1)$, this gives us $c(L(n,1))\ge 2\log_5 n$.
  This bound is logarithmic, which is dominated by the linear lower
  bound in Theorem~\ref{theorem:asymptotic-complexity-of-L(n,1)}.
\item In \cite{Matveev-Pervova:2001-1}, they showed that $c(M)\ge
  c(\pi_1(M))$, where the complexity $c(G)$ of a group $G$ is defined
  to be the minimal lengths of a finite presentation of~$G$.  The
  length of a finite presentation is the sum of the word length of the
  defining relators.  Computation of $c(G)$ is difficult in general;
  even for $G=\Z_n$, the answer seems complicated.  From the
  presentation $\langle g \mid g^n \rangle$, we obtain $c(\Z_n)\le n$.
  Interestingly, for infinitely many $n$, $c(\Z_n)$ is much smaller
  than~$n$.  For instance, let $n=k^2-1$.  Then $\Z_n$ admits a
  presentation $\langle x, y \mid x^ky^{-1},\, x^{-1}y^k\rangle$.
  Since its length is $2(k+1)$, we have $c(\Z_n) \le 2(k+1) =
  2(\sqrt{n+1}+1)$.  It follows that, for $M=L(n,1)$ with $n=k^2-1$,
  the lower bound $c(\pi_1(M))$ gives us at best $c(L(n,1)) \ge
  2(\sqrt{n+1}+1)$.  This is dominated by the linear lower bound in
  Theorem~\ref{theorem:asymptotic-complexity-of-L(n,1)}.
\end{enumerate}

From the above observations,
Theorem~\ref{theorem:rho-lower-bound-is-better} in the introduction
follows immediately.

\begin{remark}
  \leavevmode\Nopagebreak
  \begin{enumerate}
  \item In \cite{Cha:2015-2}, we show that there are closed hyperbolic
    3-manifolds (with fixed first homology) for which the complexity
    lower bounds obtained from Cheeger-Gromov invariants can be
    arbitrarily larger than the lower bound from the hyperbolic
    volume.
  \item There are closed 3-manifolds $M$ such that the invariant
    $\rhot(M,\phi)$, and consequently the lower bound of $c(M)$ given
    in Corollary~\ref{corollary:lower-bound-of-complexity}, can be
    arbitrarily larger than the Thurston norm of any generator of
    $H^1(M;\Z)$.  For instance, the computational method in
    \cite[Proposition~3.2]{Cochran-Orr-Teichner:2002-1} tells us how
    to construct a satellite knot with a fixed genus, say $g$, whose
    zero-surgery manifold $M$ admits an arbitrarily large value of
    $\rhot(M,\phi)$; the generator of $H^1(M)\cong\Z$ has Thurston
    norm $\le 2g-1$.
  \end{enumerate}
\end{remark}

\subsection{Linear Cheeger-Gromov bounds are  optimal }
\label{subsection:asymtotically-optimal}

By considering the case of lens spaces, we will prove
Theorem~\ref{theorem:linear-bound-is-optimal}, which says that the
linear Cheeger-Gromov bound in
Theorem~\ref{theorem:linear-universal-bound} is asymptotically
optimal.  Recall from the introduction that we define $\Bsc(n)$ to be
the optimal Cheeger-Gromov bound for $3$-manifolds with simplicial
complexity $n$, that is,
\[
\Bsc(n) = \sup\bigg\{ |\rhot(M,\phi)|\, \bigg| \,
\begin{tabular}{@{}c@{}}
  $M$ has simplicial complexity $\le n$ and\\
  $\phi$ is a homomorphism of $\pi_1(M)$
\end{tabular}
\bigg\}.
\]
Theorem~\ref{theorem:linear-bound-is-optimal} claims that
\[
\limsup_{n\to\infty} \frac{\Bsc(n)}{n} \ge \frac{1}{288},
\]
and consequently $\Bsc(n) \in \Omega(n)$.

\begin{proof}[Proof of Theorem~\ref{theorem:linear-bound-is-optimal}]
  Let $s_n$ be the simplicial complexity of $L(n,1)$.  By
  Lemma~\ref{lemma:rho-for-L(n,1)}, $\Bsc(s_n) \ge \frac{1}{3}n-1$.
  Also, since $L(n,1)$ is obtained by surgery along the $n$-framed
  unknot, we have $s_n\le 96n$ by
  Theorem~\ref{theorem:simplicial-complexity-bound-from-surgery}.  It
  follows that
  \begin{equation}
    \label{equation:growth-inequality}
    \frac{\Bsc(s_n)}{s_n} \ge \frac{1}{288}-\frac{1}{s_n}.
  \end{equation}
  Also, $s_n \ge c(L(n,1)) \ge (n-3)/627419520$ by
  Theorem~\ref{theorem:asymptotic-complexity-of-L(n,1)}.  So
  $s_n\to\infty$ as $n\to\infty$.  It follows that
  \eqref{equation:growth-inequality} holds for infinitely many values
  of~$s_n$.  Taking $\limsup$ of \eqref{equation:growth-inequality},
  the claimed inequality is obtained.
\end{proof}

We can also show that the Cheeger-Gromov bounds in
Theorem~\ref{theorem:universal-bound-for-heegaard-splitting}
and~\ref{theorem:universal-bound-for-surgery-presentation} are
asymptotically optimal.  To state it formally, we use the following
definitions.

\begin{definition}
  \label{definition:best-bounds-wrt-HL-surgery}
  Define
  \[
  \BHL(\ell) = \sup\bigg\{ |\rhot(M,\phi)|\, \bigg| \,
  \begin{tabular}{@{}c@{}}
    $M$ has Heegaard-Lickorish complexity $\le \ell$ and\\
    $\phi$ is a homomorphism of $\pi_1(M)$
  \end{tabular}
  \bigg\}.
  \]
  For a framed link $L$, let $n(L)$ be the number of split unknotted
  zero framed components of~$L$.  As in \cite{Cha:2015-1}, define the
  \emph{surgery complexity} of a closed $3$-manifold $M$ to be the
  minimum of $2c(L)+f(L)+n(L)$ over all framed links $L$ in $S^3$ from
  which $M$ is obtained by surgery.  Define
  \[
  \Bsurg(k) = \sup\bigg\{ |\rhot(M,\phi)|\, \bigg| \,
  \begin{tabular}{@{}c@{}}
    $M$ has surgery complexity $\le k$ and\\
    $\phi$ is a homomorphism of $\pi_1(M)$
  \end{tabular}
  \bigg\}.
  \]
\end{definition}

Theorems~\ref{theorem:universal-bound-for-heegaard-splitting}
and~\ref{theorem:universal-bound-for-surgery-presentation} tell us
that $\BHL(\ell)\in O(\ell)$ and $\Bsurg(k)\in O(k)$.

\begin{theorem}
  \label{theorem:heegaard-lickorish-dehn-surgery-bounds-are-optimal}
  $\BHL(\ell)\in \Omega(\ell)$ and $\Bsurg(k)\in \Omega(k)$.  In fact,
  \[
  \frac{1}{3} \le \limsup_{\ell\to\infty} \frac{\BHL(\ell)}{\ell} \le 251258280
  \]
  and
  \[
  \frac{1}{3} \le \limsup_{k\to\infty} \frac{\Bsurg(k)}{k} \le 34856640.
  \]
\end{theorem}

\begin{proof}
  The upper bounds are immediately obtained from
  Theorems~\ref{theorem:universal-bound-for-heegaard-splitting}
  and~\ref{theorem:universal-bound-for-surgery-presentation}.  The
  proofs of the lower bounds are identical with that of
  Theorem~\ref{theorem:linear-bound-is-optimal}; instead of the fact
  that the simplicial complexity of $L(n,1)$ is not greater than
  $96n$, we use that both the Heegaard-Lickorish complexity and the
  surgery complexity of $L(n,1)$ are not greater than~$n$.  This gives
  us the lower bound $\frac{1}{3}$ of the $\limsup$ instead of
  $\frac{1}{3\cdot 96}=\frac{1}{288}$.
\end{proof}

\subsection{Bordisms with linear 2-handle complexity are optimal}
\label{subsection:linear-2-handle-complexity-is-optimal}

Finally, we show that the $2$-handle complexity $195\cdot
d(\zeta_M)+975\cdot d(u)$ in
Theorem~\ref{theorem:existence-of-efficient-bordism} is asymptotically
best possible.  For the reader's convenience, we recall
Theorem~\ref{theorem:existence-of-efficient-bordism}: \emph{suppose
  $M$ is a closed $3$-manifold endowed with a triangulation of
  complexity~$d(\zeta_M)$.  Suppose $M$ is over $G$ via a
  simplicial-cellular map $\phi\colon M \to BG$.  If there is a
  $4$-chain $u\in C_4(BG)$ satisfying $\partial u =
  \phi_{\#}(\zeta_M)$, then there exists a smooth bordism $W$ between
  $M$ and a trivial end such that $2$-handle complexity of $W$ is at
  most $195\cdot d(\zeta_M) + 975\cdot d(u)$.}  Here $\zeta_M\in
C_3(M)$ is the sum of $3$-simplices which represents the fundamental
class of~$M$.

To state our result, we formally define ``the best possible $2$-handle
complexity'' as a function in $k:=d(\zeta_M)+d(u)$ as follows:

\begin{definition}
  \label{definition:2-handle-complexity-function}
  Let $\mathcal{M}(k)$ be the collection of pairs $(M,\phi)$ of a
  closed triangulated $3$-manifold $M$ and a simplicial-cellular map
  $\phi\colon M \to BG$ admitting a $4$-chain $u\in C_4(BG)$ such that
  $\partial u = \phi_{\#}(\zeta_M)$ and $k=d(\zeta_M)+d(u)$.  For a
  given $(M,\phi)$, let $\mathcal{B}(M,\phi)$ be the collection of
  bordisms $W$ over $G$ between $M$ and a trivial end.  Define
  \[
  B^{\mathrm{2h}}(k) := \sup_{(M,\phi)\in \mathcal{M}(k)}\;
  \mathop{\min\vphantom{p}}\limits_{W\in
    \mathcal{B}(M,\phi)}\,\{\text{$2$-handle complexity of $W$}\}.
  \]
\end{definition}

Briefly speaking, $B^{\mathrm{2h}}(k)$ is the optimal (smallest) value
for which the following holds: for any $(M,\phi)$ in $\mathcal{M}(k)$
there is a desired bordism $W$ with $2$-handle complexity not greater
that $B^{\mathrm{2h}}(k)$.

\begin{theorem}
  \label{theorem:linear-2-handle-complexity-is-best-possible}
  $B^{\mathrm{2h}}(k)\in O(k)\cap \Omega(k)$.  In fact,
  \[
  \frac{1}{107712} \le \limsup_{k\to\infty}
  \frac{B^{\textup{2h}}(k)}{k} \le 975.
  \]
\end{theorem}

\begin{proof}
  Theorem~\ref{theorem:existence-of-efficient-bordism} tells us that
  $975$ is an upper bound of $B^{\textup{2h}}(k)/k$.  Consequently
  $B^{\textrm{2h}}(k)\in O(k)$.

  To show the remaining conclusion, we consider the lens space
  $M=L(n,1)$ and $G=\bA^3(\Z_n)$.  By
  Theorem~\ref{theorem:simplicial-complexity-bound-from-surgery},
  there is a triangulation of $M$ of simplicial complexity at
  most~$96n$.  That is, $d(\zeta_M)\le 96n$.  Appealing to
  Theorem~\ref{theorem:simplicial-cellular-approximation}, choose a
  simplicial-cellular map $\phi\colon M\to B\bA^3(\Z_n)$ which induces
  the inclusion $\pi_1(M)=\Z_n \to \bA^3(\Z_n)$ defined in
  Definition~\ref{definition:mitosis}.  Similarly to the proof of
  Theorem~\ref{theorem:efficient-bordism-over-A^3}, there is a
  $4$-chain $u\in C_4(BG)$ such that $\partial u = \phi_{\#}(\zeta_M)$
  and $d(u) \le 186d(\zeta_M)$, by
  Theorem~\ref{theorem:controlled-chain-homotopy-BDH}. Let
  $k=d(\zeta_M)+d(u)$.  By definition, $(M,\phi)\in
  \mathcal{M}(k)$. Also note that
  \[
  k \le 187d(\zeta_M)\le 17952n.
  \]

  We claim that
  \[
  \min_{W\in \mathcal{B}(M,\phi)} \, \{\text{$2$-handle
    complexity of $W$}\} \ge \frac{k}{107712}
  - \frac{1}{2}.
  \]
  To show the claim, suppose $W$ is a bordism over $G$ between
  $M=L(n,1)$ and a trivial end.  Then we can compute $\rhot(M,\phi)$
  as the $L^2$-signature defect of $W$.  In particular, if $W$ has
  $2$-handle complexity $r$, then $|\rhot(M,\phi)| \le 2r$.  By the
  $L^2$-induction property and by Lemma~\ref{lemma:rho-for-L(n,1)}, we
  have
  \[
  \rhot(M,\phi) = \rhot(M,\id_{\pi_1(M)}) = \frac{n}{3} + \frac{2}{3n}
  -1.
  \]
  Combining these, we obtain
  \[
  r \ge \frac{n}{6} - \frac{1}{2} \ge \frac{k}{107712} - \frac{1}{2}
  \]
  as claimed.

  From the claim, it follows that
  \begin{equation}
    \label{equation:lower-bound-of-B2h(k)}
    B^{\textrm{2h}}(k) \ge \frac{k}{107712} - \frac{1}{2}.
  \end{equation}
  Obviously $k\ge d(\zeta_M)\ge c(L(n,1))$, and by
  Theorem~\ref{theorem:asymptotic-complexity-of-L(n,1)}, $c(L(n,1))
  \to \infty$ as $n\to \infty$.  It follows that
  \eqref{equation:lower-bound-of-B2h(k)} holds for infinitely
  many~$k$.  This completes the proof.
\end{proof}

\appendix
\setcounter{secnumdepth}{0}

\section{Appendix: simplicial sets and simplicial classifying spaces}
\label{section:simplicial-bar-construction}

\setcounter{secnumdepth}{3}
\setcounter{theorem}{0}
\setcounter{equation}{0}
\def\thesection{A}
\def\thesubsubsection{\S\arabic{subsubsection}}

In this appendix we give a quick review of basic definitions and facts
on simplicial sets, for readers not familiar with them, focusing on
those we needed in this paper, and present a detailed proof of
Theorem~\ref{theorem:simplicial-cellular-approximation} stated in the
body.  (See
Proposition~\ref{proposition:simplicial-realization}.) 
There are numerous excellent references on simplicial sets.  For
instance, \cite{May:1992-1}, \cite{Goerss-Jardine:1999-1} provide
thorough extensive treatements, and \cite{Friedman:2012-1} is an
easily accesible introduction for non-experts.

\subsubsection{Simplicial sets and geometric realizations}
\label{appendix:simplicial-sets}

We begin with a formal definition of a simplicial set.  A
\emph{simplicial set} $X$ is a collection $\{X_0,X_1,\ldots\}$ of sets
$X_n$ together with functions $d_i\colon X_n \to X_{n-1}$
($n=1,2,\ldots$, $i=0,\ldots,n$) and $s_i\colon X_n \to X_{n+1}$
($n=0,1,\ldots$, $i=0,\ldots,n$) satisfying the following:
\begin{equation}
\begin{alignedat}{4}
  d_i d_j &= d_{j-1} d_i &\quad&\text{if }i < j,\qquad\quad &   d_i s_j &=s_j d_{i-1} &\quad&\text{if }i > j+1,
  \\                                                                                     
  d_i s_j &= s_{j-1} d_i &&\text{if }i < j,       &   s_i s_j &= s_{j+1} s_i &&\text{if }i \le j,
  \\
  d_j s_j &= \rlap{$d_{j+1} s_j = \id$.}
\end{alignedat}
\label{equation:simplicial-relations}
\end{equation}
An element $\sigma\in X_n$ is called an \emph{$n$-simplex} of $X$, and
$d_i$ and $s_i$ are called the \emph{face map} and \emph{degeneracy
  map}.  A simplex $\sigma\in X_n$ is called \emph{degenerate} if
$\sigma=s_i\tau$ for some $i$ and $\tau\in X_{n-1}$.

A \emph{morphism} $f\colon X\to Y$ of simplicial sets is defined to be
a collection of maps $f\colon X_n \to Y_n$ satisfying $fd_i = d_if$
and $fs_i = s_if$.  Simplicial sets and their morphisms form a
category, which we denote by~$\sSet$.

The underlying geometric picture is as follows. Define the standard
$n$-simplex $\Delta^n$ to be the convex hull $[e_1,\ldots,e_n]$ of the
standard basis in~$\R^{n+1}$.  Then the face map $d_i$ is an
incarnation of taking the $i$th face
$[e_1,\ldots,\widehat{e_i},\ldots,e_n]$ of $\Delta^n$ by omitting the
$i$th vertex; similarly $s_i$ corresponds to producing a degenerate
$(n+1)$-simplex $[e_1,\ldots,e_i,e_i,\ldots,e_n]$ from $\Delta^n$ by
repeating the $i$th vertex.  It is straightforward to verify the above
relations of the $d_i$ and $s_i$ for the case of~$\Delta^n$. As the
key information of a simplicial set, the maps $d_i$ and $s_i$ indicate
how the simplices are assembled in the geometric picture: for an
$n$-simplex $\sigma$ and an $(n-1)$-simplex $\tau$, $d_i\sigma=\tau$
corresponds to an identification of $\tau$ with the $i$th face
of~$\sigma$, and similarly, $s_i\tau=\sigma$ corresponds to an
identification of $\sigma$ with $\tau$ via a collapsing.

The above geometric idea is formalized to the following definition of
the \emph{geometric realization} $|X|$ of a simplicial set~$X$.
Let $D_i\colon \Delta^n\to \Delta^{n+1}$ be the $i$th face inclusion,
i.e., the affine map determined by $(e_0,\ldots,e_n)\to
(e_0,\ldots,\widehat{e_{i}},\ldots,e_{n+1})$.  Let $S_i\colon
\Delta^{n+1}\to \Delta^n$ be the projection onto the $i$th face, i.e.,
the affine map determined by by $(e_0,\ldots,e_{n+1})\to
(e_0,\ldots,e_{i},e_i,\ldots,e_n)$.  Then
\[
|X| := \bigg(\coprod_{n\ge 0} X_n \times \Delta^n \bigg) \bigg/ \mathord{\sim}
\]
where the equivalence relation $\sim$ is generated by
$(\sigma,D_i(p))\sim(d_i(\sigma),p)$ for $\sigma\in X_{n+1}$ and $p\in
\Delta^n$, $(\sigma,S_i(p))\sim (s_i(\sigma),p)$ for $\sigma\in X_n$
and $p\in \Delta^{n+1}$.

Due to Milnor~\cite{Milnor:1957-3}, the space $|X|$ is a CW-complex
whose $n$-cells are in 1-1 correspondence to nondegenerate
$n$-simplices of~$X$; if $\sigma\in X_n$ is nondegenerate, the
characteristic map of the corresponding $n$-cell (which we call an
$n$-simplex of $|X|$) is given by
\[
\varphi_\sigma\colon \Delta^n = \{\sigma\}\times \Delta^n
\hookrightarrow \coprod_{n\ge 0} X_n \times \Delta^n \xrightarrow{q}
|X|.
\]
From this it follows that $|X|$ is a simplicial-cell complex in
the sense of Definition~\ref{definition:simplicial-cell-complex} in
the body of the paper.

A morphism $f\colon X\to Y$ of simplicial sets gives rise to a
continuous map $|f|\colon |X|\to |Y|$ induced by $\{\sigma\}\times
\Delta^n \xrightarrow{\id} \{f(\sigma)\} \times\Delta^n$:
\[
\vcenter{\hbox{$
\xymatrix@M=.5em{
  \{\sigma\}\times \Delta \ar@{^{(}->}[r] \ar[d]_-{\text{id}}
  &
  \coprod_{n\ge 0} X_n \times \Delta^n \ar[r]^-{q} \ar[d]_-{f\times\id}
  &
  |X| \ar[d]_{|f|}
  \\
  \{f(\sigma)\}\times \Delta \ar@{^{(}->}[r]
  &
  \coprod_{n\ge 0} Y_n \times \Delta^n  \ar[r]_-{q}
  &
  |Y|
}
$}}
.
\]
We remark that even when $\sigma\in X_n$ is nondegenerate,
$f(\sigma)\in Y_n$ may be degenerate: $q(\{f(\sigma)\}\times
\Delta^n)$ may be a $k$-simplex in $|Y|$ with $k<n$.

From the above diagram, it follows that $|f|$ is a simplicial-cellular
map in the sense of
Definition~\ref{definition:simplicial-cell-complex} in the body of the
paper.

\subsubsection{Chain complexes}
\label{appendix:chain-complexes}

A based chain complex $\Z X_*$ called the (unnormalized) \emph{Moore
  complex} is naturally associated to a simplicial set $X$, similarly
to the construction for an ordered simplicial complex: define $\Z X_n$
to be the free abelian group generated by~$X_n$, and define the
boundary map $\partial\colon \Z X_n\to \Z X_{n-1}$ by
$\partial(\sigma)=\sum_{i=0}^n (-1)^n d_i(\sigma)$ for an
$n$-simplex~$\sigma\in X_n$.  Then $(\Z X_*, \partial)$ becomes a
based chain complex with the $n$-simplices as basis elements.  This
gives rise to a functor $\sSet \to \Chpb$ to the category $\Chpb$ of
positive based chain complexes.

We remark that the chain complex $\Z X_*$ of a simplicial set is
distinct from the cellular chain complex $C_*(X):= C_*(|X|)$ of its
realization $|X|$, since degenerate simplices are still generators of
$\Z X_*$, while they do not give a cell of~$|X|$.

The chain complexes $\Z X_*$ and $C_*(X)$ are related as follows.  Let
$D_*(X)$ be the subgroup of $\Z X_*$ generated by degenerate simplices
of $X$, that is, simplices of the form $s_i\tau$ for some other
simplex~$\tau$.  It is known that $D_*(X)$ is a contractible
subcomplex and $C_*(X)\cong \Z X_* / D_*(X)$. 
Consequently we have a short exact sequence
\[
0 \to D_*(X) \to \Z X_* \xrightarrow{p} C_*(X) \to 0
\]
where the projection $p$ is a chain homotopy equivalence.  We remark
that the essential reason is that the $n$-cells of the CW complex
$|X|$ are in 1-1 correspondence with the nondegenerate $n$-simplices of
the simplicial set~$X$.  For a proof, see \cite[\S 22]{May:1992-1} or
\cite[p.~236]{MacLane:1995-1}.

We note that the projection $p\colon \Z X_* \to C_*(X)$ is a natural
transformation between the functors $\Z(-)_*$, $C_*(-)\colon \sSet \to
\Chpb$.  That is, if $\phi\colon X \to Y$ is a morphism of simplicial
sets, then $p\phi = \phi_{\#}p$.

We also note that if $X$ is an (ordered) simplicial complex which is
viewed as a simplicial set, then $C_*(X)$ can be viewed as a
subcomplex of $\Z X_*$; for, in this case, the $i$th face $d_i\sigma$
of a nondegenerate simplex $\sigma$ is nondegenerate, and consequently
the nondegenerate simplices generate a subcomplex of $\Z X_*$ which
can be identified with $C_*(X)$.  We remark that it does not hold for
an arbitrary simplicial set $X$; as an exercise, such an example can
be easily obtained using the simplicial classifying space $BG$
discussed in~\ref{appendix:simplicial-classifying-spaces}.

\subsubsection{Products}
\label{appendix:products}

One of the technical advantages of simplicial sets (in particular
allowing degenerate simplices) is that the product construction is
simple.  For two simplicial sets $X$ and $Y$, $X\times Y$ is defined
by $(X\times Y)_n := X_n\times Y_n$; together with $d_i(\sigma,\tau) =
(d_i\sigma, d_i\tau)$ and $s_i(\sigma,\tau)=(s_i\sigma, s_i\tau)$,
$X\times Y$ becomes a simplicial set.

\subsubsection{Simplicial classifying spaces}
\label{appendix:simplicial-classifying-spaces}

Let $G$ be a group.  The \emph{simplicial classifying space} $BG$ is
defined by the bar construction: $BG$ is the simplicial set with $BG_n
= \{[g_1,\ldots,g_n]\mid g_i\in G\}$ (in particular $BG_0=\{[\,]\}$
consists of one element) where the face map $d_i\colon BG_{n}\to
BG_{n-1}$ and the degeneracy map $s_i:BG_{n}\to BG_{n+1}$ are given by
\[
\begin{aligned}
  d_i[g_1,\ldots,g_n] &= 
  \begin{cases}
    [g_2,\ldots,g_n] &\text{if }i = 0,\\
    [g_1,\ldots, g_{i-1},g_i g_{i+1}, g_{i+2},\ldots,g_n] &\text{if }0<i<n,\\
    [g_1,\ldots,g_{n-1}] &\text{if }i = n,
  \end{cases}
  \\
  s_i[g_1,\ldots,g_n] & = [g_1,\ldots,g_i,e,g_{i+1},\ldots,g_n].
\end{aligned}
\]

From the definition, it is straightforward to verify that $B\colon \Gp
\to \sSet$ is a functor of the category of groups~$\Gp$.  It is well
known that the geometric realization $|BG|$ of $BG$ is an
Eilenberg-MacLane space $K(G,1)$.

In the following statement, $\pi_1(A)$ of a space $A$ is understood as
the free product of the fundamental groups of the path components.

\begin{proposition}
  \label{proposition:simplicial-realization}
  Suppose $X$ is a simplicial set and $\phi\colon \pi_1(|X|)\to G$ is
  a group homomorphism.  Then there is a morphism $f\colon X\to BG$ of
  simplicial sets such that $|f|_*\colon \pi_1(|X|)\to \pi_1(|BG|)=G$
  is equal to~$\phi$.
\end{proposition}

We remark that Theorem~\ref{theorem:simplicial-cellular-approximation}
in the body of the paper is an immediate consequence of
Proposition~\ref{proposition:simplicial-realization}.

\begin{proof}[Proof of Proposition~\ref{proposition:simplicial-realization}]
  We will define $f$ on $X_n$ inductively and check the functoriality
  $d_if = fd_i$ and $s_if = fs_i$ at each step.

  We start by defining $f$ on $X_0$ by $f(v)=[\,]\in BG_0$ for any $v\in
  X_0$.
  For each $0$-simplex $v$ of
  $X$, choose a path $\gamma_v$ to it from the basepoint of its
  component in~$|X|$.  (For example one may take a spanning forest of
  the $1$-skeleton to determine the $\gamma_v$.)  For $\sigma\in X_1$
  from $w:=d_1\sigma$ to $v:=d_0\sigma$, we define
  \[
  f(\sigma)=[\phi(\gamma_{w}^{\vphantom{1}} \cdot \psi_\sigma
  ^{\vphantom{1}}\cdot \gamma_v^{-1})]\in BG_1.
  \]
  We have that $f(d_i\sigma)= [\,] = d_if(\sigma)$ for $\sigma\in X_1$,
  and $f(s_i\tau) = s_i[\,] = s_if(\tau)$ for $\tau\in X_0$.  Also note that
  $f(\sigma)=[e]$ when $\sigma$ is a degenerate $1$-simplex (that is,
  $\sigma=s_i\sigma'$ for some $\sigma'\in X_0$).

  For notational convenience, for $\sigma=[g_1,\ldots,g_k]\in BG_k$ we
  often denote by $\sigma$ the sequence $g_1,\ldots,g_k$ obtained by
  removing the brackets.  In particular if $\sigma\in BG_k$ and
  $\tau\in BG_\ell$, then $[\sigma,\tau]$ denotes an element
  in~$BG_{k+\ell}$.

  For $\sigma\in X_2$, define
  \[
  f(\sigma)=[f(d_2\sigma), f(d_0\sigma)]\in BG_2.
  \]
  Note that we have $f(d_0\sigma)\cdot f(d_1\sigma)^{-1}\cdot
  f(d_2\sigma)=e$ in~$G$ since $\partial\sigma =
  d_0\sigma-d_1\sigma+d_2\sigma$.  Using this we check the
  functoriality: for $\sigma\in X_2$ and $\tau\in X_1$,
  \[
  \begin{aligned}
    d_0 f(\sigma) &= d_0 [f(d_2\sigma), f(d_0\sigma)] = [f(d_0\sigma)]
    = f(d_0\sigma),
    \\
    d_1 f(\sigma) &= d_1 [f(d_2\sigma), f(d_0\sigma)] =
    [f(d_2\sigma)f(d_0\sigma)] = [f(d_1\sigma)] = f(d_1\sigma),
    \\
    d_2 f(\sigma) &= d_2 [f(d_2\sigma), f(d_0\sigma)] = [f(d_2\sigma)]
    = f(d_2\sigma),
    \\
    f(s_0\tau) & = [f(d_2s_0\tau),f(d_0s_0\tau)] = [f(s_0d_1\tau),
    f(\tau)] = [e,f(\tau)]= s_0f(\tau),
    \\
    f(s_1\tau) & = [f(d_2s_1\tau),f(d_0s_1\tau)] = [f(\tau),
    f(s_0d_0\tau)] = [f(\tau),e]= s_1f(\tau).
  \end{aligned}
  \]

  In general, suppose $f$ has been defined on $X_k$ for $k<n$.  For
  $\sigma\in X_n$ we define $f$ by
  \begin{equation}
    f(\sigma) = [f(d_n\sigma),f(d_0^{n-1}\sigma)].
    \addtagsub{n}
    \label{equation:definition-f-on-X_n}
  \end{equation}
  We claim that
  \begin{equation}
    f(\sigma) = [f(d_2\cdots d_n\sigma),f(d_0\sigma)].
    \addtagsub{n}
    \label{equation:alt-definition-f-on-X_n}
  \end{equation}
  For, it obviously holds when $n=2$; for $n>2$, using
  \eqref{equation:definition-f-on-X_n}$_{n-1}$ and
  \eqref{equation:alt-definition-f-on-X_n}$_{n-1}$ as induction
  hypotheses, we obtain
  \begin{alignat*}{2}
    f(\sigma) &= [f(d_n\sigma),f(d_0^{n-1}\sigma)] &\quad&
    \text{by~\eqref{equation:definition-f-on-X_n}}_n\\
    & = [f(d_2\cdots d_{n-1}(d_n\sigma)), f(d_0 d_n\sigma),
    f(d_0^{n-1}\sigma)] && \text{by
      \eqref{equation:alt-definition-f-on-X_n}}_{n-1}\\
    & = [f(d_2\cdots d_{n-1}d_n\sigma), f(d_{n-1} d_0\sigma),
    f(d_0^{n-1}\sigma)] && \text{by \eqref{equation:simplicial-relations}}\\
    & = [f(d_2\cdots d_{n-1}d_n\sigma), f(d_0\sigma)] && \text{by
      \eqref{equation:definition-f-on-X_n}}_{n-1}.
  \end{alignat*}
  Now using \eqref{equation:simplicial-relations},
  \eqref{equation:definition-f-on-X_n}, and
  \eqref{equation:alt-definition-f-on-X_n} we verify the
  functoriality: for $\sigma\in X_n$, if $i<n-1$, we have
  \begin{multline*}
    d_i f(\sigma) = d_i[f(d_n\sigma),f(d_0^{n-1}\sigma)] =
    [d_if(d_n\sigma),f(d_0^{n-1}\sigma)]
    \\
    = [f(d_i d_n \sigma),f(d_0^{n-1}\sigma)] = [f(d_{n-1}d_{i}
    \sigma),f(d_0^{n-2}d_i\sigma)] = f(d_i\sigma),
  \end{multline*}
  and if $i > 1$, we have
  \begin{multline*}
    d_i f(\sigma) = d_i[f(d_2\cdots d_n\sigma), f(d_0\sigma)] =
    [f(d_2\cdots d_n\sigma), d_{i-1}f(d_0\sigma)]
    \\
    = [f(d_2\cdots d_n\sigma), f(d_{i-1}d_0\sigma)] = [f(d_2\cdots
    d_{n-1}d_i\sigma), f(d_0d_i\sigma)] = f(d_i\sigma).
  \end{multline*}
  So, in any case, we have $d_i f(\sigma) = f(d_i\sigma)$.  Also, for
  $\tau\in X_{n-1}$, if $i<n-1$, we have
  \begin{multline*}
    s_i f(\tau) = s_i[f(d_{n-1}\tau),f(d_0^{n-2}\tau)] =
    [s_if(d_{n-1}\tau),f(d_0^{n-2}\tau)] \\
    = [f(s_id_{n-1}\tau),f(d_0^{n-2}\tau)] =
    [f(d_{n}s_i\tau),f(d_0^{n-1}s_i\tau)] = f(s_i\tau),
  \end{multline*}
  and if $i>0$, we have
  \begin{multline*}
    s_i f(\tau) = s_i[f(d_2\cdots d_{n-1}\tau), f(d_0\tau)] = [f(d_2\cdots
    d_{n-1}\tau), s_{i-1}f(d_0\tau)]
    \\
    = [f(d_2\cdots d_{n-1}\tau), f(s_{i-1}d_0\tau)] = [f(d_2\cdots
    d_{n}s_i\tau), f(d_0s_i\tau)] = f(s_i\tau).
  \end{multline*}
  This completes the proof that $f\colon X\to BG$ is a well-defined
  morphism of simplicial sets.

  From the definition of $f$ on $X_1$, it follows that $f$ induces the
  given homomorphism $\phi\colon\pi_1(|X|)\to G$.
\end{proof}

\bibliographystyle{amsalpha}
\renewcommand{\MR}[1]{}
\bibliography{research}

\def\cprime{$'$}
\providecommand{\bysame}{\leavevmode\hbox to3em{\hrulefill}\thinspace}
\providecommand{\MR}{\relax\ifhmode\unskip\space\fi MR }
\providecommand{\MRhref}[2]{%
  \href{http://www.ams.org/mathscinet-getitem?mr=#1}{#2}
}
\providecommand{\href}[2]{#2}
\begin{thebibliography}{MPV09}

\bibitem[APS75]{Atiyah-Patodi-Singer:1975-2}
M.~F. Atiyah, V.~K. Patodi, and I.~M. Singer, \emph{Spectral asymmetry and
  {R}iemannian geometry. {II}}, Math. Proc. Cambridge Philos. Soc. \textbf{78}
  (1975), no.~3, 405--432. \MR{MR0397798 (53 \#1655b)}

\bibitem[AS68]{Atiyah-Singer:1968-3}
M.~F. Atiyah and I.~M. Singer, \emph{The index of elliptic operators. {III}},
  Ann. of Math. (2) \textbf{87} (1968), 546--604. \MR{0236952 (38 \#5245)}

\bibitem[BDH80]{Baumslag-Dyer-Heller:1980-1}
G.~Baumslag, E.~Dyer, and A.~Heller, \emph{The topology of discrete groups}, J.
  Pure Appl. Algebra \textbf{16} (1980), no.~1, 1--47. \MR{549702 (81i:55012)}

\bibitem[CFP14]{Cha-Friedl-Powell:2012-1}
Jae~Choon Cha, Stefan Friedl, and Mark Powell, \emph{Concordance of links with
  identical {A}lexander invariants}, Bull. Lond. Math. Soc. \textbf{46} (2014),
  no.~3, 629--642. \MR{3210718}

\bibitem[CG85]{Cheeger-Gromov:1985-1}
Jeff Cheeger and Mikhael Gromov, \emph{Bounds on the von {N}eumann dimension of
  {$L\sp 2$}-cohomology and the {G}auss-{B}onnet theorem for open manifolds},
  J. Differential Geom. \textbf{21} (1985), no.~1, 1--34. \MR{MR806699
  (87d:58136)}

\bibitem[Chaa]{Cha:2015-1}
Jae~Choon Cha, \emph{Complexities of 3-manifolds from triangulations, heegaard
  splittings, and surgery presentations}, Preprint.

\bibitem[Chab]{Cha:2015-2}
\bysame, \emph{Complexity of surgery manifolds and {C}heeger-{G}romov
  invariants}, Preprint.

\bibitem[Cha14a]{Cha:2010-1}
\bysame, \emph{Amenable {$L^2$}-theoretic methods and knot concordance}, Int.
  Math. Res. Not. IMRN (2014), no.~17, 4768--4803. \MR{3257550}

\bibitem[Cha14b]{Cha:2012-1}
\bysame, \emph{Symmetric {W}hitney tower cobordism for bordered 3-manifolds and
  links}, Trans. Amer. Math. Soc. \textbf{366} (2014), no.~6, 3241--3273.
  \MR{3180746}

\bibitem[CHL08]{Cochran-Harvey-Leidy:2008-1}
Tim~D. Cochran, Shelly Harvey, and Constance Leidy, \emph{Link concordance and
  generalized doubling operators}, Algebr. Geom. Topol. \textbf{8} (2008),
  no.~3, 1593--1646. \MR{MR2443256 (2009h:57014)}

\bibitem[CHL09]{Cochran-Harvey-Leidy:2009-1}
\bysame, \emph{Knot concordance and higher-order {B}lanchfield duality}, Geom.
  Topol. \textbf{13} (2009), no.~3, 1419--1482. \MR{MR2496049 (2009m:57006)}

\bibitem[CHL11]{Cochran-Harvey-Leidy:2009-3}
\bysame, \emph{2-torsion in the {$n$}-solvable filtration of the knot
  concordance group}, Proc. Lond. Math. Soc. (3) \textbf{102} (2011), no.~2,
  257--290. \MR{2769115 (2012c:57011)}

\bibitem[CL]{cha-livingston:knotinfo}
Jae~Choon Cha and Charles Livingston, \emph{{K}not{I}nfo: table of knot
  invariants}, \url{http://www.indiana.edu/~knotinfo}.

\bibitem[CO12]{Cha-Orr:2009-1}
Jae~Choon Cha and Kent~E. Orr, \emph{${L}^2$-signatures, homology localization,
  and amenable groups}, Comm. Pure Appl. Math. \textbf{65} (2012), 790--832.

\bibitem[COT03]{Cochran-Orr-Teichner:1999-1}
Tim~D. Cochran, Kent~E. Orr, and Peter Teichner, \emph{Knot concordance,
  {W}hitney towers and {$L\sp 2$}-signatures}, Ann. of Math. (2) \textbf{157}
  (2003), no.~2, 433--519. \MR{1 973 052}

\bibitem[COT04]{Cochran-Orr-Teichner:2002-1}
\bysame, \emph{Structure in the classical knot concordance group}, Comment.
  Math. Helv. \textbf{79} (2004), no.~1, 105--123. \MR{MR2031301 (2004k:57005)}

\bibitem[CP14]{Cha-Powell:2013-1}
Jae~Choon Cha and Mark Powell, \emph{Nonconcordant links with homology
  cobordant zero-framed surgery manifolds}, Pacific J. Math. \textbf{272}
  (2014), no.~1, 1--33. \MR{3270170}

\bibitem[CT07]{Cochran-Teichner:2003-1}
Tim~D. Cochran and Peter Teichner, \emph{Knot concordance and von {N}eumann
  {$\rho$}-invariants}, Duke Math. J. \textbf{137} (2007), no.~2, 337--379.
  \MR{MR2309149 (2008f:57005)}

\bibitem[CT08]{Constantio-Thurston:2008-1}
Francesco Costantino and Dylan Thurston, \emph{3-manifolds efficiently bound
  4-manifolds}, J. Topol. \textbf{1} (2008), no.~3, 703--745. \MR{2417451
  (2009g:57034)}

\bibitem[CW03]{Chang-Weinberger:2003-1}
Stanley Chang and Shmuel Weinberger, \emph{On invariants of {H}irzebruch and
  {C}heeger-{G}romov}, Geom. Topol. \textbf{7} (2003), 311--319 (electronic).
  \MR{MR1988288 (2004c:57052)}

\bibitem[EM53]{Eilenberg-MacLane:1953-1}
Samuel Eilenberg and Saunders MacLane, \emph{Acyclic models}, Amer. J. Math.
  \textbf{75} (1953), 189--199. \MR{0052766 (14,670b)}

\bibitem[FQ90]{Freedman-Quinn:1990-1}
Michael~H. Freedman and Frank Quinn, \emph{Topology of 4-manifolds}, Princeton
  Mathematical Series, vol.~39, Princeton University Press, Princeton, NJ,
  1990. \MR{MR1201584 (94b:57021)}

\bibitem[Fra13]{Franklin:2013-1}
Bridget~D. Franklin, \emph{The effect of infecting curves on knot concordance},
  Int. Math. Res. Not. IMRN (2013), no.~1, 184--217. \MR{3041699}

\bibitem[Fri12]{Friedman:2012-1}
Greg Friedman, \emph{Survey article: an elementary illustrated introduction to
  simplicial sets}, Rocky Mountain J. Math. \textbf{42} (2012), no.~2,
  353--423. \MR{2915498}

\bibitem[GJ99]{Goerss-Jardine:1999-1}
Paul~G. Goerss and John~F. Jardine, \emph{Simplicial homotopy theory}, Progress
  in Mathematics, vol. 174, Birkh\"auser Verlag, Basel, 1999. \MR{1711612
  (2001d:55012)}

\bibitem[Har08]{Harvey:2006-1}
Shelly Harvey, \emph{Homology cobordism invariants and the
  {C}ochran-{O}rr-{T}eichner filtration of the link concordance group}, Geom.
  Topol. \textbf{12} (2008), 387--430.

\bibitem[JR]{Jaco-Rubinstein:2006-1}
William Jaco and J.~Hyam Rubinstein, \emph{Layered-triangulations of
  3-manifolds}, arXiv:math/0603601.

\bibitem[JRT09]{Jaco-Rubinstein-Tillman:2009-1}
William Jaco, Hyam Rubinstein, and Stephan Tillmann, \emph{Minimal
  triangulations for an infinite family of lens spaces}, J. Topol. \textbf{2}
  (2009), no.~1, 157--180. \MR{2499441 (2010b:57016)}

\bibitem[JRT11]{Jaco-Rubinstein-Tillman:2011-1}
William Jaco, J.~Hyam Rubinstein, and Stephan Tillmann, \emph{Coverings and
  minimal triangulations of 3-manifolds}, Algebr. Geom. Topol. \textbf{11}
  (2011), no.~3, 1257--1265. \MR{2801418 (2012h:57043)}

\bibitem[JRT13]{Jaco-Rubinstein-Tillman:2013-1}
\bysame, \emph{{$\mathbb{Z}_2$}-{T}hurston norm and complexity of
  {$3$}-manifolds}, Math. Ann. \textbf{356} (2013), no.~1, 1--22. \MR{3038119}

\bibitem[Kim06]{Kim:2006-1}
Taehee Kim, \emph{New obstructions to doubly slicing knots}, Topology
  \textbf{45} (2006), no.~3, 543--566. \MR{MR2218756}

\bibitem[KS77]{Kirby-Siebenmann:1977-1}
Robion~C. Kirby and Laurence~C. Siebenmann, \emph{Foundational essays on
  topological manifolds, smoothings, and triangulations}, Princeton University
  Press, Princeton, N.J., 1977, With notes by John Milnor and Michael Atiyah,
  Annals of Mathematics Studies, No. 88. \MR{0645390 (58 \#31082)}

\bibitem[KT76]{Kan-Thurston:1976-1}
D.~M. Kan and W.~P. Thurston, \emph{Every connected space has the homology of a
  {$K(\pi ,1)$}}, Topology \textbf{15} (1976), no.~3, 253--258. \MR{0413089 (54
  \#1210)}

\bibitem[Lic62]{Lickorish:1962-1}
W.~B.~R. Lickorish, \emph{A representation of orientable combinatorial
  {$3$}-manifolds}, Ann. of Math. (2) \textbf{76} (1962), 531--540. \MR{0151948
  (27 \#1929)}

\bibitem[LS03]{Lueck-Schick:2003-1}
Wolfgang L{\"u}ck and Thomas Schick, \emph{Various {$L\sp 2$}-signatures and a
  topological {$L\sp 2$}-signature theorem}, High-dimensional manifold
  topology, World Sci. Publishing, River Edge, NJ, 2003, pp.~362--399.
  \MR{MR2048728 (2005b:58027)}

\bibitem[L{\"u}c98]{Lueck:1998-1}
Wolfgang L{\"u}ck, \emph{Dimension theory of arbitrary modules over finite von
  {N}eumann algebras and {$L\sp 2$}-{B}etti numbers. {I}. {F}oundations}, J.
  Reine Angew. Math. \textbf{495} (1998), 135--162. \MR{MR1603853 (99k:58176)}

\bibitem[L{\"u}c02]{Lueck:2002-1}
\bysame, \emph{{$L\sp 2$}-invariants: theory and applications to geometry and
  {$K$}-theory}, Ergebnisse der Mathematik und ihrer Grenzgebiete. 3. Folge. A
  Series of Modern Surveys in Mathematics [Results in Mathematics and Related
  Areas. 3rd Series. A Series of Modern Surveys in Mathematics], vol.~44,
  Springer-Verlag, Berlin, 2002. \MR{MR1926649 (2003m:58033)}

\bibitem[Mat90]{Matveev:1990-1}
S.~V. Matveev, \emph{Complexity theory of three-dimensional manifolds}, Acta
  Appl. Math. \textbf{19} (1990), no.~2, 101--130. \MR{1074221 (92e:57029)}

\bibitem[Mat03]{Matveev:2003-1}
\bysame, \emph{Complexity of three-dimensional manifolds: problems and results
  [translated from {\it {p}roceedings of the {c}onference ``{g}eometry and
  {a}pplications'' dedicated to the seventieth birthday of {v}. {a}.
  {t}oponogov ({r}ussian) ({n}ovosibirsk, 2000)}, 102--110, {R}oss. {A}kad.
  {N}auk {S}ib. {O}td. {I}nst. {M}at., {N}ovosibirsk, 2001]}, Siberian Adv.
  Math. \textbf{13} (2003), no.~3, 95--103. \MR{2028410}

\bibitem[May92]{May:1992-1}
J.~Peter May, \emph{Simplicial objects in algebraic topology}, Chicago Lectures
  in Mathematics, University of Chicago Press, Chicago, IL, 1992, Reprint of
  the 1967 original. \MR{1206474 (93m:55025)}

\bibitem[Mil57]{Milnor:1957-3}
John Milnor, \emph{The geometric realization of a semi-simplicial complex},
  Ann. of Math. (2) \textbf{65} (1957), 357--362. \MR{0084138 (18,815d)}

\bibitem[ML95]{MacLane:1995-1}
Saunders Mac~Lane, \emph{Homology}, Classics in Mathematics, Springer-Verlag,
  Berlin, 1995, Reprint of the 1975 edition. \MR{1344215 (96d:18001)}

\bibitem[MP01]{Matveev-Pervova:2001-1}
S.~V. Matveev and E.~L. Pervova, \emph{Lower bounds for the complexity of
  three-dimensional manifolds}, Dokl. Akad. Nauk \textbf{378} (2001), no.~2,
  151--152. \MR{1863880 (2002g:57043)}

\bibitem[MPV09]{Matveev-Petronio-Vesnin:2009-1}
Sergei Matveev, Carlo Petronio, and Andrei Vesnin, \emph{Two-sided asymptotic
  bounds for the complexity of some closed hyperbolic three-manifolds}, J.
  Aust. Math. Soc. \textbf{86} (2009), no.~2, 205--219. \MR{2507595
  (2010i:57034)}

\bibitem[Ram93]{Ramachandran:1993-1}
Mohan Ramachandran, \emph{von {N}eumann index theorems for manifolds with
  boundary}, J. Differential Geom. \textbf{38} (1993), no.~2, 315--349.
  \MR{1237487 (94j:58164)}

\bibitem[RG72]{Rademacher-Grosswald:1972-1}
Hans Rademacher and Emil Grosswald, \emph{Dedekind sums}, The Mathematical
  Association of America, Washington, D.C., 1972, The Carus Mathematical
  Monographs, No. 16. \MR{0357299 (50 \#9767)}

\bibitem[Rol76]{Rolfsen:1976-1}
Dale Rolfsen, \emph{Knots and links}, Publish or Perish Inc., Berkeley, Calif.,
  1976, Mathematics Lecture Series, No. 7. \MR{58 #24236}

\bibitem[Tho54]{Thom:1954-1}
Ren{\'e} Thom, \emph{Quelques propri\'et\'es globales des vari\'et\'es
  diff\'erentiables}, Comment. Math. Helv. \textbf{28} (1954), 17--86.
  \MR{0061823 (15,890a)}

\bibitem[Wei94]{Weibel:1994-1}
Charles~A. Weibel, \emph{An introduction to homological algebra}, Cambridge
  Studies in Advanced Mathematics, vol.~38, Cambridge University Press,
  Cambridge, 1994. \MR{1269324 (95f:18001)}

\end{thebibliography}

\end{document}